\RequirePackage{fix-cm}
\documentclass[smallextended]{svjour3}       
\smartqed  
\usepackage[utf8]{inputenc} 
\usepackage{url}            
\usepackage{booktabs}       
\usepackage{nicefrac}       
\usepackage{microtype}      
\usepackage{bm}
\usepackage{float}
\usepackage{relsize}
\usepackage{graphicx}
\usepackage{color}
\usepackage{multirow}
\usepackage[bookmarks]{hyperref}
\usepackage{lineno}
\usepackage{stmaryrd}
\usepackage{xcolor}
\usepackage{caption}
\usepackage[textsize=small]{todonotes}
\usepackage{hyperref}
\usepackage{comment}
\usepackage{ragged2e}
\usepackage{ulem}
\hypersetup{
	colorlinks=true, 
	linkcolor=blue, 
	urlcolor=red, 
	citecolor=magenta,
	linktoc=all 
}

\usepackage{lipsum}
\usepackage{amsfonts}
\usepackage{graphicx,epstopdf}
\usepackage{epstopdf}
\usepackage{algorithmic}
\usepackage{lscape}
\usepackage{mathtools}
\usepackage[caption=false]{subfig}
\usepackage{graphicx}

\usepackage{amsmath,amsfonts,amssymb}
\usepackage{mathtools}        
\usepackage{breqn}

\newcommand{\bc}{\textbf{C}}

\newcommand{\p}{\partial}
\newcommand{\bb}{\boldsymbol}
\newcommand{\bu}{\textbf{u}}

\newtheorem{thm}{Theorem}[section]

\newcommand{\jump}[1]{[\![#1]\!]}
\newcommand{\df}[2]{\frac{\partial #1}{\partial #2}} 
\newcommand{\bp}{\bb{P}}
\newcommand{\B}{\bb{B}}

\newcommand{\DP}{\Delta P}

\newcommand{\hb}{b}

\newcommand{\pll}{p_\parallel}
\newcommand{\per}{p_{\perp}}
\newcommand{\bhat}{\boldsymbol{b}}
\newcommand{\con}{\textbf{U}}
\newcommand{\W}{\textbf{W}}
\newcommand{\f}{\textbf{f}}
\newcommand{\F}{\textbf{F}}

\newcommand{\ent}{\mathcal{E}}
\newcommand{\entf}{\mathcal{Q}}
\newcommand{\evar}{\textbf{V}}
\newcommand{\s}{\textbf{S}}
\newcommand{\iph}{i+\frac{1}{2}}
\newcommand{\imh}{i-\frac{1}{2}}
\newcommand{\jph}{j+\frac{1}{2}}
\newcommand{\jmh}{j-\frac{1}{2}}
\newcommand{\Dx}{\Delta x}
\newcommand{\Dy}{\Delta y}

\newcommand{\ote}{\textbf{O2ES-EXP}}
\newcommand{\othe}{\textbf{O3ES-EXP}}
\newcommand{\ofe}{\textbf{O4ES-EXP}}
\newcommand{\oti}{\textbf{O2ES-IMEX}}
\newcommand{\othi}{\textbf{O3ES-IMEX}}
\newcommand{\ofi}{{\textbf{O4ES-IMEX}}}
\begin{document}
	
	\title{Entropy stable finite difference schemes for Chew, Goldberger \& Low anisotropic plasma flow equations
	}
	
	\titlerunning{Entropy stable finite difference schemes for CGL equations}        
	
	\author{Chetan Singh \and Anshu Yadav\footnote{Corresponding Author} \and Deepak Bhoriya \and Harish Kumar \and Dinshaw S. Balsara      
	}
	
	
	\institute{Chetan Singh \at
		Department of Mathematics, Indian Institute of Technology Delhi, India \\
		\email{maz218518@maths.iitd.ac.in}           
		\and
		Anshu Yadav \at
		Department of Mathematics, Indian Institute of Technology Delhi, India\\
		\email{maz178435@maths.iitd.ac.in}    
		\and
		Deepak Bhoriya \at
		Physics Department, University of Notre Dame, USA \\
		\email{dbhoriy2@nd.edu}
		\and
		Harish Kumar \at
		Department of Mathematics, Indian Institute of Technology Delhi, India\\
		\email{hkumar@iitd.ac.in}
		\and
		Dinshaw S. Balsara \at
		Physics Department, University of Notre Dame, USA\\
		ACMS, University of Notre Dame, USA\\
		\email{dbalsara@nd.edu}
	}
	
	\date{Received: date / Accepted: date}

	\maketitle
	
	\begin{abstract}
		In this article, we consider the Chew, Goldberger \& Low (CGL) plasma flow equations, which is a set of nonlinear, non-conservative hyperbolic PDEs modelling anisotropic plasma flows. These equations incorporate the double adiabatic approximation for the evolution of the pressure, making them very valuable for plasma physics, space physics and astrophysical applications. We first present the entropy analysis for the weak solutions. We then propose entropy-stable finite-difference schemes for the CGL equations. The key idea is to rewrite the CGL equations such that the non-conservative terms do not contribute to the entropy equations. The conservative part of the rewritten equations is very similar to the magnetohydrodynamics (MHD) equations. We then symmetrize the conservative part by following Godunov's symmetrization process for MHD. The resulting equations are then discretized by designing entropy conservative numerical flux and entropy diffusion operator based on the entropy scaled eigenvectors of the conservative part. We then prove the semi-discrete entropy stability of the schemes for CGL equations. The schemes are then tested using several test problems derived from the corresponding MHD test cases.
		\keywords{CGL Equations \and Symmetrization \and Entropy stability \and Higher-order finite difference schemes}
		\subclass{65M06 \and 65M12 \and 65M20 \and 76P99 \and 76W05}
	\end{abstract}
	
	\section{Introduction}
	Plasma flows play an important role in numerous applications in astrophysics, e.g., relativistic jets from active galactic nuclei \cite{landau1987relativistic}, earth's magnetosphere~\cite{heinemann1999role,dumin2002corotation,karimabadi2014link,burch2016magnetospheric}, the solar wind~\cite{hollweg1976collisionless,zouganelis2004transonic,chandran2011incorporating}. Often, to simulate plasma flows, equations of Magnetohydrodynamics (MHD) are used. It is one of the simplest descriptions of the plasma flows as it treats plasma as a single fluid having a single density, velocity and pressure. However, the derivation of MHD equations involves several assumptions, which may not be true in many cases. Hence, several efforts have been made to generalize the MHD equations (\cite{otto19903d,zhao_positivity-preserving_2014,hakim_high_2006,kumar_entropy_2012,abgrall_robust_2014}). 
	\par One of the key assumptions in the derivation of MHD equations is the assumption of fluid being {\em local thermodynamical equilibrium}, which allows the scalar description of the pressure (or temperature) (see \cite{hunana2019brief,goedbloed2004principles}). However, for space plasmas, this assumption is often not true \cite{meena2019robust,brown1995numerical,dubroca2004magnetic,hakim2008extended,johnson2014gaussian,sangam2008hllc,meng2012classical,Hirabayashi2016new,wang_exact_2020} and a tensorial description of the pressure is needed adequately describe non-local thermodynamical equilibrium effects. In \cite{berthon2006numerical,meena_robust_2017,meena_positivity-preserving_2017,meena_positivity-preserving_2017,meena_well-balanced_2017,biswas2021entropy,sen_entropy_2018,Hirabayashi2016new}, a three-dimensional tensor is used to describe the pressure. However, pressure tensor in plasmas is often rotated by magnetic field (see \cite{meena2019robust,hakim2008extended,sangam2008hllc}) and in this case, pressure tensor can be simplified by using tensorial components parallel to magnetic field and perpendicular to the magnetic field (\cite{huang2019six,hunana2019brief,meng2012classical}). Furthermore, if two-fluid effects are not important, then plasma flow can be described using  Chew, Goldberger \& Low (CGL) equations (\cite{chew1956boltzmann,gombosi1991transport,hunana2019brief,meng2012classical}).
	
	The CGL equations were first proposed in \cite{chew1956boltzmann} and are the simplest model to describe anisotropic plasma flows. These equations apply to dilute plasmas where the plasma may be approximated as collisionless. In the collisionless limit, charged particles gyrating around a magnetic field have to satisfy certain invariants. The CGL equations are built on the idea that there are two adiabatic invariants, one for the pressure parallel to the magnetic field and another for the pressure perpendicular to the magnetic field. In addition to the mass and momentum conservation equations, the fluid is described using two scalar pressure components (one parallel to the magnetic field and another perpendicular to the magnetic field). Furthermore, the evolution of the magnetic field is modelled via an induction equation. The CGL equations are a set of non-conservative hyperbolic systems of PDEs (\cite{meng2012classical,huang2019six}). Due to this, the design of the numerical schemes for CGL equations is challenging. The weak solution formulations for these equations require choosing an appropriate path which is often not known \cite{dal1995definition}. Furthermore, the numerical solutions are sensitive to the numerical viscosity \cite{abgrall2010comment}. In practice, a linear path is often considered to design approximated Riemann solvers (\cite{huang2019six,meng2012classical}).
	
	In this article, we consider entropy stable discretization of the equations following the ideas in \cite{chandrashekar2013kinetic,fjordholm2012arbitrarily,ismail2009affordable,tadmor2003entropy,yadav2023entropy,chandrashekar2016entropy}. We proceed as follows:
	\begin{enumerate}
		\item We first present the entropy analysis of the CGL system at the continuous level.
		\item Following the idea of \cite{yadav2023entropy}, we then reformulate the CGL equations such the non-conservative terms do not contribute to the entropy evolution. \item The conservative flux of the reformulated equations is very similar to the MHD flux. Furthermore, the entropy equation for the conservative part depends on the divergence of the magnetic field. Hence, we use Godunov's symmetrization for the conservative equations following \cite{chandrashekar2016entropy}.
		\item We then design entropy conservative numerical flux for the conservative part of equations. Following \cite{fjordholm2012arbitrarily}; we also design a higher-order numerical diffusion operator based on the entropy-scaled eigenvectors and sign-preserving reconstruction process.
		\item Finally, combining the entropy-stable discretization of the conservative part with a non-conservative part, we prove the semi-discrete entropy stability of the proposed scheme.
	\end{enumerate}
	The rest of the article is organized as follows: In the next Section, we describe the CGL equations and present an entropy analysis of the equations. In Section \ref{sec:reformulation},  we present the reformulated CGL equations and symmetrize the conservative part. In Section \ref{sec:semi_discrete}, we propose entropy-stable schemes and prove semi-discrete entropy stability for the CGL equations. Time stepping is discussed in Section \ref{sec:time_disc}, followed by numerical results in Section \ref{sec:num_res}. Final conclusions are presented in Section \ref{sec:conc}.
	
	\section{CGL equations for anisotropic plasma flows} 
	Following \cite{huang2019six,meng2012classical}, the CGL equations can be written as follows:
	\begin{subequations}
		\label{eq:cgl_con}
		\begin{align}
			&\frac{ \p\rho}{\p t}+\nabla \cdot(\rho \bu)=0,&\label{eq:cgl_con_1}\\
			&\frac{\p \rho \bu}{\p t}+\nabla\cdot\left[\rho \bu\bu+\per \textbf{I}+(\pll-\per)\bhat\bhat-\left(\B\B-\frac{|\B|^2}{2}\textbf{I}\right)\right]=0,\label{eq:cgl_con_2}&\\
			&\frac{\p \pll}{\p t}+\nabla\cdot(\pll\bu)=-2\pll\bhat\cdot\nabla \bu\cdot\bhat,&\label{eq:cgl_con_3}\\
			&\frac{\p e}{\p t}+\nabla\cdot\left[\bu\left(e+ \per+ \frac{|\B|^2}{2}\right) +\bu\cdot\left((\pll-\per)\bhat\bhat-\B\B\right)\right]=0,&\label{eq:cgl_con_4}\\
			&\frac{\p \B}{\p t}+\nabla \times [-(\bu \times \B)]=0.\label{eq:cgl_con_5}
		\end{align}
	\end{subequations}
	Here, $\rho, \bu=(u_x,u_y,u_z)^\top$, and $\B=(B_x, B_y, B_z)^\top$ are the density, velocity, and magnetic field, respectively.  The vector $\bhat$ is the unit vector in the direction of the magnetic field, defined as $\bhat=\B/|\B|$. One of the key shortcomings of the CGL system is that a physics-free study of the CGL equations seems to require magnetic field $\B$ to be non-zero. Otherwise, the conventional wisdom has been that the magnetic field direction is not defined, and hence the equations are not valid. However, it is useful to mention that a recent paper by~\cite{bhoriya2024} shows that a physics-based approach to the CGL equations resolves this and many other issues associated with simulating the CGL equations. However, we will assume that the magnetic field is non-zero. Integration of the mathematics-based approach here with the physics-based approach in \cite{bhoriya2024} will be the topic of a future paper. The pressure tensor $\bp$, rotated by the magnetic field is described using scalars $\pll$ and $\per$ via relation,
	\begin{align*}
		&\bp=\pll\bhat\bhat+\per(\textbf{I}-\bhat\bhat)+\bb{\Pi} \\
		&=\pll\begin{pmatrix}
			b_{x}b_{x} & b_xb_y & b_{x}b_{z}\\
			b_{y}b_{x} & b_{y}b_{y} & b_{y}b_{z}\\
			b_{z}b_{x} &  b_{z}b_{y} & b_{z}b_{z}
		\end{pmatrix}+\per\begin{pmatrix}
			1-b_{x}b_{x} & -b_xb_y & -b_{x}b_{z}\\
			-b_{y}b_{x} & 1-b_{y}b_{y} & -b_{y}b_{z}\\
			-b_{z}b_{x} &  -b_{z}b_{y} & 1-b_{z}b_{z}
		\end{pmatrix}+\begin{pmatrix}
			\Pi_{xx} & \Pi_{xy} & \Pi_{xz}\\
			\Pi_{yx} & \Pi_{yy} & \Pi_{yz}\\
			\Pi_{zx} &  \Pi_{zy} & \Pi_{zz}
		\end{pmatrix}
	\end{align*}
	where the scalars are given by $\pll=\bp:\bhat\bhat$ and $\per=\bp:\textbf{I}-\bhat\bhat$, denoting parallel and perpendicular components of the pressure tensor, respectively. The CGL model is obtained by neglecting the  Finite Larmor Radius pressure tensor, i.e., we assume $\bb{\Pi}=0$. The equation of state is given by,
	\begin{align}
		e=\frac{\rho |\bu|^2}{2}+\frac{|\B|^2}{2}+\frac{3 p}{2}
		\label{eq:EOS}
	\end{align}
	where, $p=\frac{2 \per+\pll}{3}$ is the average scalar pressure. Given the above notations, we define the vector of conservative variables as $\mathbf{U}=(\rho, \rho \bu, \pll, e, \B)^\top$. Equation \eqref{eq:cgl_con_1} represents the conservation of mass, followed by  Equation \eqref{eq:cgl_con_2}, representing momentum conservation. The parallel component of the pressure tensor is governed by \eqref{eq:cgl_con_3}. This is the only equation that contains non-conservative terms. The energy conservation is given by equation \eqref{eq:cgl_con_4}, followed by the induction equation \eqref{eq:cgl_con_5} for the magnetic field. In addition, we also have divergence-free conditions on the magnetic field given by,
	\begin{equation}
		\nabla \cdot \B=0
		\label{eq:cgl_divB}
	\end{equation}
	However, this condition is an outcome of the induction equations \eqref{eq:cgl_con_5} and divergence free initial condition.
	
	The eigenvalues of the CGL system~\cite{kato1966propagation,meng2012classical} in the $x$-direction are given by,
	\begin{equation}
		\bb{\Lambda}=(u_x,~ u_x, ~0, ~u_x\pm c_a,~u_x\pm c_f,~u_x\pm c_s)
		\label{eq:eigenvalues}
	\end{equation}
	where,
	\begin{align*}
		& c_a= \sqrt{{\frac{B_x^2}{\rho}-\frac{(\pll-p_\perp){b}_x^2}{\rho}}},&\\&
		c_f= \frac{1}{\sqrt{2\rho}}\big[|\B|^2+2p_\bot +b_x^2(2\pll-\per) + \{(|\B|^2+2\per +b_x^2(2\pll-\per))^2&\\&~~~~~~+4(\per^2b_x^2(1-b_x^2)-3\pll \per b_x^2(2-b_x^2)+3\pll^2 b_x^4-3B_x^2 \pll)\}^{\frac{1}{2}}\big]^{\frac{1}{2}},&\\&
		c_s=\frac{1}{\sqrt{2\rho}}\big[|\B|^2+2\per +b_x^2(2\pll-\per) - \{(|\B|^2+2\per +b_x^2(2\pll-\per))^2&\\&~~~~~~+4(\per^2b_x^2(1-b_x^2)-3\pll \per b_x^2(2-b_x^2)+3\pll^2 b_x^4-3B_x^2 \pll )\}^{\frac{1}{2}}\big]^{\frac{1}{2}}.
	\end{align*}
	
	It was shown in \cite{kato1966propagation} that eigenvalues of the system can become imaginary. Hence, some conditions on the admissible domain need to be imposed. Following \cite{kato1966propagation}, we define,
	$$
	p_m=\frac{p_{\bot}^2}{6p_{\bot}+3|\B|^2}, \text{     and    } p_M={|\B|^2}+p_{\bot}.
	$$
	Then, the admissible domain for the CGL system is given by,
	\begin{equation}
		\label{eq:cgl_domain_new}
		\Omega=\{\con\in \mathbb{R}^9 |~\rho>0,~\pll>0,\per>0,~p_{m} \leq \pll \leq p_{M} \}.
	\end{equation}
	In addition, it was observed in \cite{kato1966propagation}, that the slow wave is not necessarily slower than the Alfv\'en  wave, hence the domain $\Omega$ has to be divided as follows:
	
	\begin{enumerate}
		\item $p_{m}\le \pll \le\frac{p_{M}}{4},~~~~~~~~~~\text{if}~~~~~~~ c_s\le c_{a}\le c_f,$
		\item $\frac{p_{M}}{4}\le \pll \le\frac{p_{M}}{4}+\frac{3p_{m}}{4},~~\text{if}~~~~~~ c_s\le c_{a}< c_f,$
		\item $\frac{p_{M}}{4}+\frac{3p_{m}}{4}\le \pll \le p_{M},~~\text{if}~~~~~~ c_a\le c_{s}< c_f.$
	\end{enumerate}

	\subsection{Entropy analysis}
	\label{subsec:ent}
	Several hyperbolic systems consider pressure tensor as one of the variables. Following \cite{berthon2006numerical,biswas2021entropy,sen_entropy_2018,berthon2015entropy,yadav2023entropy}, we define the entropy $\ent$ and the entropy flux $\entf_x$ for the CGL system \eqref{eq:cgl_con}, as
	\begin{align}\label{entropy-pair}
		\ent= -\rho s, \qquad \entf_x=-\rho u_x s, \qquad
		s=\ln \left( \dfrac{\det\bp}{\rho^5} \right)=\ln \left( \dfrac{\pll \per^2}{\rho^5} \right).
	\end{align}
	
	\begin{lemma}\label{lemma:ent_eq}
		If $\con$ is a smooth solution of the CGL system \eqref{eq:cgl_con_1}-\eqref{eq:cgl_con_5}, then it satisfies,
		\begin{equation}
			\label{eq:ent_eql_1d}
			\p_t \ent+ \p_x \entf_x=0.
		\end{equation}
	\end{lemma}
	
	\begin{proof}
		By the chain rule,
		\begin{align*}
			\p_t (\pll\per^2)=\per^2 \p _t \pll+2\pll\per\p_t \per. 
		\end{align*}
		
		Using \eqref{eq:cgl_con_1}-\eqref{eq:cgl_con_5} to derive the evolution of $\per$, to get,	
		\begin{align}\label{eq:cgl_p}
			\p_t \per=&-\p_x(\per u_x)-\per\p_x u_x+\per b_x( b_x{\p_x u_x}+b_y{\p_x u_y}+b_z{\p_x u_z})-{(\bu\cdot\B)}{\p_x B_x}.
		\end{align}
		Combing with \eqref{eq:cgl_con_3}, we get, 
		\begin{eqnarray*}
			\p_t (\pll\per^2)		&=&\per^2 \p _t \pll+2\pll\per\p_t \per\\ \nonumber
			&=&-u_x \p_x (det \bp)-5 (det(\bp)) \p_x u_x - 2\pll\per(\bu\cdot\B)\p_x B_x,
		\end{eqnarray*}
		Using \eqref{eq:cgl_con_1}, we have,	
		\begin{align*}
			\p_t s+ u_x \p_x s+2\frac{(\bu\cdot\B)}{\per}\p_x B_x=0.
		\end{align*}
		which immediately results in
		\begin{equation}
			\label{eq:ent_eql_1d_withdiv}
			\p_t \ent+ \p_x \entf_x+\frac{2\rho(\bu\cdot\B)}{\per}\p_x B_x=0.
		\end{equation}
		This combining with $\nabla\cdot\B=0$, gives \eqref{eq:ent_eql_1d}.
	\end{proof}
	\begin{remark}
		In two dimensions, \eqref{eq:ent_eql_1d_withdiv} is replaced with,
		\begin{equation}
			\label{eq:ent_eql_2d_withdiv}
			\p_t \ent+ \p_x \entf_x+\p_y\entf_y+\frac{2\rho(\bu\cdot\B)}{\per}\nabla\cdot\B=0,
		\end{equation}
		with $\entf_y=-\rho u_y s$.	Furthermore, the entropy equality \eqref{eq:ent_eql_1d}, for the weak solutions results in,
		\begin{equation}
			\label{eq:ent_ineql_2d}
			\p_t \ent+ \p_x \entf_x+\p_y\entf_y\le0.
		\end{equation}
	\end{remark}
	We aim to design the entropy stable schemes so that \eqref{eq:ent_ineql_2d} holds at the semi-discrete level.

	\section{Reformulation of the CGL system}
	\label{sec:reformulation}
	The set of equations \eqref{eq:cgl_con} are not suitable for designing the entropy stable schemes. To do so, motivated by \cite{yadav2023entropy}, we will rewrite the equations.  We will consider the system~\eqref{eq:cgl_con} in two dimensions. We define entropy variable $\evar=\frac{\p \ent}{\p \con}$, which is given by,
	\begin{align}
		\evar&= \left( V_1,~V_2,~V_3,~V_4,~V_5,~V_6,~V_7,~V_8,~V_9 \right)^\top\nonumber \\
		&=\Bigg(5-s-\beta_\perp |\bu|^2,~2\beta_\perp\bu,~-\beta_\parallel+\beta_\perp,~-2\beta_\perp,~2\beta_\perp\B\Bigg)^\top,
		\label{eq:envar}
	\end{align} 
	where, $\beta_\perp=\frac{\rho}{\per},$ and $\beta_\parallel=\frac{\rho}{\pll}$.  Let us now recall the following definition from \cite{yadav2023entropy}:
	\begin{definition}
		A convex function $\ent(\con)$ is said to be an entropy function for the system
		\begin{equation*}
			\frac{\p \con}{\p t}+\frac{\p \f_{x}}{\p x} +\frac{\p \f_{y}}{\p y}+ \bc_{x}(\con)\frac{\p \con}{\p x} + \bc_{y}(\con)\frac{\p \con}{\p y}=0.
		\end{equation*}
		if there exist smooth functions $\entf_x(\con)$ and $\entf_y(\con)$ such that
		\begin{equation}\label{eq:ent_def}
			{\entf_x}'(\con) = \evar{\f_x}'(\con), \qquad {\entf_y}'(\con) = \evar{\f_y}'(\con)
		\end{equation}
		and
		\begin{equation}\label{eq:ent_noncons_prod_def}
			\evar^\top \bc_x(\con) = \evar^\top \bc_y(\con) = 0.
		\end{equation}
		The functions ($\ent, \entf_x, \entf_y$) form an entropy-entropy flux pair.
	\end{definition}
	The non-conservative terms in \eqref{eq:cgl_con} do not satisfy \eqref{eq:ent_noncons_prod_def}. Hence, we rewrite the equations \eqref{eq:cgl_con} in the form,
	\begin{align}\label{eq:cgl_ref_noncons}
		\frac{\p \con}{\p t}+\frac{\p \f_{x}}{\p x} + \frac{\p \f_{y}}{\p y} + \bc_{x}(\con)\frac{\p \con}{\p x} + \bc_{y}(\con)\frac{\p \con}{\p y}=0.
	\end{align}
	with flux functions given by,
	\begin{align*}
		\f_x=\begin{pmatrix*}
			\rho u_x\\
			\rho u_x^2 + \per-\left(B_x^2-\frac{|\B|^2}{2}\right)\\
			\rho u_x u_y -  B_x B_y\\
			\rho u_x u_z -  B_x B_z\\
			\pll u_x\\
			u_x\left(e+ \per+ \frac{|\B|^2}{2}\right) -{B_x}(\B\cdot\bu)\\
			0\\
			u_x B_y-u_y B_x\\
			u_x B_z-u_z B_x
		\end{pmatrix*},~
		\textbf{f}_y&=\begin{pmatrix*}
			\rho u_y\\
			\rho u_x u_y - B_x B_y\\
			\rho u_y^2 +\per-  \left(B_y^2-\frac{|\B|^2}{2}\right)\\
			\rho u_y u_z-  B_y B_z\\
			\pll u_y\\
			u_y\left(e+ \per+ \frac{|\B|^2}{2}\right) -B_y(\B\cdot\bu)\\
			u_y B_x - u_x B_y\\
			0\\
			u_y B_z - u_z B_y
		\end{pmatrix*}.
	\end{align*}
	We note that the above fluxes are very similar to the MHD fluxes. The rest of the terms are considered in the non-conservative form with matrices $\bc_x$ and $\bc_y$. 
	The non-conservative form with matrices $\bc_x$ and $\bc_y$ given by,
	\begin{landscape}
		Matrix $\bc_{x}(\con)$ is given below,
		\begin{align*}
			\begin{pmatrix}
				0&0&0&0&0&0&0&0&0\\
				-\frac{b_x^2 |\bu|^2}{2}& b_x^2 u_x& b_x^2 u_y& b_x^2 u_z&\frac{3}{2} b_x^2 & -b_x^2 & b_x^2 B_x+\frac{2\Gamma_x b_x}{|\B|} & b_x^2 B_y - \frac{2\DP b_{yxx}}{|\B|} & b_x^2 B_z-\frac{2\DP b_{zxx}}{|\B|}\\
				-\frac{b_x b_y |\bu|^2}{2} & b_x b_y u_x & b_x b_y u_y & b_x b_y u_z & \frac{3}{2}b_x b_y & -b_x b_y & b_x b_y B_x+\frac{\Gamma_x b_y - \DP b_{yxx}}{|\B|} & b_x b_y B_y+\frac{\Gamma_y b_x - \DP b_{xyy}}{|\B|} & b_x b_y B_z - \frac{2\DP b_{xyz}}{|\B|}\\
				-\frac{b_x b_z |\bu|^2}{2} & b_x b_z u_x & b_x b_z u_y & b_x b_z u_z & \frac{3}{2}b_x b_z & -b_x b_z & b_x b_z B_x+\frac{\Gamma_x b_z - \DP b_{zxx}}{|\B|} & b_x b_z B_y-\frac{2\DP b_{xyz}}{|\B|} & b_x b_z B_z + \frac{\Gamma_z b_x - \DP b_{xzz}}{|\B|}\\
				-\frac{2\pll b_x}{\rho}(\bhat\cdot\bu) & \frac{2\pll b_x}{\rho}b_x & \frac{2\pll b_x}{\rho}b_y & \frac{2\pll b_x}{\rho}b_z & 0 & 0 & 0 & 0 & 0\\
				\Upsilon_1^x & \Upsilon_2^x & \Upsilon_3^x & \Upsilon_4^x & \frac{3}{2} b_x(\bhat\cdot\bu) & -b_x(\bhat\cdot\bu) & b_x(\bhat\cdot\bu)B_x + \Theta^x_1 & b_x(\bhat\cdot\bu)B_y + \Theta^x_2 & b_x(\bhat\cdot\bu)B_z + \Theta^x_3\\
				0&0&0&0&0&0&0&0&0\\
				0&0&0&0&0&0&0&0&0\\
				0&0&0&0&0&0&0&0&0
			\end{pmatrix}
		\end{align*}
		and matrix $\bc_{y}(\con)$ is given below,
		\begin{align*}
			\begin{pmatrix}
				0&0&0&0&0&0&0&0&0\\
				-\frac{b_x b_y |\bu|^2}{2} & b_x b_y u_x & b_x b_y u_y & b_x b_y u_z & \frac{3}{2}b_x b_y & -b_x b_y & b_x b_y B_x+\frac{\Gamma_x b_y-\DP b_{yxx}}{|\B|} & b_x b_y B_y+\frac{\Gamma_y b_x-\DP b_{xyy}}{|\B|} & b_x b_y B_z -\frac{2\DP b_{xyz}}{|\B|}\\
				-\frac{b_y^2 |\bu|^2}{2} & b_y^2 u_x & b_y^2 u_y & b_y^2 u_z & \frac{3}{2} b_y^2 & -b_y^2 & b_y^2 B_x - \frac{2\DP b_{xyy}}{|\B|} & b_y^2 B_y + \frac{2\Gamma_y b_y}{|\B|} & b_y^2 B_z-\frac{2\DP b_{zyy}}{|\B|} \\
				-\frac{b_y b_z |\bu|^2}{2} & b_y b_z u_x & b_y b_z u_y & b_y b_z u_z & \frac{3}{2}b_y b_z & -b_y b_z & b_y b_z B_x - \frac{2\DP b_{xyz}}{|\B|} & b_y b_z B_y + \frac{\Gamma_y b_z - \DP b_{zyy}}{|\B|} & b_y b_z B_z + \frac{\Gamma_z b_y - \DP b_{yzz}}{|\B|}\\
				-\frac{2\pll b_y}{\rho}(\bhat\cdot\bu) & \frac{2\pll b_y}{\rho}b_x & \frac{2\pll b_y}{\rho}b_y & \frac{2\pll b_y}{\rho}b_z & 0 & 0 & 0 & 0 & 0\\
				\Upsilon_1^y & \Upsilon_2^y & \Upsilon^y_3 & \Upsilon_4^y & \frac{3}{2}b_y(\bhat\cdot\bu) & -b_y(\bhat\cdot\bu) & b_y(\bhat\cdot\bu)B_x + \Theta_1^y & b_y(\bhat\cdot\bu)B_y + \Theta_2^y & b_y(\bhat\cdot\bu)B_z + \Theta_3^y\\
				0&0&0&0&0&0&0&0&0\\
				0&0&0&0&0&0&0&0&0\\
				0&0&0&0&0&0&0&0&0
			\end{pmatrix}
		\end{align*}
	\end{landscape}
	where,\\
	$\DP = (\pll - \per),~\Gamma_i = \DP (1-b_i^2),~\forall i\in\left\{x,y,z\right\}$\\
	$b_{lmn}=b_l b_m b_n,~\forall~l,m,n\in\left\{x,y,z\right\}$\\
	$\Theta^x_1 = \{(\bhat\cdot\bu)+ b_x u_x\}\left(\frac{\Gamma_x}{|\B|}\right)- \DP\frac{b^2_x b_y u_y}{|\B|} - \DP\frac{b^2_x b_z u_z}{|\B|},$\\
	$\Theta^x_2 = \Gamma_y \left(\frac{b_x u_y}{|\B|}\right) - \DP\frac{b_x b_y b_z u_z}{|\B|} - \{(\bhat\cdot\bu)+ b_x u_x\}\frac{\DP b_x b_y}{|\B|},$\\
	$\Theta^x_3 = \Gamma_z \left(\frac{b_x u_z}{|\B|}\right) - \DP\frac{b_x b_y b_z u_y}{|\B|} - \{(\bhat\cdot\bu)+b_x u_x\}\frac{\DP b_x b_z}{|\B|}$,\\
	$\Theta_1^y = \Gamma_x \left(\frac{b_y u_x}{|\B|}\right) - \DP\frac{b_x b_y b_z u_z}{|\B|} - \{(\bhat\cdot\bu)+ b_y u_y\}\frac{\DP b_x b_y}{|\B|},$\\
	$\Theta_2^y = \{(\bhat\cdot\bu) + b_y u_y\}\left(\frac{\Gamma_y}{|\B|}\right)- \DP\frac{b^2_y b_x u_x}{|\B|} - \DP\frac{b^2_y b_z u_z}{|\B|},$\\
	$\Theta_3^y = \Gamma_z \left(\frac{b_y u_z}{|\B|}\right) - \DP\frac{b_x b_y b_z u_x}{|\B|} - \{(\bhat\cdot\bu)+b_y u_y\}\frac{\DP b_y b_z}{|\B|},$\\
	$\Upsilon_1^x = -b_x(\bhat\cdot\bu)\frac{|\bu|^2}{2}-\frac{\DP b_x (\bhat\cdot\bu)}{\rho},~\Upsilon_2^x = b_x(\bhat\cdot\bu)u_x + \frac{\DP b^2_x}{\rho},$\\
	$\Upsilon_3^x = b_x(\bhat\cdot\bu)u_y +\frac{\DP b_x b_y}{\rho},~\Upsilon_4^x = b_x(\bhat\cdot\bu)u_z+\frac{\DP b_x b_z}{\rho}$\\
	$\Upsilon_1^y = -b_y(\bhat\cdot\bu)\frac{|\bu|^2}{2} - \frac{\DP b_y(\bhat\cdot\bu)}{\rho},~\Upsilon_2^y = b_y(\bhat\cdot\bu)u_x+\frac{\DP b_x b_y}{\rho},$\\
	$\Upsilon_3^y = b_y(\bhat\cdot\bu)u_y+\frac{\DP b_y^2}{\rho}$ and $\Upsilon_4^y = b_y(\bhat\cdot\bu)u_z+\frac{\DP b_y b_z}{\rho}$.\\
	
	With this choice of the non-conservative matrices $\bc_x$ and $\bc_y$, we satisfy the equation \eqref{eq:ent_noncons_prod_def}. The complete proof is given in the Appendix \eqref{A.4}.

	\begin{remark}
		The existence of an entropy pair does not guarantee the symmetrizability of the system in the case of non-conservative hyperbolic systems. In particular, the CGL system is not symmetrizable. The proof is given in Appendix \eqref{A.5}
	\end{remark}
	%
	\subsection{Godunov's Symmetrization of the conservative part}
	We will now analyze the conservative part of the equation \eqref{eq:cgl_ref_noncons} for entropy and symmetrization, i.e. we consider,
	\begin{align}\label{eq:cgl_con_part}
		\frac{\p \con}{\p t}+\frac{\p \textbf{f}_x}{\p x}+\frac{\p \textbf{f}_y}{\p y}=0,   
	\end{align}
	
	\label{subsec:sym_cons_flux}
	\begin{definition}
		The conservation law \eqref{eq:cgl_con_part} is said to be symmetrizable if there exists a change of variable $\con \rightarrow \evar$ which symmetrizes it, i.e., \eqref{eq:cgl_con_part} with a change of variable can be written as,
		$$
		\frac{\p \con}{\p \evar}\frac{\p \evar}{\p {t}}+\frac{\p \textbf{f}_x}{\p \con}\frac{\p \con}{\p \evar}\frac{\p \evar}{\p {x}}+\frac{\p \textbf{f}_y}{\p \con}\frac{\p \con}{\p \evar}\frac{\p \evar}{\p {y}}=0.
		$$
		where $\frac{\p \con}{\p \evar}$ is a symmetric positive definite matrix.  Furthermore, $\frac{\p \textbf{f}_x}{\p \con}  \frac{\p \con}{\p \evar}$ and $\frac{\p \textbf{f}_y}{\p \con}  \frac{\p \con}{\p \evar}$ are symmetric matrices.
	\end{definition}
	The symmetrization of the hyperbolic conservation law and the existence of an entropy pair are closely related. Let us recall the following Theorem from \cite{godlewski2013numerical}.
	\begin{thm}
		A necessary and sufficient condition for the system to have a strictly convex entropy $\ent(\con)$ is that there exists a change of dependent variables $\con= \con(\evar)$ that symmetrizes system \eqref{eq:cgl_con_part}. 
	\end{thm} 
	We note that for the system \eqref{eq:cgl_con_part},
	\begin{align*}
		\frac{\p \textbf{f}_\alpha}{\p \evar}\neq\bigg(\frac{\p \textbf{f}_\alpha}{\p \evar}\bigg)^{\top},~~\alpha=x,y.
	\end{align*}
	Furthermore, the entropy-pair $(\ent,\entf_x,\entf_y)$  does not satisfy \eqref{eq:ent_def}. In fact, we have,
	\begin{align*}
		\frac{\p \entf_\alpha}{\p \con}=\evar^\top\cdot\frac{\p \textbf{f}_\alpha}{\p \con}+\frac{2\rho}{ \per}(\bu\cdot\B)B_\alpha'(\con).
	\end{align*}
	This is similar to the case of MHD equations \cite{chandrashekar2016entropy}. Hence, the conservative part of the CGL system is not symmetrizable. We follow Godunov's approach to attain the symmetrizability of the conservative part of the CGL system with divergence constraints.  We now follow \cite{chandrashekar2016entropy}, to symmetrize \eqref{eq:cgl_con_part}, which based on the earlier works in \cite{godunov,barth1999numerical,barth2006role}. We will consider the equations in $x$-direction only. The discussion can be extended to include $y$-directional flux, similarly.
	
	Consider the modified conservative system \eqref{eq:cgl_con_part} in $x$-direction, 
	\begin{equation}
		\frac{\p \con}{\p t}+\frac{\p \textbf{f}_x}{\p x}+\phi'(\evar)^\top \p_x B_x=0,\label{eq:cgl_modified}
	\end{equation} 
	where $\phi(\evar)$ is a homogeneous function of degree one, i.e.,
	\begin{equation}
		\evar \phi ' (\evar)=\phi(\evar). \label{eq:phi}
	\end{equation}
	The above modification is consistent because of the divergence-free constraint. If the system \eqref{eq:cgl_modified} is symmetrized by the transformation $\con \rightarrow \evar$, then it can be inferred that there are twice differentiable functions $\mathcal{U}(\evar)$ and $\mathcal{F}_x(\evar)$, where $\mathcal{U}(\evar)$ is a strictly convex function such that
	\begin{equation}\label{eq:congugate}
		\con=\mathcal{U}'(\evar)^\top, \quad \textbf{f}_x=\mathcal{F}_x'(\evar)^\top-\phi'(\evar)^\top B_x. 
	\end{equation}
	Combining \eqref{eq:cgl_modified} and \eqref{eq:congugate}, we get the symmetric form,
	$$ \mathcal{U}''(\evar) \frac{\p \evar}{\p t}+(\mathcal{F}_x''(\evar)-\phi''(\evar)B_x)\frac{\p \evar}{\p x} =0.$$
	Here, $ \mathcal{U}''(\evar)$  is a symmetric positive definite matrix and $(\mathcal{F}_x''(\evar)-\phi''(\evar)B_x)$ is a symmetric matrix. 
	Let us also define,
	\begin{equation}
		\ent(\con)=\evar(\con)\cdot\con-\mathcal{U}(\evar(\con))
		, \quad
		\entf_x(\con)=\evar(\con)\cdot\mathcal{F}_x'(\evar(\con))-\mathcal{F}_x(\evar(\con)) \label{eq:ent_ent_flux}
	\end{equation}	
	On differentiating the above equation and using \eqref{eq:congugate} we get
	\begin{equation}
		\label{eq:conj_entf}
		\ent'(\con)=\evar'(\con)\cdot\con+\evar(\con)-\mathcal{U}'(\evar)\evar'(\con) =
		\evar'(\con)\cdot\con+\evar(\con)-\con\cdot\evar'(\con)= \evar
	\end{equation}
	i.e. $$\evar=\ent'(\con)^\top.$$
	
	Hence, the change of variable that symmetrizes the equations is entropy variable $\evar (\con)$. Now differentiating $\entf_x$ from   
	\eqref{eq:conj_entf},
	\begin{align}
		\entf_x'(\con)&=\bigg(\evar(\con)\cdot\mathcal{F}_x'(\evar(\con))-\mathcal{F}_x(\evar(\con)) \bigg)' \nonumber\\
		&=\bigg(\evar(\con)\cdot(\textbf{f}_x(\evar)^\top+\phi'(\evar)^\top B_x)-\mathcal{F}_x(\evar(\con)) \bigg)' \nonumber\\
		&=\bigg(\evar(\con)\cdot\textbf{f}_x+\phi (\evar) B_x-\mathcal{F}_x(\evar(\con)) \bigg)' \quad\text{(from eq. \eqref{eq:phi})} \nonumber\\
		&=\bigg(\evar(\con)\cdot \textbf{f}_x'(\con)+\phi (\evar) B_x'(\con) \bigg)
		=\bigg(\ent'(\con)\cdot\textbf{f}_x'(\con)+\phi (\evar) B_x'(\con) \bigg). \label{eq:cgl_entf'}
	\end{align}
	Now, taking the dot product of $\evar$ with \eqref{eq:cgl_modified}, we get,
	\begin{align*}
		\evar\cdot\bigg(\frac{\p \con}{\p t}+\frac{\p \textbf{f}_x}{\p x}+\phi'(\evar) \nabla \cdot \boldsymbol{B}\bigg) 
		&=
		\frac{\p \ent}{\p t}+\bigg({\ent}'(\con)\cdot\frac{\p \textbf{f}_x}{\p \con}+\phi(\evar)   {B'_x(\con)} \bigg)\frac{\p \con}{\p x}\\
		&=\frac{\p \ent}{\p t}+\frac{\p \entf_x}{\p x}=0. \quad \text{ (using \eqref{eq:cgl_entf'})}
	\end{align*}
	Now using,\eqref{eq:cgl_entf'} with \eqref{eq:congugate}, we get,
	\begin{equation}
		\phi (\evar)=\frac{2\rho}{ \per}(\bu\cdot\B)=2\beta_\perp (\bu\cdot \B).
	\end{equation}
	We also note that the function $\phi(\evar)$ can be written using the components of $\evar$, which results in,
	$$\phi(\evar)=-\frac{V_2 V_7+V_3 V_8+V_4 V_9}{V_6}.$$ We note that this is a homogeneous function of degree one. Furthermore,
	$$\phi'(\evar)=[0, \B,0,\bu\cdot\B,\bu].$$
	Now, the entropy flux is given by,
	\begin{equation}
		\label{eq:ent_x_flux_mod}
		\entf_x=\evar \cdot \textbf{f}_x+\phi B_x- \mathcal{F}_x
	\end{equation}
	with 
	$$\mathcal{U}=\evar \cdot \con-\ent=2\rho+\beta_{\perp}|\B|^2,$$
	and
	$$\mathcal{F}_x=\evar \cdot \textbf{f}_x+\phi B_x -\entf_x=2\rho u_x+\beta_{\perp} u_x|\B|^2= \mathcal{U}u_x.$$
	Similarly, for the $y$-direction, we get,
	\begin{equation}
		\label{eq:ent_y_flux_mod}
		\entf_y=\evar \cdot \textbf{f}_y+\phi B_y- \mathcal{F}_y,
	\end{equation}
	with 
	$$\mathcal{F}_y=\evar \cdot \textbf{f}_y+\phi B_y -\entf_y=2\rho u_y+\beta_{\perp} u_y|\B|^2= \mathcal{U}u_y.$$
	We will now construct entropy stable numerical schemes for CGL equations by adding $\phi'(\evar)^\top\nabla\cdot \B$,   i.e. we will consider,
	\begin{align}\label{eq:cgl_ref_godunov}
		\frac{\p \con}{\p t}+\frac{\p \f_{x}}{\p x} + \frac{\p \f_{y}}{\p y} + \bc_{x}(\con)\frac{\p \con}{\p x} + \bc_{y}(\con)\frac{\p \con}{\p y}+\phi'(\evar)^\top ( \nabla\cdot \B)=0.
	\end{align}
	
	\section{Semi-discrete entropy stable numerical schemes}
	\label{sec:semi_discrete}
	In this Section, we will develop semi-discrete numerical schemes for the system \eqref{eq:cgl_ref_godunov}, which satisfy a discrete version of the entropy inequality~\eqref{eq:ent_ineql_2d}, with modified entropy fluxes \eqref{eq:ent_x_flux_mod} and \eqref{eq:ent_y_flux_mod}. 
	
	We consider a uniform mesh of size $\Delta x\times \Delta y$. The cell centers are represented as $(i,j)$ with cell interfaces for the cell $(i,j)$ is given by $(\imh,j)$ and $(\iph,j)$ in $x$-direction and $(i,\jmh)$ and $(i,\jph)$ in $y$-direction. Let us define the arithmetic averages across the interfaces for the quantity $a$ as
	$$
	\bar{a}_{\iph,j}= \frac{a_{i+1,j} + a_{i,j}}{2}\qquad \text{ and }\qquad  \bar{a}_{i,\jph}= \frac{a_{i,j+1} + a_{i,j}}{2}.
	$$
	Similarly, let us denote the jumps across the interfaces in $x$ and $y$-directions as,
	$$
	[\![a]\!]_{\iph,j} =a_{i+1,j}-a_{i,j} \qquad \text{ and } \qquad [\![a]\!]_{i,\jph} =a_{i,j+1}-a_{i,j}.
	$$
	The semi-discrete finite difference schemes for \eqref{eq:cgl_ref_godunov} is given by,
	\begin{align}\label{eq:semi-discrete_fd}
		\frac{d}{d t}\con_{i,j}&+\frac{\F_{x, i+\frac{1}{2}, j}-\F_{x, i-\frac{1}{2}, j}}{\Delta x}+\frac{\F_{y, i, j+\frac{1}{2}}-\F_{y, i, j-\frac{1}{2}}}{\Delta y}\nonumber\\
		&+\phi'(\evar_{i, j})^\top\bigg(\left(\frac{\p B_x}{\p x}\right)_{i,j} + \left(\frac{\p B_y}{\p y}\right)_{i,j}\bigg)\nonumber \\
		&+\bc_x(\con_{i,j})\left(\frac{\p \con}{\p x}\right)_{i,j} +\bc_y(\con_{i,j}) \left(\frac{\p \con}{\p y}\right)_{i,j}=0.
	\end{align}
	
	Here $\F_{x, i+\frac{1}{2}, j}$ and $\F_{ y,i, j+\frac{1}{2}}$ are the numerical fluxes in $x$ and $y$ directions, respectively. Furthermore, $\left(\frac{\p B_x}{\p x}\right)_{i,j},\left(\frac{\p B_y}{\p y}\right)_{i,j},\left(\frac{\p \con}{\p x}\right)_{i,j}$ and $\left(\frac{\p \con}{\p y}\right)_{i,j}$ are discretize using central difference schemes of suitable orders. We will first construct an entropy conservative scheme for \eqref{eq:cgl_ref_godunov} by designing entropy conservative fluxes.

	\subsection{Entropy conservative schemes}
	\label{subsec:ent_cons_schemes}
	To design entropy conservative schemes of the form \eqref{eq:semi-discrete_fd}, we will combine two key ideas. The first is to design the entropy conservative scheme for the symmetrized conservative part. We achieve this by following \cite{chandrashekar2016entropy}, where a suitable entropy conservative flux is used to design second-order numerical schemes. Here, we will also extend this scheme to higher-order schemes as well. The second idea is using the fact that the nonconservative terms in reformulated equations do not contribute to the entropy evolution.
	
	Following \cite{tadmor2003entropy,chandrashekar2016entropy}, we have the following result:
	
	\begin{thm}\label{thm:4.1}
		Assume that the numerical fluxes $\tilde{\F}_{x,i+\frac{1}{2}, j}, ~\tilde{\F}_{ y,i, j+\frac{1}{2}}$ consistent with continuous fluxes in $x$ and $y$ direction satisfy,
		\begin{align}
			\quad \jump{\evar}_{\iph,j}  \cdot \tilde{\F}_{x,i+\frac{1}{2}, j}&=\jump{\mathcal{F}_x}_{\iph,j}- \jump{\phi}_{\iph,j} \bar{B}_{x, i+\frac{1}{2}, j},\label{eq:consrvative_flux_x}&\\
			\quad\jump{\evar}_{i,\jph}  \cdot \tilde{\F}_{y,i,j+\frac{1}{2}}&=\jump{\mathcal{F}_y}_{i,\jph}- \jump{\phi}_{i,\jph} \bar{B}_{y,i,j +\frac{1}{2}}.\label{eq:consrvative_flux_y}
		\end{align}
		Furthermore, the non-conservative terms $\left(\frac{\p B_x}{\p x}\right)_{i,j},\left(\frac{\p B_y}{\p y}\right)_{i,j},\left(\frac{\p \con}{\p x}\right)_{i,j}$ and $\left(\frac{\p \con}{\p y}\right)_{i,j}$ are discretized with second order central difference formula. Then the semi-discrete finite difference scheme \eqref{eq:semi-discrete_fd} is second-order accurate and entropy conservative, i.e., it satisfies
		\begin{align}
			\label{eq:ent_eql_2nd_order}
			\frac{d}{d t} \ent\left(\con_{i, j}\right)+\frac{\tilde{\entf}_{x, i+\frac{1}{2}, j}-\tilde{\entf}_{x, i-\frac{1}{2}, j}}{\Delta x}+\frac{\tilde{\entf}_{y, i, j+\frac{1}{2}}-\tilde{\entf}_{y, i, j-\frac{1}{2}}}{\Delta y}=0
		\end{align}
		where the numerical entropy fluxes consistent with \eqref{eq:ent_x_flux_mod} and \eqref{eq:ent_y_flux_mod} are given by
		\begin{align}
			\label{eq:consrvative_ent_flux_x}
			&\tilde{\entf}_{x, i+\frac{1}{2}, j}=\tilde{\entf}_{x}(\con_{i,j},\con_{i+1,j})=\bar{\evar}_{i+\frac{1}{2}, j} \cdot \tilde{\F}_{x,i+\frac{1}{2}, j}+\bar{\phi}_{i+\frac{1}{2}, j} \bar{B}_{x, i+\frac{1}{2}, j}-\bar{\mathcal{F}}_{x, i+\frac{1}{2}, j}, \\
			&\tilde{\entf}_{y, i, j+\frac{1}{2}}=\tilde{\entf}_{y}(\con_{i,j},\con_{i,j+1})=\bar{\evar}_{i, j+\frac{1}{2}} \cdot \tilde{\F}_{ y,i, j+\frac{1}{2}}+\bar{\phi}_{i, j+\frac{1}{2}} \bar{B}_{y, i, j+\frac{1}{2}}-\bar{\mathcal{F}}_{y, i, j+\frac{1}{2}}.		\label{eq:consrvative_ent_flux_y}
		\end{align}
	\end{thm}
	\begin{proof}
		Following \cite{chandrashekar2016entropy}, we note that,
		\begin{align}\label{eq:theorem4p1_1}
			\tilde{\entf}_{x, i+\frac{1}{2}, j}-\tilde{\entf}_{x, i-\frac{1}{2}, j} \nonumber&= \bar{\evar}_{i+\frac{1}{2}, j} \cdot \tilde{\F}_{x,i+\frac{1}{2}, j} - \bar{\evar}_{i-\frac{1}{2}, j} \cdot \tilde{\F}_{x,i-\frac{1}{2}, j}\\
			&+\bar{\phi}_{i+\frac{1}{2}, j} \bar{B}_{x, i+\frac{1}{2}, j} - \bar{\phi}_{i-\frac{1}{2}, j} \bar{B}_{x, i-\frac{1}{2}, j} -\left(\bar{\mathcal{F}}_{x, i+\frac{1}{2}, j} - \bar{\mathcal{F}}_{x, i-\frac{1}{2}, j}\right)
		\end{align}
		Now,
		\begin{align*}
			\bar{\mathcal{F}}_{x, i+\frac{1}{2}, j}-\bar{\mathcal{F}}_{x, i-\frac{1}{2}, j} &= 
			\frac{1}{2}\left(\mathcal{F}_{x, i+1, j}+\mathcal{F}_{x, i, j}\right) +\frac{1}{2}\left(\mathcal{F}_{x, i, j}+\mathcal{F}_{x, i-1, j}\right).
		\end{align*}
		Using \eqref{eq:consrvative_flux_x}, we get,
		\begin{align}\label{eq:theorem4p1_2}
			\bar{\mathcal{F}}_{x, i+\frac{1}{2}, j}-\bar{\mathcal{F}}_{x, i-\frac{1}{2}, j} \nonumber&= \frac{1}{2}\left(\left(\evar_{i+1, j}-\evar_{i, j}\right) \cdot \tilde{\F}_{x,i+\frac{1}{2}, j} + \left(\phi_{i+1, j}-\phi_{i, j}\right) \bar{B}_{x, i+\frac{1}{2}, j}\right)\\
			&+\frac{1}{2}\left(\left(\evar_{i, j}-\evar_{i-1, j}\right) \cdot \tilde{\F}_{x,i-\frac{1}{2}, j} + \left(\phi_{i, j}-\phi_{i-1, j}\right) \bar{B}_{x, 1-\frac{1}{2}, j}\right).
		\end{align}
		Substituting \eqref{eq:theorem4p1_2} in \eqref{eq:theorem4p1_1} results in, 
		\begin{align}
			\label{eq:entf_diff}
			\tilde{\entf}_{x, i+\frac{1}{2}, j}-\tilde{\entf}_{x, i-\frac{1}{2}, j} &= \evar_{i,j}\cdot\left(\tilde{\F}_{x,i+\frac{1}{2}, j}-\tilde{\F}_{x,i-\frac{1}{2}, j}\right) + \phi_{i,j}\left(\bar{B}_{x, i+\frac{1}{2}, j}-\bar{B}_{x, i-\frac{1}{2}, j}\right)\\
			&= \evar_{i,j}\cdot\left(\tilde{\F}_{x,i+\frac{1}{2}, j}-\tilde{\F}_{x,i-\frac{1}{2}, j}\right) +\Dx \evar_{i,j}\cdot \phi'(\con_{i,j})\left( \left(\frac{\p B_x}{\p x}\right)_{i,j}\right)\nonumber\\
			&+ \Dx \evar_{i,j}^\top \bc_x(\con_{i,j})\left(\frac{\p \con}{\p x}\right)_{i,j}\nonumber.
		\end{align}
		Similarly, using \eqref{eq:consrvative_flux_y}, we can get, 
		\begin{align*}
			\tilde{\entf}_{y, i,j+\frac{1}{2}}-\tilde{\entf}_{y, i,j-\frac{1}{2}} &= \evar_{i,j}\cdot\left(\tilde{\F}_{y,i,j+\frac{1}{2}}-\tilde{\F}_{y,i,j-\frac{1}{2}}\right) +\Dy \evar_{i,j}\cdot \phi'(\con_{i,j})\left( \left(\frac{\p B_y}{\p y}\right)_{i,j}\right)\\
			&+ \Dy \evar_{i,j}^\top \bc_y(\con_{i,j})\left(\frac{\p \con}{\p y}\right)_{i,j}.
		\end{align*}
		Combining, we get the discrete entropy equality \eqref{eq:ent_eql_2nd_order}. 
	\end{proof}
	%
	%
	%
	\subsubsection{Entropy conservative fluxes}
	We now aim to design second order entropy conservative fluxes $\tilde{\F}_{x,i+\frac{1}{2}, j}$ and  $\tilde{\F}_{y,i,j+\frac{1}{2},}$ which satisfy  \eqref{eq:consrvative_flux_x} and  \eqref{eq:consrvative_flux_y}, respectively. To simplify the presentation, we will drop the mesh index and consider left and right states denoted by $\con_l$ and $\con_r$. Using this, let us denote,
	$$
	\jump{a}= a_r-a_l \qquad \text{ and } \qquad \bar{a} = \frac{a_r+a_l}{2}.
	$$
	Then we aim to construct $\tilde{\F}_{x} (\con_l,\con_r)$, such that,
	$$
	\jump{\evar}\cdot \tilde{\F}_x = \jump{\mathcal{F}_x} - \bar{B}_x\jump{\phi}.
	$$
	We will first proceed in $x$-direction. The case for $y$-direction is similar. We also note that the above equation is a scalar. So, we need a vector $\tilde{\F}_x$, satisfying the above scalar equations.  Hence, in general, we can have many entropy-conservative numerical fluxes. For Euler equations of compressible equations, several such fluxes are presented in \cite{tadmor2003entropy,chandrashekar2013kinetic,ismail2009affordable}. 
	$$
	\jump{ab} = \bar{a}\jump{b}+\bar{b}\jump{a}.
	$$
	Following \cite{roe2006affordable}, let us introduce the logarithmic average.
	$$
	a^{\ln} = \frac{\jump{a}}{\jump{\ln{a}}}.
	$$
	When $a_l\approx a_r,$ we use the Taylor expansion-based formula from \cite{ismail2009affordable}.  Let us first simplify the $\jump{s}$, to get,
	\begin{align*}
		[\![ s]\!] & =[\![\left(\ln\left(\pll\per^2/\rho^3\right)-2\ln(\rho)\right)]\!]                                                           & \\
		& =[\![\left(\ln\left(\pll/\rho\right)+\ln\left(\per^2/\rho^2\right)-2\ln(\rho)\right)]\!]                                               & \\
		&=-[\![\ln(\beta_{\parallel})]\!]-2[\![\ln(\beta_{\perp})]\!]-2[\![\ln(\rho)]\!]                            & \\                                                    
		& =-\dfrac{[\![\beta_{\parallel}]\!]}{{\beta}^{\ln}_{\parallel}}-2\dfrac{[\![ \beta_{\perp}]\!]}{{\beta}^{\ln}_{\perp}}-2\dfrac{[\![ \rho]\!]}{{\rho}^{\ln}}
	\end{align*}
	Following the approach in \cite{chandrashekar2013kinetic}, we expand the jumps of entropy variables $\evar$ in terms of the jumps in $\rho, u_{x}, u_{y}, u_{z}, \beta_\parallel, \beta_\perp, B_{x}, B_{y}, B_{z}$ as
	\begin{align*}
		[\![ \evar]\!] =\left[\begin{array}{c}
			\frac{2}{\rho^{\ln}} [\![ \rho]\!] +\frac{1}{{\beta}^{\ln}_\parallel} [\![ \beta_\parallel]\!] - 2 \bar{\beta}_\perp\left(\bar{u}_{x} [\![ u_{x}]\!] +\bar{u}_{y} [\![ u_{y}]\!] +\bar{u}_{z} [\![ u_{z}]\!] \right)+\left[\frac{2}{{\beta}^{\ln}_\perp}-\overline{|u|^{2}}\right][\![ \beta_\perp]\!] \\
			2 \bar{\beta}_\perp [\![ u_{x}]\!]+2 \bar{u}_{x} [\![ \beta_\perp]\!] \\
			2 \bar{\beta}_\perp [\![ u_{y}]\!]+2 \bar{u}_{y} [\![ \beta_\perp]\!] \\
			2 \bar{\beta}_\perp [\![ u_{z}]\!]+2 \bar{u}_{z} [\![ \beta_\perp]\!] \\
			-[\![ \beta_\parallel]\!] +[\![ \beta_\perp]\!]\\
			-2 [\![ \beta_\perp]\!] \\
			2 \bar{\beta}_\perp [\![ B_{x}]\!]+2 \bar{B}_{x} [\![ \beta_\perp]\!] \\
			2 \bar{\beta}_\perp [\![ B_{y}]\!]+2 \bar{B}_{y} [\![ \beta_\perp]\!] \\
			2 \bar{\beta}_\perp [\![ B_{z}]\!]+2 \bar{B}_{z} [\![ \beta_\perp]\!]
		\end{array}\right].
	\end{align*}
	Let us consider the numerical flux,
	$$\tilde{\F}_x=\left[ \tilde{F}_{x}^{(1)},\tilde{F}_{x}^{(2)}, \tilde{F}_{x}^{(3)}, \tilde{F}_{x}^{(4)}, \tilde{F}_{x}^{(5)}, \tilde{F}_{x}^{(6)}, \tilde{F}_{x}^{(7)},\tilde{F}_{x}^{(8)},\tilde{F}_{x}^{(9)} \right]^{\top}.
	$$
	Then $\jump{\evar}\cdot \tilde{\F}_x$ can be written as,
	\begin{align}
		[\![ \evar]\!] \cdot \tilde{\F}_{x}=& \frac{2\tilde{F}_{x}^{(1)}}{{\rho}^{\ln}} [\![ \rho]\!] +\left(-2 \bar{u}_{x} \bar{\beta}_\perp \tilde{F}_{x}^{(1)}+2 \bar{\beta}_\perp \tilde{F}_{x}^{(2)}\right) [\![ u_{x}]\!]\nonumber \\
		&+ \left(-2 \bar{u}_{y} \bar{\beta}_\perp \tilde{F}_{x}^{(1)}+2 \bar{\beta}_\perp \tilde{F}_{x}^{(3)}\right) [\![ u_{y}]\!] \nonumber\\
		\label{eq:vdotf}
		&+\left(-2 \bar{u}_{z} \bar{\beta}_\perp \tilde{F}_{x}^{(1)}+2 \bar{\beta}_\perp \tilde{F}_{x}^{(4)}\right) [\![ u_{z}]\!] +\left(\frac{\tilde{F}_{x}^{(1)}}{{\beta}^{\ln}_\parallel}-\tilde{F}_{x}^{(5)}\right)[\![ \beta_\parallel]\!]\nonumber \\
		&+\Bigg[\left(\frac{2}{ {\beta}^{\ln}_\perp}-\overline{\left|u\right|^{2}}\right) \tilde{F}_{x}^{(1)}+2 \bar{u}_{x} \tilde{F}_{x}^{(2)}+2 \bar{u}_{y} \tilde{F}_{x}^{(3)}+2 \bar{u}_{z} \tilde{F}_{x}^{(4)}+\tilde{F}_{x}^{(5)}-2 \tilde{F}_{x}^{(6)}\nonumber\\
		&+2 \bar{B}_{x} \tilde{F}_{x}^{(7)}+2 \bar{B}_{y} \tilde{F}_{x}^{(8)}+2 \bar{B}_{z} \tilde{F}_{x}^{(9)}\Bigg] [\![ \beta_\perp]\!]\nonumber \\
		&+2 \bar{\beta}_\perp \tilde{F}_{x}^{(7)} [\![ B_x]\!]+2 \bar{\beta}_\perp \tilde{F}_{x}^{(8)} [\![ B_y]\!]+2 \bar{\beta}_\perp \tilde{F}_{x}^{(9)} [\![ B_z]\!] 
	\end{align}
	Similarly, we simplify the $[\![ \mathcal{F}_{x}]\!]$ and $[\![ \phi]\!]$, to get,
	$$
	\begin{aligned}
		[\![ \mathcal{F}_{x}]\!]=& [\![\left(2\rho u_{x}\right)]\!] + [\![\left(\beta_\perp u_{x}|\boldsymbol{B}|^{2}\right)]\!] \\
		=& 2\bar{u}_{x} [\![ \rho]\!] + \left(2\bar{\rho}+\bar{\beta}_\perp \overline{|\B|^{2}}\right) [\![ u_x]\!]+\bar{u}_{x} \overline{\left|\boldsymbol{B}\right|^{2}} [\![\beta_\perp]\!] +2 \overline{\beta_\perp u_{x}} \bar{B}_{x} [\![ B_{x}]\!] \\
		&+2 \overline{\beta_\perp u_{x}} \bar{B}_{y} [\![ B_{y}]\!] + 2 \overline{\beta_\perp u_{x}} \bar{B}_{z} [\![ B_{z}]\!],
	\end{aligned}
	$$
	and
	$$
	\begin{aligned}
		[\![ \phi]\!]=& 2 [\![(\beta_\perp(\boldsymbol{u} \cdot \boldsymbol{B})) ]\!] \\
		=& 2 \bar{\beta}_\perp \bar{B}_{x} [\![ u_x]\!] +2 \bar{\beta}_\perp \bar{B}_{y} [\![ u_y]\!] + 2 \bar{\beta}_\perp \bar{B}_{z} [\![ u_z]\!] + 2\left(\bar{u}_{x} \bar{B}_{x}+\bar{u}_{y} \bar{B}_{y}+\bar{u}_{z} \bar{B}_{z}\right) [\![ \beta_\perp]\!]\\
		&+2 \overline{\beta u_{x}} [\![ B_{x}]\!] + 2 \overline{\beta u_{y}} [\![ B_{y}]\!] +2 \overline{\beta u_{z}} [\![ B_{z}]\!].
	\end{aligned}
	$$
	Combining, we get,
	
	\begin{align}
		[\![ \mathcal{F}_{x}]\!] - \bar{B}_{x} [\![ \phi]\!] &= 2\bar{u}_{x} [\![ \rho]\!]
		+\left(2\bar{\rho} + \bar{\beta}_\perp \overline{|\boldsymbol{B}|^{2}}-2 \bar{\beta}_\perp \bar{B}_{x} \bar{B}_{x}\right) [\![ u_x]\!]  \nonumber\\
		&-2 \bar{\beta}_\perp \bar{B}_{x} \bar{B}_{y} [\![ u_y]\!] -2 \bar{\beta}_\perp \bar{B}_{x} \bar{B}_{z} [\![ u_z]\!] \nonumber\\
		&+\left[\bar{u}_{x} \overline{|\B|^{2}}-2\left(\bar{u}_{x} \bar{B}_{x}+\bar{u}_{y} \bar{B}_{y}+\bar{u}_{z} \bar{B}_{z}\right) \bar{B}_{x}\right] [\![\beta_\perp]\!] \nonumber\\
		\label{eq:vdotf_rhs}
		&+2\left(\overline{\beta_\perp u_{x}} \bar{B}_{y}-\overline{\beta_\perp u_{y}} \bar{B}_{x}\right) [\![ B_{y}]\!] \nonumber\\
		&+2\left(\overline{\beta_\perp u_{x}} \bar{B}_{z}-\overline{\beta_\perp u_{z}} \bar{B}_{z}\right) [\![ B_{z}]\!]
	\end{align}
	
	Equating the coefficients of $[\![ \rho]\!], [\![ u_x]\!], [\![ u_y]\!], [\![ u_z]\!], [\![ \beta_\parallel]\!], [\![\beta_\perp]\!], [\![ B_{x}]\!], [\![ B_{y}]\!]$ and $[\![ B_{z}]\!]$ in \eqref{eq:vdotf} and \eqref{eq:vdotf_rhs}, we get,
	$$
	\begin{aligned}
		\tilde{F}_{x}^{(1)}=& {\rho}^{\ln} \bar{u}_{x},~\tilde{F}_{x}^{(2)}= \frac{\bar{\rho}}{ \bar{\beta}_\perp}+\bar{u}_{x} \tilde{F}_{x}^{(1)}+\frac{1}{2} \overline{\left|\boldsymbol{B}\right|^{2}}-\bar{B}_{x} \bar{B}_{x} \\
		\tilde{F}_{x}^{(3)}=& \bar{u}_{y} \tilde{F}_{x}^{(1)}-\bar{B}_{x} \bar{B}_{y},~\tilde{F}_{x}^{(4)}= \bar{u}_{z} \tilde{F}_{x}^{(1)}-\bar{B}_{x} \bar{B}_{z},~\tilde{F}_{x}^{(5)}= \frac{\tilde{F}_{x}^{(1)}}{{\beta}^{\ln}_\parallel},\\
		\tilde{F}_{x}^{(7)}=& 0,~\tilde{F}_{x}^{(8)}= \frac{1}{\bar{\beta}_\perp}\left(\overline{\beta_\perp u_{x}} \bar{B}_{y}-\overline{\beta_\perp u_{y}} \bar{B}_{x}\right),~\tilde{F}_{x}^{(9)}= \frac{1}{\bar{\beta}_\perp}\left(\overline{\beta_\perp u_{x}} \bar{B}_{z}-\overline{\beta u_{z}}  \bar{B_{x}}\right),\\
		\tilde{F}_{x}^{(6)}=& \frac{1}{2}\left[\frac{2}{ {\beta}^{\ln}_\perp}-\overline{|u|^{2}}\right] \tilde{F}_{x}^{(1)}+\bar{u}_{x} \tilde{F}_{x}^{(2)}+\bar{u}_{y} \tilde{F}_{x}^{(3)}+\bar{u}_{z} \tilde{F}_{x}^{(4)} +\frac{\tilde{F}_{x}^{(5)}}{2}\\
		&+\bar{B}_{x} \tilde{F}_{x}^{(7)}+\bar{B}_{y} \tilde{F}_{x}^{(8)}+\bar{B}_{z} \tilde{F}_{x}^{(9)}-\frac{1}{2} \bar{u}_{x} \overline{\left|\boldsymbol{B}\right|^{2}}+\left(\bar{u}_{x} \bar{B}_{x}+\bar{u}_{y} \bar{B}_{y}+\bar{u}_{z} \bar{B}_{z}\right) \bar{B}_{x},
	\end{aligned}
	$$
	Similarly, we obtain the  entropy conservative numerical flux in $y$-direction,
	$\tilde{\F}_{y}=\left[\tilde{F}_{y}^{(1)}, \tilde{F}_{y}^{(2)}, \tilde{F}_{y}^{(3)}, \tilde{F}_{y}^{(4)},\right.$ $ \left.\tilde{F}_{y}^{(5)}, \tilde{F}_{y}^{(6)}, \tilde{F}_{y}^{(7)}, \tilde{F}_{y}^{(8)}, \tilde{F}_{y}^{(9)}\right]^\top$, where
	$$
	\begin{aligned}
		\tilde{F}_{y}^{(1)}=& {\rho}^{\ln} \bar{u}_{y},~\tilde{F}_{y}^{(2)}= \bar{u}_{x} \tilde{F}_{y}^{(1)}-\bar{B}_{x} \bar{B}_{y}, \\
		\tilde{F}_{y}^{(3)}=& \frac{\bar{\rho}}{ \bar{\beta}_\perp}+\bar{u}_{y} \tilde{F}_{y}^{(1)}+\frac{1}{2} \overline{\left|\boldsymbol{B}\right|^{2}}-\bar{B}_{y} \bar{B}_{y},~\tilde{F}_{y}^{(4)}= \bar{u}_{z} \tilde{F}_{y}^{(1)}-\bar{B}_{y} \bar{B}_{z},~\tilde{F}_{y}^{(5)}= \frac{\tilde{F}_{y}^{(1)}}{{\beta}^{\ln}_\parallel}, \\
		\tilde{F}_{y}^{(7)}=& \frac{1}{\bar{\beta}_\perp}\left(\overline{\beta_\perp u_{y}} \bar{B}_{x}-\overline{\beta_\perp u_{x}} \bar{B}_{y}\right),~\tilde{F}_{y}^{(8)}= 0,~\tilde{F}_{y}^{(9)}= \frac{1}{\bar{\beta}_\perp}\left(\overline{\beta_\perp u_{y}} \bar{B}_{z}-\overline{\beta_\perp u_{z}}  \bar{B_{y}}\right),\\
		\tilde{F}_{y}^{(6)}=& \frac{1}{2}\left[\frac{2}{ {\beta}^{\ln}_\perp}-\overline{|u|^{2}}\right] \tilde{F}_{y}^{(1)}+\bar{u}_{x} \tilde{F}_{y}^{(2)}+\bar{u}_{y} \tilde{F}_{y}^{(3)}+\bar{u}_{z} \tilde{F}_{y}^{(4)} +\frac{\tilde{F}_{y}^{(5)}}{2}\\
		&+\bar{B}_{x} \tilde{F}_{y}^{(7)}+\bar{B}_{y} \tilde{F}_{y}^{(8)}+\bar{B}_{z} \tilde{F}_{y}^{(9)}-\frac{1}{2} \bar{u}_{y} \overline{\left|\boldsymbol{B}\right|^{2}}+\left(\bar{u}_{x} \bar{B}_{x}+\bar{u}_{y} \bar{B}_{y}+\bar{u}_{z} \bar{B}_{z}\right) \bar{B}_{y},
	\end{aligned}
	$$
	It is easy to check that the above numerical fluxes are consistent. So, the numerical scheme \eqref{eq:semi-discrete_fd} is now well defined with $\tilde{\F}_{x,\iph,j} =\tilde{\F}_x(\con_{i,j},\con_{i+1,j})$ and $\tilde{\F}_{y,i,\jph} =\tilde{\F}_y(\con_{i,j},\con_{i,j+1}).$
	\subsubsection{Higher order entropy conservative schemes}
	In \cite{leFloch2002}, authors have constructed $2p$-th ($p$ is any positive integer) order entropy conservative numerical fluxes by combining second-order numerical flux. We will use the same formulations and construct a fourth-order numerical flux. Using the approach we present here, we can further extend to design $2p$-th order entropy conservative schemes. 
	
	Following \cite{leFloch2002}, let us define the $4$-th order numerical fluxes,
	\begin{align}
		\tilde{\F}_{x,i+\frac{1}{2},j}^4&=\frac{4}{3}\tilde{\F}_{x}(\con_{i,j},\con_{i+1,j})-\frac{1}{6} \bigg( \tilde{\F}_{x} (\con_{i-1,j},\con_{i+1,j})+
		\tilde{\F}_{x}(\con_{i,j},\con_{i+2,j}) \bigg).\label{eq:4thorder_numflux_x}
	\end{align}
	and 
	\begin{align}
		\tilde{\F}_{y,i,j+\frac{1}{2}}^4&=\frac{4}{3}\tilde{\F}_{y}(\con_{i,j},\con_{i,j+1})-\frac{1}{6} \bigg( \tilde{\F}_{y} (\con_{i,j-1},\con_{i,j+1})+
		\tilde{\F}_{y}(\con_{i,j},\con_{i,j+2}) \bigg).\label{eq:4thorder_numflux_y}
	\end{align}
	Using these numerical fluxes, let us consider the following fourth-order scheme:
	\begin{align}\label{eq:semi-discrete_4th_order}
		\frac{d}{d t}\con_{i,j}&+\frac{\tilde{\F}^4_{x, i+\frac{1}{2}, j}-\tilde{\F}^4_{x, i-\frac{1}{2}, j}}{\Delta x}+\frac{\tilde{\F}^4_{y, i, j+\frac{1}{2}}-\tilde{\F}^4_{y, i, j-\frac{1}{2}}}{\Delta y}\nonumber\\
		&+\phi'(\evar_{i, j})^\top\bigg(\left(\frac{\p B_x}{\p x}\right)_{i,j} + \left(\frac{\p B_y}{\p y}\right)_{i,j}\bigg)\nonumber \\
		&+\bc_x(\con_{i,j})\left(\frac{\p \con}{\p x}\right)_{i,j} +\bc_y(\con_{i,j}) \left(\frac{\p \con}{\p y}\right)_{i,j}=0.
	\end{align}
	where, the nonconservative terms $\left(\frac{\p B_x}{\p x}\right)_{i,j},\left(\frac{\p B_y}{\p y}\right)_{i,j},\left(\frac{\p \con}{\p x}\right)_{i,j}$ and $\left(\frac{\p \con}{\p y}\right)_{i,j}$ are discretized using fourth order central difference formula. Then we have the following results:
	
	\begin{thm}\label{thm:4.2}
		The numerical scheme \eqref{eq:semi-discrete_4th_order} is entropy conservative i.e. it satisfies,
		\begin{align}\label{eq:entropy_conservative_hi}
			\frac{d}{d t} \ent\left(\con_{i,j}\right)+\frac{\tilde{\entf}^4_{x,i+\frac{1}{2}, j}-\tilde{\entf}^4_{x,i-\frac{1}{2}, j}}{\Delta x} +\frac{\tilde{\entf}^4_{y,i,j+\frac{1}{2}, }-\tilde{\entf}^4_{y,i,j-\frac{1}{2}}}{\Delta y} =0
		\end{align}
		where the consistent numerical entropy fluxes are given by
		\begin{align*}
			&\tilde{\entf}^4_{x,i+\frac{1}{2}, j}=\frac{4}{3}\tilde{\entf}_{x}(\con_{i,j},\con_{i+1,j})-\frac{1}{6} \bigg( \tilde{\entf}_{x} (\con_{i-1,j},\con_{i+1,j})+
			\tilde{\entf}_{x}(\con_{i,j},\con_{i+2,j}) \bigg),\\
			&\tilde{\entf}^4_{y,i,j+\frac{1}{2}}=\frac{4}{3}\tilde{\entf}_{y}(\con_{i,j},\con_{i,j+1})\nonumber-\frac{1}{6} \bigg( \tilde{\entf}_{y} (\con_{i,j-1},\con_{i,j+1})+
			\tilde{\entf}_{y}(\con_{i,j},\con_{i,j+2}) \bigg).
		\end{align*}
	\end{thm}
	
	\begin{proof}
		We will first consider the $x$-directional case. Note that fourth-order central discretization of $\frac{\p B_x}{\p x}$ can be written as,
		$$
		\left(\frac{\p B_x}{\p x}\right)_{i,j} = \frac{\bar{B}^4_{x,\iph,j} -\bar{B}^4_{x,i,\jph}  }{\Dx}
		$$
		where
		\begin{align*}
			\bar{B}_{x,i+\frac{1}{2},j}^4&=\frac{4}{3}\bar{B}_{x,i+\frac{1}{2},j} - \frac{1}{6} \left( \frac{B_{x,i-1,j} +B_{x,i+1,j} }{2} + \frac{B_{x,i,j} +B_{x,i+2,j} }{2}  \right).
		\end{align*}
		Now multiply the right hand side of \eqref{eq:semi-discrete_4th_order} with $\evar_{i,j}$, we get,
		\begin{align*}
			\evar_{i,j}\left(\tilde{\F}_{x,i+\frac{1}{2},j}^4-\tilde{\F}_{x,i-\frac{1}{2},j}^4\right)&+ \evar_{i,j}\phi'(\evar_{i,j})	\left(\bar{B}_{x,i+\frac{1}{2},j}^4-\bar{B}_{x,i-\frac{1}{2},j}^4\right)+\\
			\evar_{i,j}\bc_x(\con_{i,j})\left(\frac{\p \con}{\p x}\right)_{i,j}
			&=		\evar_{i,j}\left(\tilde{\F}_{x,i+\frac{1}{2},j}^4-\tilde{\F}_{x,i-\frac{1}{2},j}^4\right)\\
			& + \phi(\evar_{i,j}) \left(\bar{B}_{x,i+\frac{1}{2},j}^4-\bar{B}_{x,i-\frac{1}{2},j}^4\right)\\
			&=\evar_{i,j}\frac{4}{3}\left(\tilde{\F}_{x}(\con_{i,j},\con_{i+1,j} )- \tilde{\F}_{x}(\con_{i-1,j},\con_{i,j})\right) \\
			&+\phi(\evar_{i,j}) \frac{4}{3}\left(\bar{B}_{x,\iph,j}- \bar{B}_{x,\imh,j}\right) \\
			&- \evar_{i,j}\frac{1}{6}\left( \tilde{\F}_{x}(\con_{i,j},\con_{i+2,j}) - \tilde{\F}_{x}(\con_{i-2,j},\con_{i,j})\right) \\
			&- \phi(\evar_{i,j})\frac{1}{6}\left(  \frac{B_{x,i,j} +B_{x,i+2,j} }{2}  -  \frac{B_{x,i-2,j} +B_{x,i,j} }{2} \right) \\
			&=\frac{4}{3}\left(\tilde{\entf}_{x}(\con_{i,j},\con_{i+1,j})- \tilde{\entf}_{x}(\con_{i-1,j},\con_{i,j})\right) ~\text{ using \eqref{eq:entf_diff} }\\
			&-\frac{1}{6}\left(\tilde{\entf}_x(\con_{i,j},\con_{i+2,j}) -\tilde{\entf}_{x}(\con_{i-2,j},\con_{i,j})  \right) \\
			&=\left(\tilde{\entf}^4_{x,\iph,j}- \tilde{\entf}^4_{x,\imh,j}\right). 
		\end{align*}
		Similarly, for $y$-direction,
		\begin{align*}
			\left(\tilde{\entf}^4_{y,i,\jph}- \tilde{\entf}^4_{y,i,\jmh}\right) &=\evar_{i,j}\left(\tilde{\F}_{y,i,j+\frac{1}{2}}^4-\tilde{\F}_{y,i,j-\frac{1}{2}}^4\right)\\
			&+\evar_{i,j}\phi'(\evar_{i,j})	\left(\frac{\p B_y}{\p y}\right)_{i,j}+\evar_{i,j}\bc_y(\con_{i,j})\left(\frac{\p \con}{\p y}\right)_{i,j}.
		\end{align*}
		Combining $x$ and $y$-direction entropy flux difference, we now get \eqref{eq:entropy_conservative_hi}.
	\end{proof}

	\subsection{Entropy stable schemes}
	\label{subsec:ent_stable_schemes}
	Entropy conservative schemes produce oscillations when the solutions contain shock. To ensure a stable scheme, we need to have entropy stability of the numerical scheme. We will now modify the numerical flux to ensure the entropy stability of the scheme. Following \cite{tadmor2003entropy}, we introduce the modified numerical flux as follows:
	
	\begin{equation}
		\begin{aligned}
			{\hat{\F}}_{x,i+\frac{1}{2},j} =\tilde{\F}_{x,i+\frac{1}{2},j} - \frac{1}{2} \textbf{D}_{x,i+\frac{1}{2},j}[\![ \evar]\!]_{i+\frac{1}{2},j},~
			{\hat{\F}}_{y,i,j+\frac{1}{2}} = \tilde{\F}_{y,i,j+\frac{1}{2}} - \frac{1}{2} \textbf{D}_{y,i,j+\frac{1}{2}}[\![ \evar]\!]_{i,j+\frac{1}{2}},
			\label{es_numflux}
		\end{aligned}
	\end{equation}
	where $\textbf{D}_{x,i+\frac{1}{2},j}$ and $\textbf{D}_{y,i,j+\frac{1}{2}}$ are symmetric positive definite matrices. Combining Theorem \ref{thm:4.1}, with Theorem 3.3 in \cite{chandrashekar2016entropy}, we have the following result:
	\begin{thm} The numerical scheme \eqref{eq:semi-discrete_fd} with the modified numerical fluxes \eqref{es_numflux} is entropy stable, i.e., the solutions satisfies,
		\begin{equation*}
			\frac{d}{dt}  \ent(\con_{i,j})  +\frac{1}{\Delta x} \left( \hat{\entf}_{x,i+\frac{1}{2},j} - \hat{\entf}_{x,i-\frac{1}{2},j}\right)+\frac{1}{\Delta y}\left( \hat{\entf}_{y,i,j+\frac{1}{2}} - \hat{\entf}_{y,i,j-\frac{1}{2}}\right) \le 0,
		\end{equation*}
		with consistent numerical entropy flux functions,
		\begin{equation*}
			\begin{aligned}
				\hat{\entf}_{x,i+\frac{1}{2},j}=  \tilde{\entf}_{x,i+\frac{1}{2},j} - \frac{1}{2}\bar{\evar}_{i+\frac{1}{2},j}^{\top}  \textbf{D}_{x,i+\frac{1}{2},j}[\![ \evar]\!]_{i+\frac{1}{2},j}
		\end{aligned} \end{equation*}
		and	
		\begin{equation*}
			\begin{aligned}
				\qquad \hat{\entf}_{y,i,j+\frac{1}{2}}=    \tilde{\entf}_{y,i,j+\frac{1}{2}} -  \frac{1}{2}\bar{\evar}_{i,j+\frac{1}{2}}^{\top}  \textbf{D}_{y,i,j+\frac{1}{2}}[\![ \evar]\!]_{i,j+\frac{1}{2}}.\end{aligned} \end{equation*}
	\end{thm}
	To define the diffusion part, we consider the symmetrized conservative part \eqref{eq:cgl_modified}, in quasi-linear form, i.e.
	\begin{equation}
		\frac{\p \con}{\p t} + \mathcal{A}_d \frac{\p \con}{\p x} =0 ,\qquad  \mathcal{A}_d = \frac{\p \f_d}{\p \con} + \phi'(\evar)B'_d (\con).
	\end{equation}
	with $d \in \{x,y\}$. As the above system is symmetrizable, we follow \cite{barth1999numerical,chandrashekar2016entropy} to derive entropy scaled eigenvectors in Appendix \ref{scaledrev}. We denote the matrix of entropy scaled right eigenvectors as $\tilde{R}_d$. Furthermore, let us take ${\Lambda_d},\, d \in \{x,y\},$ are $9\times 9$ diagonal matrices of the form
	$${\Lambda_d}=\text{diag}  \left\lbrace \left( \max_{1 \leq k \leq 9} |\lambda^k_d|\right) \mathbf{I}_{9 \times 9} \right\rbrace,  \ \ \, d \in \{x,y\}.$$
	where $\{\lambda^k_d: 1 \leq k \leq 9 \}$ is the set of eigenvalues of the jacobian $\mathcal{A}_d$. Then we define diffusion matrices, 
	\begin{equation}  \label{diffusiontype}
		\textbf{D}_{x,i+\frac{1}{2},j} = \tilde{R}_{x,i+\frac{1}{2},j} \Lambda_{x,i+\frac{1}{2},j} \tilde{R}_{x,i+\frac{1}{2},j}^{\top}~\text{and}~\textbf{D}_{y,i,j+\frac{1}{2}} = \tilde{R}_{y,i,j+\frac{1}{2}} \Lambda_{y,i,j+\frac{1}{2}} \tilde{R}_{y,i,j+\frac{1}{2}}^{ \top}.
	\end{equation}
	With this choice, the numerical scheme with modified fluxes is stable in entropy. However, the scheme contains the jump term $\jump{\evar}$ in both directions, which results in only first-order accuracy, even with second or higher-order entropy conservative numerical flux. One simple way to overcome this is to reconstruct the jumps $[\![\evar]\!]_{i+\frac{1}{2},j}$ and $[\![\evar]\!]_{i,j+\frac{1}{2}}$ using higher order reconstruction process. This indeed leads to higher-order schemes, but it is not possible to prove the entropy stability of the scheme. Instead, we follow \cite{fjordholm2012arbitrarily} to design higher-order entropy stable schemes.
	
	\subsubsection{Higher order entropy stable schemes}
	Following \cite{fjordholm2012arbitrarily}, we define {\em scaled entropy variables} 
	$$\W^{\pm}_{x,i,j}= {R}^{{\top}}_{x,i\pm\frac{1}{2},j}\evar_{i,j}.$$
	The key idea is to reconstruct $\W^{\pm}_{x,i,j}$ instead on $\evar^{\pm}_{x,i,j}$. Let $\hat{\W}^{\pm}_{x,i,j}$, denotes the {\em sign preserving} higher order reconstruction of $\W^{\pm}_{x,i,j}$ in $x$-direction. Then, we define reconstructed $\evar_{x,i,j}^{\pm}$ as,
	$$
	\hat{\evar}^{\pm}_{x,i,j} =  \left\lbrace \tilde{ R}^{{\top}}_{x,i\pm\frac{1}{2},j}\right\rbrace ^{(-1)}\hat{\W}_{x,i,j}^{\pm}.
	$$
	Hence, now the reconstructed $\jump{\evar}_{\iph,j}$ of $k$-th order is defined as,
	$$
	\jump{\hat{\evar}}^k_{x,i,j} = \hat{\evar}^-_{x,i+1,j}-\hat{\evar}^+_{x,i,j}.
	$$
	Using this, we now define higher-order entropy stable numerical flux in $x$-direction as,
	\begin{equation}
		{\hat{\F}}^{k}_{x,i+\frac{1}{2},j}\,=\,\tilde{\F}^{2p}_{x,i+\frac{1}{2},j}\,-\,\frac{1}{2}\,\textbf{D}_{x,i+\frac{1}{2},j}[\![ \bb{\hat{V}}]\!]_{x,i+\frac{1}{2},j}^k.
		\label{eq:entropy_stable_flux_x}
	\end{equation}
	We choose $p=2$ for first and second-order schemes and $p=4$ for third and fourth-order schemes. We can follow a similar process in $y$-direction. To ensure the entropy stability condition of the above scheme, the reconstruction of the scaled entropy variable needs to satisfy {\em sign preserving property} for each component, i.e., we need  
	\begin{equation}
		\label{eq:sign_preser}
		\textrm{sign} \left( {\hat{\W}}^{-}_{x,i+1,j,l} - {\hat{\W}}^{+}_{x,i,j,l} \right) = \textrm{sign} \left(  \W_{i+1,j,l} - \W_{i,j,l} \right) \quad \text{for}\quad 1\le l \le 9
	\end{equation}
	where $l$ denote the $l$-th component of vector $\W$. This can be ensured by using MinMod reconstruction for the second order scheme ($k=2$). For third ($k=3$) and fourth ($k=4$) order schemes, we need to use ENO-based reconstruction \cite{fjordholm2013eno} of the third and fourth order. Similarly, in $y$-direction, we define higher numerical flux as:
	\begin{equation}
		{\hat{\F}}^{k}_{y,i+\frac{1}{2},j}\,=\,\tilde{\F}^{2p}_{y,i+\frac{1}{2},j}\,-\,\frac{1}{2}\,\textbf{D}_{y,i+\frac{1}{2},j}[\![ \bb{\hat{V}}]\!]_{y,i+\frac{1}{2},j}^k.
		\label{eq:entropy_stable_flux_y}
	\end{equation}
	where reconstruction is performed in $y$-direction. 
	\subsection{Semi-discrete entropy stability}\label{subsec:semidisscheme}
	We will now prove the entropy stability of the proposed numerical schemes. Consider the following numerical discretization:
	
	\begin{align}
		\label{eq:semi-discrete_final_scheme}
		\frac{d}{d t}\con_{i,j}&+\frac{\hat{\F}^k_{x, i+\frac{1}{2}, j}-\hat{\F}^k_{x, i-\frac{1}{2}, j}}{\Delta x}+\frac{\hat{\F}^k_{y, i, j+\frac{1}{2}}-\hat{\F}^k_{y, i, j-\frac{1}{2}}}{\Delta y}\nonumber&\\&+\phi'(\evar_{i, j})^\top \bigg(\left(\frac{\p B_x}{\p x}\right)_{i,j} + \left(\frac{\p B_y}{\p y}\right)_{i,j}\bigg)\nonumber \\
		&+\bc_x(\con_{i,j})\left(\frac{\p \con}{\p x}\right)_{i,j} +\bc_y(\con_{i,j}) \left(\frac{\p \con}{\p y}\right)_{i,j}=0,
	\end{align}
	where we use the higher order entropy stable numerical flux \eqref{eq:entropy_stable_flux_x} and  \eqref{eq:entropy_stable_flux_y}. Furthermore, the derivatives
	$\left(\frac{\partial B_x}{\partial x}\right)_{i,j},~\left(\frac{\partial B_y}{\partial y}\right)_{i,j},~\left(\frac{\partial \con}{\partial x}\right)_{i,j}$ and $\left(\frac{\partial \con}{\partial y}\right)_{i,j}$ in the non-conservative term are discretized using second-order central difference for $k=1$ and $k=2$ and fourth order central difference for $k=3$ and $k=4$. Then, we have the following result:

	\begin{thm}
		The semi-discrete scheme \eqref{eq:semi-discrete_final_scheme} is $k$ order accurate and entropy stable, i.e. it satisfies,
		\begin{equation}
			\label{eq:semi-disc_ent_stab}
			\frac{d}{dt}  \ent(\con_{i,j})  +\frac{1}{\Delta x} \left( \hat{\entf}^k_{x,i+\frac{1}{2},j} - \hat{\entf}^k_{x,i-\frac{1}{2},j}\right)+\frac{1}{\Delta y}\left( \hat{\entf}^k_{y,i,j+\frac{1}{2}} - \hat{\entf}^k_{y,i,j-\frac{1}{2}}\right) \le 0,
		\end{equation}
		where $ \hat{\entf}^k_x$ and $ \hat{\entf}^k_y$ are given by,
		\begin{equation*}
			\begin{aligned}
				\hat{\entf}^k_{x,i+\frac{1}{2},j}=  \tilde{\entf}^{2p}_{x,i+\frac{1}{2},j} - \frac{1}{2}\bar{\evar}_{i+\frac{1}{2},j}^{\top}  \textbf{D}_{x,i+\frac{1}{2},j}[\![ \hat{\evar}]\!]^k_{x,i+\frac{1}{2},j}
		\end{aligned} \end{equation*}
		
		and \begin{equation*}
			\begin{aligned}
				\qquad \hat{\entf}^k_{y,i,j+\frac{1}{2}}=    \tilde{\entf}^{2p}_{y,i,j+\frac{1}{2}} -  \frac{1}{2}\bar{\evar}_{i,j+\frac{1}{2}}^{\top}  \textbf{D}_{y,i,j+\frac{1}{2}}[\![ \hat{\evar}]\!]^k_{y,i,j+\frac{1}{2}}.\end{aligned} \end{equation*}
	\end{thm}
	\begin{proof} Accuracy follows from \cite{fjordholm2012arbitrarily}. We also note that from Theorems \eqref{thm:4.1} and \eqref{thm:4.2}, we get,
		\begin{align*}
			\frac{d}{dt}  \ent(\con_{i,j})  &+\frac{1}{\Delta x} \left( \hat{\entf}^k_{x,i+\frac{1}{2},j} - \hat{\entf}^k_{x,i-\frac{1}{2},j}\right)+\frac{1}{\Delta y}\left( \hat{\entf}^k_{y,i,j+\frac{1}{2}} - \hat{\entf}^k_{y,i,j-\frac{1}{2}}\right)\\
			&=-\frac{1}{2 \Dx} \left[ \jump{\evar}_{\iph,j} \textbf{D}_{x,\iph,j}  \jump{\hat{\evar}}^k_{x,\iph,j} + \jump{\evar}_{\imh,j} \textbf{D}_{x,\imh,j}  \jump{\hat{\evar}}^k_{x,\imh,j} \right]\\
			&\quad\;-\frac{1}{2 \Dy} \left[ \jump{\evar}_{i,\jph} \textbf{D}_{y,i,\jph}  \jump{\hat{\evar}}^k_{y,i,\jph} + \jump{\evar}_{i,\jmh} \textbf{D}_{y,i,\jmh}  \jump{\hat{\evar}}^k_{y,i,\jmh} \right]\\
			&=-\frac{1}{2 \Dx} \left[ \jump{\W}_{\iph,j} {\Lambda}_{x,\iph,j}  \jump{\hat{\W}}^k_{x,\iph,j} + \jump{\W}_{\imh,j} {\Lambda}_{x,\imh,j}  \jump{\hat{\W}}^k_{x,\imh,j} \right]\\
			&\quad\;-\frac{1}{2 \Dy} \left[ \jump{\W}_{i,\jph} \Lambda_{y,i,\jph}  \jump{\hat{\W}}^k_{y,i,\jph} + \jump{\W}_{i,\jmh} \Lambda_{y,i,\jmh}  \jump{\hat{\W}}^k_{y,i,\jmh} \right]\\
			&\le 0 \qquad \qquad \text{					Using sign preserving property \eqref{eq:sign_preser}}.
		\end{align*}
	\end{proof}
	
	To compute the solutions close to isotropic (MHD) limits, in some test cases, we consider an additional source term in the equations \eqref{eq:cgl_ref_godunov} by adding,
	\begin{equation}
		\label{eq:src_mhd}
		\s=\Biggl\{0,~0,~0,~0,~\frac{\per - \pll}{\tau},~0,~0,~0,~0\Biggr\}^\top,
	\end{equation}
	with $\tau=10^{-5}$. This effects only $p_{\parallel}$ equation \eqref{eq:cgl_con_3}. In that case, \eqref{eq:semi-discrete_final_scheme} results in,
	\begin{align}
		\label{eq:semi-discrete_final_scheme_source}
		\frac{d}{d t}\con_{i,j}&+\frac{\hat{\F}^k_{x, i+\frac{1}{2}, j}-\hat{\F}^k_{x, i-\frac{1}{2}, j}}{\Delta x}+\frac{\hat{\F}^k_{y, i, j+\frac{1}{2}}-\hat{\F}^k_{y, i, j-\frac{1}{2}}}{\Delta y}\nonumber&\\&+\phi'(\evar_{i, j})^\top \bigg(\left(\frac{\p B_x}{\p x}\right)_{i,j} + \left(\frac{\p B_y}{\p y}\right)_{i,j}\bigg)\nonumber \\
		&+\bc_x(\con_{i,j})\left(\frac{\p \con}{\p x}\right)_{i,j} +\bc_y(\con_{i,j}) \left(\frac{\p \con}{\p y}\right)_{i,j}=\s(\con_{i,j}).
	\end{align}
	In that case \eqref{eq:semi-disc_ent_stab}, results in,
	\begin{align}
		\frac{d}{dt}  \ent(\con_{i,j})  &+\frac{1}{\Delta x} \left( \hat{\entf}^k_{x,i+\frac{1}{2},j} - \hat{\entf}^k_{x,i-\frac{1}{2},j}\right)+\frac{1}{\Delta y}\left( \hat{\entf}^k_{y,i,j+\frac{1}{2}} - \hat{\entf}^k_{y,i,j-\frac{1}{2}}\right)\nonumber \\
		& \le \evar_{i,j}^\top\cdot \s(\con_{i,j}) = -\frac{\rho_{i,j}}{p_{\parallel, i,j} p_{\perp,i,j} \tau}(p_{\parallel,i,j} - p_{\perp,i,j})^{2} \le 0.
		\label{eq:semi-disc_ent_stab_source}
	\end{align}
	\section{Fully discrete schemes}
	\label{sec:time_disc}
	The semi-discrete scheme \eqref{eq:semi-discrete_final_scheme_source}, can be written as 
	\begin{equation}
		\frac{d }{dt}\con_{i,j}(t) = \mathcal{L}_{i,j}(\con(t)) +\s(\con_{i,j}(t))
		\label{fullydiscrete}
	\end{equation}
	where,
	\begin{align*}
		\mathcal{L}_{i,j}(\con(t))
		&=-\frac{\hat{\F}^k_{x, i+\frac{1}{2}, j}-\hat{\F}^k_{x, i-\frac{1}{2}, j}}{\Delta x}-\frac{\hat{\F}^k_{y, i, j+\frac{1}{2}}-\hat{\F}^k_{y, i, j-\frac{1}{2}}}{\Delta y}\nonumber&\\&-\phi'(\evar_{i, j})^\top \bigg(\left(\frac{\p B_x}{\p x}\right)_{i,j} + \left(\frac{\p B_y}{\p y}\right)_{i,j}\bigg) \\
		&-\bc_x(\con_{i,j})\left(\frac{\p \con}{\p x}\right)_{i,j} -\bc_y(\con_{i,j}) \left(\frac{\p \con}{\p y}\right)_{i,j}.
	\end{align*}
	To integrate the ODE~\eqref{fullydiscrete}, we use explicit and Implicit-Explicit (IMEX) schemes (see \cite{gottlieb2001strong,pareschi2005implicit,bhoriya2023high}).
	\subsection{Explicit schemes}
	Following \cite{gottlieb2001strong}, we use explicit strong stability preserving Runge Kutta (SSP-RK) methods for the time discretization. The second and third-order accurate SSP-RK schemes are as follows:
	\begin{enumerate}
		\item Set $\con^{0} \ = \ \con^n$.
		\item For $k$ in $\{1,\dots,m+1\}$, compute
		{\small
			\begin{eqnarray*}
				\con_{i,j}^{(k)} \
				= \
				\sum_{l=0}^{k-1}\gamma_{kl}\con_{i,j}^{(l)}
				+
				\delta_{kl}\Delta t \Big(\mathcal{L}_{i,j}(\con^{(l)}) + \s(\con_{i,j}^{(l)}) \Big),
		\end{eqnarray*}}
		where $\gamma_{kl}$ and $\delta_{kl}$ are given in the Table~\eqref{table:ssp}.
		\item Finally, $\con_{i,j}^{n+1} \ =  \ \con_{i,j}^{(m+1)}$.
	\end{enumerate}
	\begin{table}[h]
		\caption{Coefficients for Explicit SSP-Runge-Kutta time stepping}
		\centering
		\begin{tabular}{l|ccc|ccc}
			\hline
			Order & \hspace{2.0cm}$\gamma_{il}$ & & & \hspace{2.0cm}$\delta_{il}$ & & \\
			\hline
			2 & 1 & & & 1 & & \\
			& 1/2 & 1/2 & & 0 & 1/2 & \\
			\hline
			3 & 1 & & & 1 & & \\
			& 3/4 & 1/4 & & 0 & 1/4 & \\
			& 1/3 & 0 & 2/3 & 0 & 0 & 2/3 \\
			\hline
		\end{tabular}
		\label{table:ssp}
	\end{table}
	The fourth order SSP-RK scheme \cite{gottlieb2001strong} is given by, 
	\begin{subequations}
		\begin{align*}
			\con^{(1)}_{i,j} &= \textbf{U}^n_{i,j} + 0.39175222700392 \Delta t \big(\mathcal{M}_{i,j}(\con^n) \big)\\
			\con^{(2)}_{i,j} &= 0.44437049406734 \con^n_{i,j} + 0.55562950593266 \con^{(1)}_{i,j} \\
			&+0.36841059262959  \Delta t \big( \mathcal{M}_{i,j}(\con^{(1)}) \big)\\
			\con^{(3)}_{i,j} &= 0.62010185138540 \con^n _{i,j}+ 0.37989814861460 \con^{(2)}_{i,j} \\
			&+0.25189177424738  \Delta t \big( \mathcal{M}_{i,j}(\con^{(2)}) \big)\\
			\con^{(4)}_{i,j} &= 0.17807995410773 \con^n_{i,j} + 0.82192004589227 \con^{(3)}_{i,j} \\
			&+ 0.54497475021237  \Delta t \big( \mathcal{M}_{i,j}(\con^{(3)}) \big)\\
			\con^{n+1}_{i,j} &= 0.00683325884039 \con^n_{i,j} + 0.51723167208978 \con^{(2)}_{i,j} + 0.12759831133288 \con^{(3)}_{i,j}\\
			&+ 0.34833675773694 \con^{(4)}_{i,j}+ 0.08460416338212  \Delta t \big( \mathcal{M}_{i,j}(\con^{(3)}) \big)\\
			&+ 0.22600748319395  \Delta t \big( \mathcal{M}_{i,j}(\con^{(4)}) \big).
		\end{align*}	
	\end{subequations}
	where $\mathcal{M}_{i,j}(\con^n)=\mathcal{L}_{i,j}(\con^n) + \s(\con^n_{i,j})$. 
	
	We will use explicit schemes to compute CGL solutions only, i.e. $\s$ in \eqref{eq:src_mhd} is ignored. We are  now denoting explicit schemes without source term $\s$, as follows:
	\begin{enumerate}
		\item \ote: We consider second-order time stepping with MinMod reconstruction for scaled entropy variables. We use the second-order central difference for non-conservative terms. The source term is ignored.
		
		\item \othe: We consider third-order time stepping with third-order ENO reconstruction for scaled entropy variables. We use the fourth-order central difference for non-conservative terms. The source term is ignored.
		
		\item \ofe: We consider fourth-order time stepping with fourth-order ENO reconstruction for scaled entropy variables. We use the fourth-order central difference for non-conservative terms. The source term is ignored.
	\end{enumerate}
	\subsection{ARK IMEX schemes} \label{sec:imex_scheme}
	The source term $\s$ contains the parameter $\tau$, which makes the source term stiff. To avoid the small time step restriction required for stability of the explicit scheme, treat the source term implicitly. In particular, we use the Additive Runge-Kutta (ARK) IMEX scheme. For the second-order discretization, we use the L-Stable second-order accurate ARK-IMEX scheme from \cite{pareschi2005implicit,bhoriya2023high}. The method has two implicit stages and is given by
	\begin{subequations}
		\begin{align*}
			\mathbf{U}_{i,j}^{(1)} &= \mathbf{U}^n_{i,j} + \Delta t \ \big( \beta \s(\mathbf{U}^{(1)}_{i,j} \big), \\
			\mathbf{U}^{(2)}_{i,j} &= \mathbf{U}^n_{i,j} + \Delta t \ \big( \mathcal{L}_{i,j}(\mathbf{U}^{(1)}) + (1-2\beta) \s(\mathbf{U}^{(1)}_{i,j})  + \beta \s(\mathbf{U}^{(2)}_{i,j}) \big), \\
			\mathbf{U}_{i,j}^{(n+1)} &= \mathbf{U}^n_{i,j} + \Delta t \ \bigg( \dfrac{1}{2}\mathcal{L}_{i,j}(\mathbf{U}^{(1)}) + \dfrac{1}{2}\mathcal{L}_{i,j}(\mathbf{U}^{(2)}) + \dfrac{1}{2} \s(\mathbf{U}^{(1)}_{i,j})
			+ \dfrac{1}{2} \s(\mathbf{U}^{(2)}_{i,j}) \bigg).
		\end{align*}
	\end{subequations}
	where $\beta = 1 - \dfrac{1}{\sqrt{2}}$.\\	
	For the third and fourth-order ARK scheme, we follow  \cite{kennedy2003additive}. In the ARK-IMEX schemes, we need to solve a system of nonlinear equations to solve the Implicit source part. We use the Newton solver, where the convergence and the next direction search are based on the back-tracing line search framework. A detailed procedure is given in \cite{dennis1996numerical}.
	
	We will use IMEX schemes to compute CGL equations with $\s$ in \eqref{eq:src_mhd}. We are denoting these IMEX schemes, with source terms $\s$, as follows:
	\begin{enumerate}
		\item \oti: We consider second-order time stepping with MinMod reconstruction for scaled entropy variables. We use the second-order central difference for non-conservative terms. Source term $\s$ is treated implicitly. 
		
		\item \othi: We consider third-order time stepping with third-order ENO reconstruction for scaled entropy variables. We use the fourth-order central difference for non-conservative terms. Source term $\s$ is treated implicitly. 
		
		\item \ofi: We consider fourth-order time stepping with fourth-order ENO reconstruction for scaled entropy variables. We use the fourth-order central difference for non-conservative terms.  Source term $\s$ is treated implicitly. 
	\end{enumerate}

	\section{Numerical results}
\label{sec:num_res}
In this Section, we present several test cases in one dimension. We also present a two-dimensional Orzag-Tang vertex test case, which is generalized from the MHD test case. The time step is computed using 
$$
\Delta t = \text{CFL} \frac{1}{\max_{i,j} \left( \frac{|\lambda_x(\con_{i,j})|}{\Delta x} + \frac{|\lambda_x(\con_{i,j})|}{\Delta y} \right)},
$$
where, $\lambda_x$ is the maximum eigenvalue in $x$-direction and $\lambda_y$ is the maximum eigenvalue in $y$-direction. To measure the entropy decay for each test case, we also measure 
\begin{align}
	\label{eq:tot_ent_exp}
	\sum_{i,j} \Big[\ent^{n+1}(\con_{i,j}) - \ent^{n}(\con_{i,j})-\frac{\Delta t}{\Delta x} \left( \hat{\entf}^k_{x,i+\frac{1}{2},j} - \hat{\entf}^k_{x,i-\frac{1}{2},j}\right)\nonumber\\
	-\frac{\Delta t}{\Delta y}\left( \hat{\entf}^k_{y,i,j+\frac{1}{2}} - \hat{\entf}^k_{y,i,j-\frac{1}{2}}\right)\Big],
\end{align}
for explicit schemes (see Equation \eqref{eq:semi-disc_ent_stab}). This measures the total entropy change after every time step. Similarly, for the implicit scheme, we measure 
\begin{align}
	\sum_{i,j}\Big[ \ent^{n+1}(\con_{i,j}) - \ent^{n}(\con_{i,j})&-\frac{\Delta t}{\Delta x} \left( \hat{\entf}^k_{x,i+\frac{1}{2},j} - \hat{\entf}^k_{x,i-\frac{1}{2},j}\right) \nonumber\\
	&-\frac{\Delta t}{\Delta y}\left( \hat{\entf}^k_{y,i,j+\frac{1}{2}} - \hat{\entf}^k_{y,i,j-\frac{1}{2}}\right)-\Delta t \evar_{i,j}\s(\con_{i,j})\Big].
	\label{eq:tot_ent_imp}
\end{align}	
which is consistent with Eqn. \eqref{eq:semi-disc_ent_stab_source}, which also takes source term into consideration. When computing with the source term, we expect pressure to be isotropic (i.e. $\pll$ and $\per$ should be equal); hence, the solution should be closer to the MHD limit. To compare with the MHD solution, we have plotted the corresponding MHD solution, which is computed using Rusanov solver, with MinMod reconstruction. In one-dimensional test cases, we have used $10000$ cells. For the Orszag-Tang vertex two-dimensional test case, we have used $1000\times1000$ cells to compute the reference MHD solution.
\subsection{One-dimensional test problems} 
\subsubsection{Accuracy test} 
\label{test:at}
To assess the accuracy of the proposed schemes, we consider a problem with a smooth solution. The computational domain is taken to be $[0,1]$ with periodic boundary conditions. Initially, the density profile is assumed to be $\rho(x,0) = 2+ \sin{(2\pi x)}$ which is then advected with velocity $\bu = (1,0,0)$. The pressures are kept constant at $\pll = \per = 1.0$. We also take $\B = (1,1,0)$. The exact solution is an advection of the density profile given by $\rho(x,t) = 2+ \sin{(2\pi (x-t))}$. 
\begin{table}[tbhp]
	\footnotesize
	\caption{\textbf{\nameref{test:at}}: $L^1$ errors and order of accuracy for $\rho$.}
	\label{tab:2}
	\begin{center}
		\begin{tabular}{c|c|c|c|c|c|c|}
			\hline Number of cells  & \multicolumn{2}{|c}{\ote} &  \multicolumn{2}{|c}{\othe} & \multicolumn{2}{|c}{\ofe}  \\
			\hline   & $L^1$ error  &  Order &  $L^1$ error      & Order & $L^1$ error      & Order \\
			\hline 40 & 1.30241e-01 & -- & 5.53282e-03 & -- & 8.98990e-04 & -- \\
			80 & 5.17593e-02 & 1.33 & 7.07952e-04 & 2.96629 & 8.19842e-05 & 3.45489 \\
			160 & 1.65341e-02 & 1.65 & 8.84319e-05 & 3.00101 & 5.86817e-06 & 3.80436 \\
			320 & 4.57448e-03 & 1.86 & 1.10535e-05 & 3.00007 & 4.11943e-07 & 3.83239 \\
			640 & 1.24571e-03 & 1.88 & 1.38170e-06 & 2.99998 & 2.84900e-08 & 3.85391 \\
			1280 & 3.35275e-04 & 1.89 & 1.72720e-07 & 2.99993 & 1.92548e-09 & 3.88717 \\
			\hline
		\end{tabular}
	\end{center}
\end{table} 
\begin{table}[tbhp]
	\footnotesize
	\caption{\textbf{\nameref{test:at}}: $L^1$ errors and order of accuracy for $\rho$.}
	\label{tab:3}
	\begin{center}
		\begin{tabular}{c|c|c|c|c|c|c|}
			\hline Number of cells  & \multicolumn{2}{|c}{\oti} &  \multicolumn{2}{|c}{\othi} & \multicolumn{2}{|c}{\ofi}  \\
			\hline   & $L^1$ error  &  Order &  $L^1$ error      & Order & $L^1$ error      & Order \\
			\hline 40 & 1.30187e-01 & -- & 5.57809e-03 & -- & 8.90421e-04 & -- \\
			80 & 5.16249e-02 & 1.33 & 7.19948e-04 & 2.95381 & 8.24102e-05 & 3.43359 \\
			160 & 1.65536e-02 & 1.64 & 9.01302e-05 & 2.99781 & 5.96265e-06 & 3.7888 \\
			320 & 4.58388e-03 & 1.85 & 1.12664e-05 & 2.99999 & 4.19101e-07 & 3.83059 \\
			640 & 1.24898e-03 & 1.88 & 1.40835e-06 & 2.99994 & 2.90423e-08 & 3.85107 \\
			1280 & 3.36345e-04 & 1.89 & 1.76049e-07 & 2.99996 & 1.95243e-09 & 3.89482 \\
			\hline
		\end{tabular}
	\end{center}
\end{table} 
The simulations are conducted till the final time $t = 2.0$. Table \eqref{tab:2} presents the $L^1$-errors of the density for the \ote, \othe~and \ofe~schemes. We observe that all the schemes have the expected order of accuracy. Similarly, in Table \eqref{tab:3}, we present the $L^1$-errors of the density for the \oti, \othi~and \ofi~schemes. Again, we see that all the schemes have a theoretical order of accuracy.
\subsubsection{Riemann problem 1: Brio-Wu shock tube problem}
\label{test:bw}
This test case is a generalized form of the MHD test case proposed by Brio and Wu \cite{Brio1988}. We consider the computational domain of $[-1,1]$ with outflow boundary conditions. The initial discontinuity is centred at $x=0$, with initial states given by,
\[(\rho, \bu, \pll, \per, B_{y}, B_{z}) = \begin{cases}
	(1, 0, 0, 0, 1, 1, 1, 0), & \textrm{if } x\leq 0\\
	(0.125, 0, 0, 0, 0.1, 0.1, -1, 0), & \textrm{otherwise}
\end{cases}\]
We consider $B_{x}=0.75$ and perform the computations till final time $t=0.2$ on $2000$ cells.
\begin{figure}[!htbp]
	\begin{center}
		\subfloat[Anisotropic case: Density for explicit schemes without source term]{\includegraphics[width=1.5in, height=1.5in]{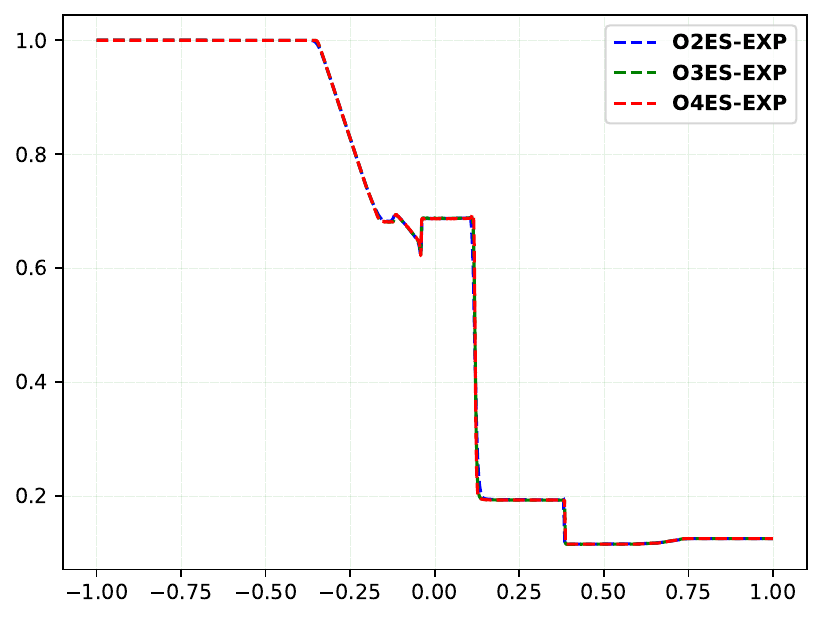} \label{fig:bw_exp_rho}}~
		\subfloat[Anisotropic case: Pressure component $\pll$ for explicit schemes without source term]{\includegraphics[width=1.5in, height=1.5in]{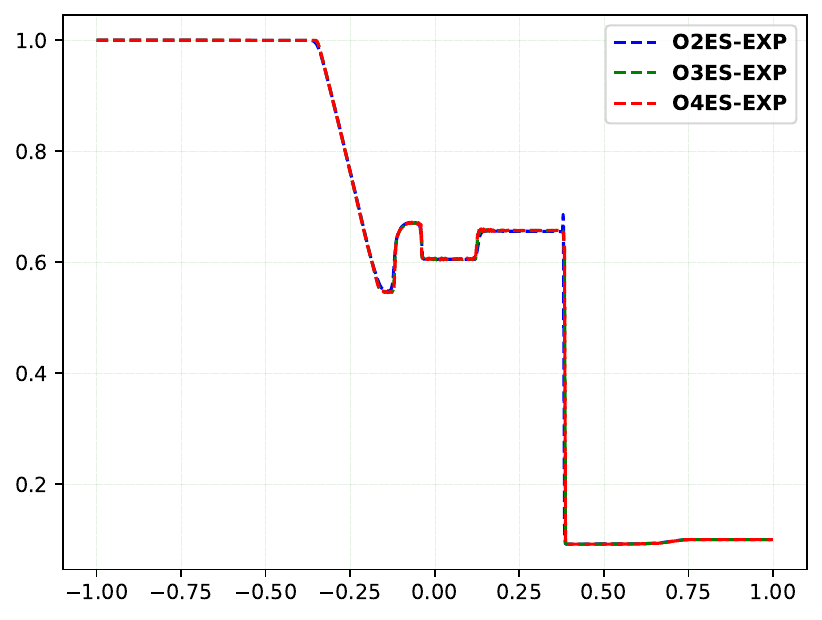}\label{fig:bw_exp_pll}}~
		\subfloat[Anisotropic case: Pressure component $\per$ for explicit schemes without source term]{\includegraphics[width=1.5in, height=1.5in]{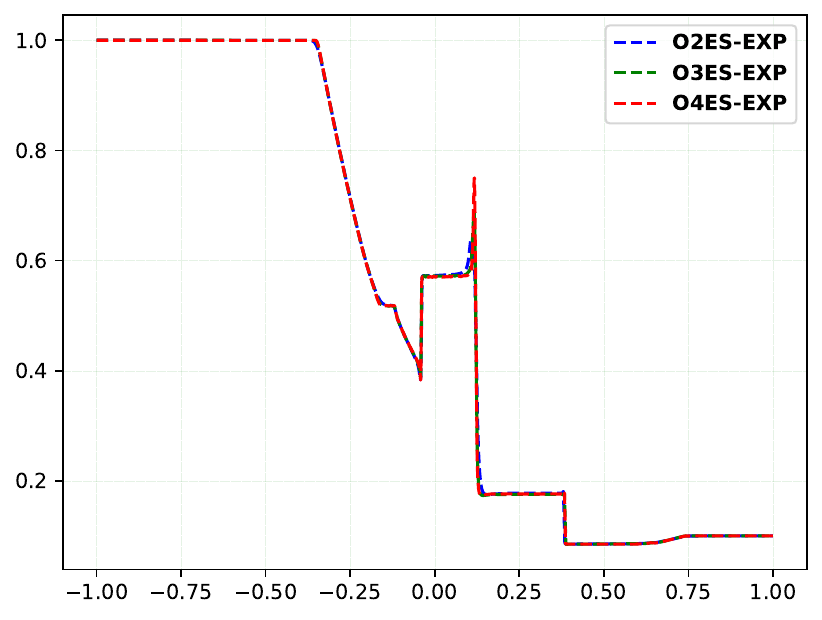}\label{fig:bw_exp_perp}}\\
		\subfloat[Isotropic case: Density for IMEX schemes with source term]{\includegraphics[width=1.5in, height=1.5in]{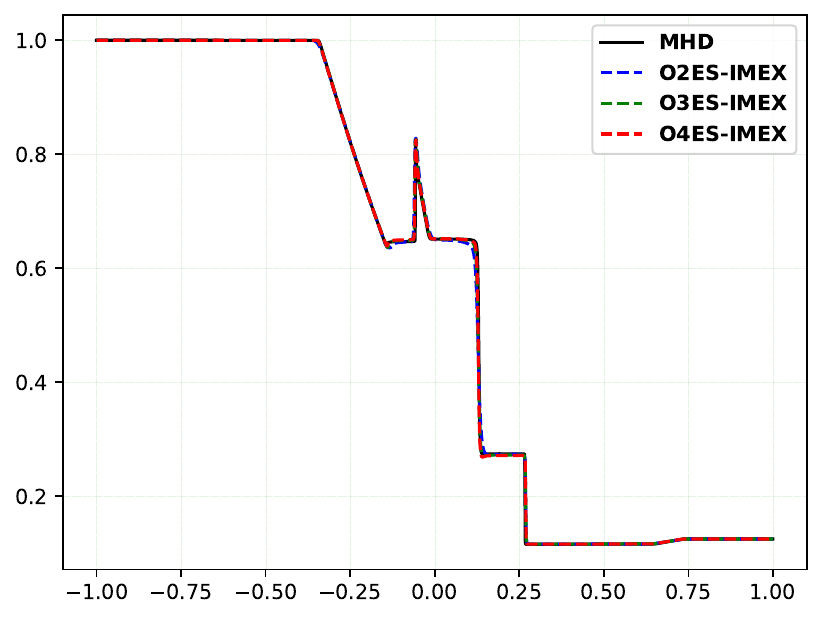}\label{fig:bw_imp_rho}}~
		\subfloat[Isotropic case: Pressure component $\pll$ for IMEX schemes with source term]{\includegraphics[width=1.5in, height=1.5in]{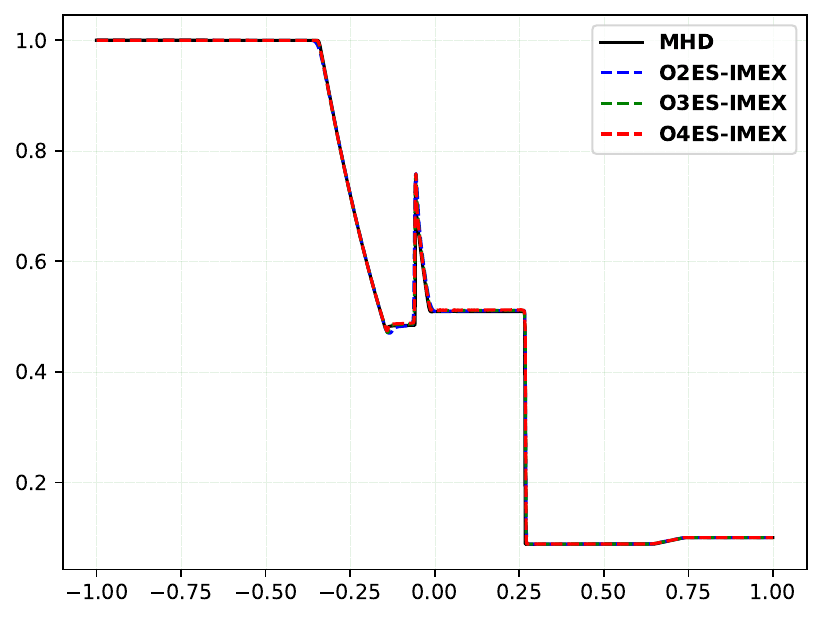}\label{fig:bw_imp_pll}}~
		\subfloat[Isotropic case: Pressure component $\per$ for IMEX schemes with source term]{\includegraphics[width=1.5in, height=1.5in]{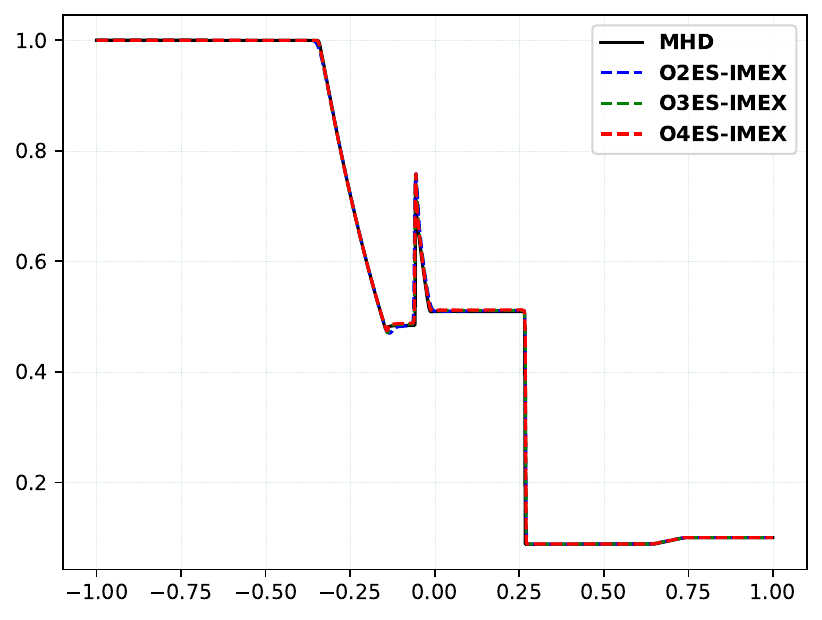}\label{fig:bw_imp_perp}}\\
		\subfloat[Anisotropic case: Total entropy decay at each time step for explicit schemes]{\includegraphics[width=2.0in, height=1.5in]{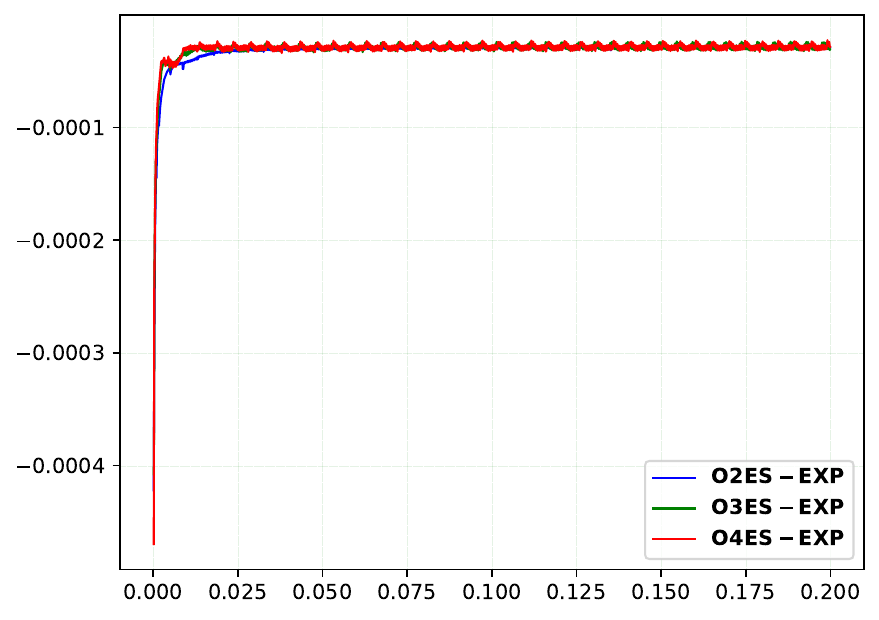}\label{fig:bw_exp_ent}}~
		\subfloat[Isotropic case: Total entropy decay at each time step for IMEX schemes with source term]{\includegraphics[width=2.0in, height=1.5in]{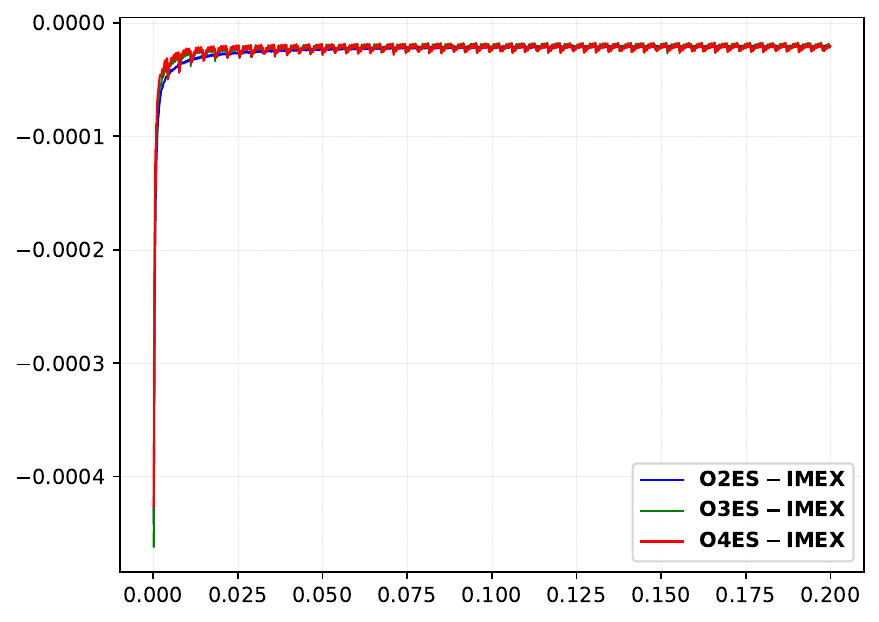}\label{fig:bw_imp_ent}}
		\caption{\textbf{\nameref{test:bw}}: Plots of density, parallel and perpendicular pressure components, and total entropy decay at each time step for explicit schemes without source term and IMEX scheme with source term using $2000$ cells at final time $t = 0.2$.}
		\label{fig:bw}
	\end{center}    
\end{figure}
Numerical results are presented in Figure \eqref{fig:bw}. In Figures \eqref{fig:bw_exp_rho}, \eqref{fig:bw_exp_pll} and \eqref{fig:bw_exp_perp}, we have presented the density, parallel component of pressure, and perpendicular component of the pressure using explicit schemes. We observe all the schemes have resolved all the waves. We also note that both components of the pressures are different, as expected, because we are computing anisotropic solutions. The results are similar to the anisotropic solution given in \cite{Hirabayashi2016new}. We also note that \ote~is most diffusive compared to \othe~and \ofe.

The isotropic solutions (with source terms) with IMEX schemes are presented in Figures \eqref{fig:bw_imp_rho}, \eqref{fig:bw_imp_pll} and \eqref{fig:bw_imp_perp}. We observe that the profile of $\pll$ and $\per$ are the same due to the effect of the source term. Furthermore, the solution for all the components matches the reference MHD solution.  This is similar to the isotropic solution presented in \cite{Hirabayashi2016new}. For the IMEX schemes, we again observe that the \oti~is the most diffusive, with \othi~and \ofi~being comparable. 

In Figures \eqref{fig:bw_exp_ent} and \eqref{fig:bw_imp_ent}, we have plotted the total entropy change at each time step for explicit and IMEX schemes, respectively. For schemes to be entropy stable, we need these values to be less than zero. At the same time, we desire that the schemes do not decay too much entropy if no new discontinuity is generated in the solution. We see that all schemes decay lots of entropy initially, as we have several discontinuities being generated. After that, as no new discontinuity is generated, all the schemes decay much less entropy. Again, we see that both second-order schemes, \ote~and \oti~are most diffusive as they decay most entropy, followed by third-order and fourth-order schemes. We also see that there is not much difference in the entropy decay of explicit and IMEX schemes of the same order.
\subsubsection{Riemann problem 2: Ryu-Jones problem}
\label{test:ryu}
In this test case, we consider a generalization of the Ryu–Jones test problem \cite{Ryu1995Numerical}. The computational domain is taken to be $[-0.5,0.5]$, with outflow boundary conditions. The initial conditions are given by,
\[(\rho, \bu, \pll, \per, B_{y}, B_{z}) = \begin{cases}
	(1.08, 1.2, 0, 0, 0.95, 0.95, \frac{3.6}{\sqrt{4\pi}}, \frac{2}{\sqrt{4\pi}}), & \textrm{if } x\leq 0\\
	(1, 0, 0, 0, 1, 1, \frac{4}{\sqrt{4\pi}}, \frac{2}{\sqrt{4\pi}}), & \textrm{otherwise}
\end{cases}\]
with $B_{x}=\frac{2}{\sqrt{4\pi}}$. We compute the solutions using $2000$ cells till the final time $t=0.2$. 
\begin{figure}[!htbp]
	\begin{center}
		\subfloat[Anisotropic case: Density for explicit schemes without source term]{\includegraphics[width=1.5in, height=1.5in]{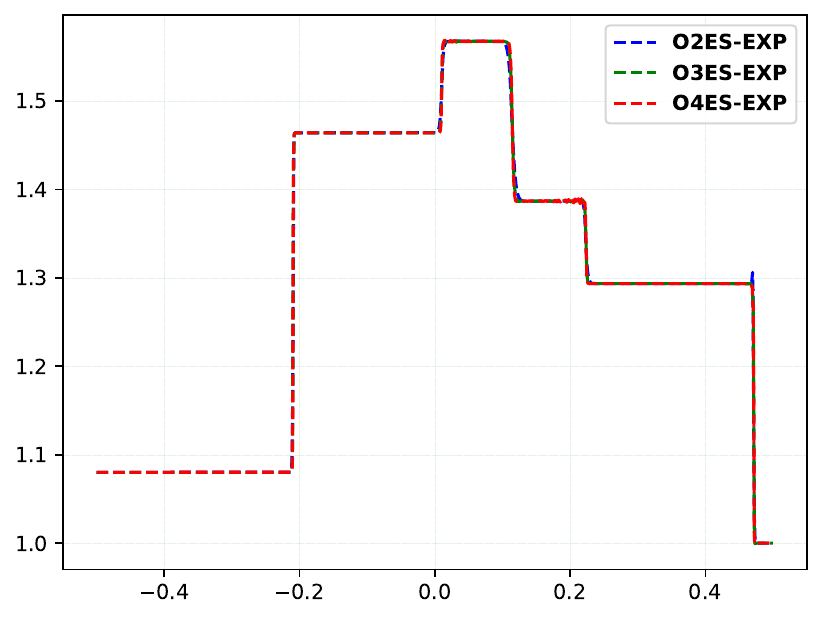} \label{fig:ryu_exp_rho}}~
		\subfloat[Anisotropic case: Pressure component $\pll$ for explicit schemes without source term]{\includegraphics[width=1.5in, height=1.5in]{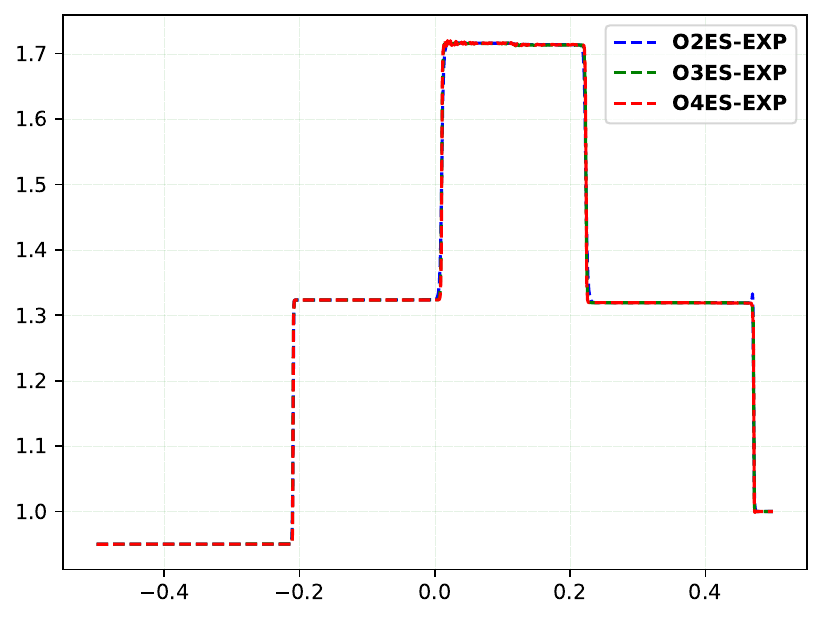}\label{fig:ryu_exp_pll}}~
		\subfloat[Anisotropic case: Pressure component $\per$ for explicit schemes without source term]{\includegraphics[width=1.5in, height=1.5in]{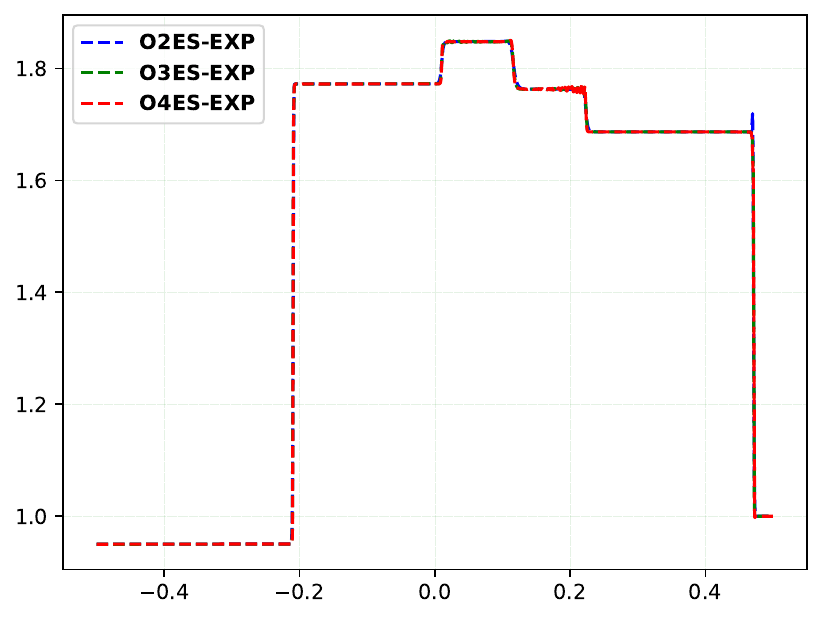}\label{fig:ryu_exp_perp}}\\
		\subfloat[Isotropic case: Density for IMEX schemes with source term]{\includegraphics[width=1.5in, height=1.5in]{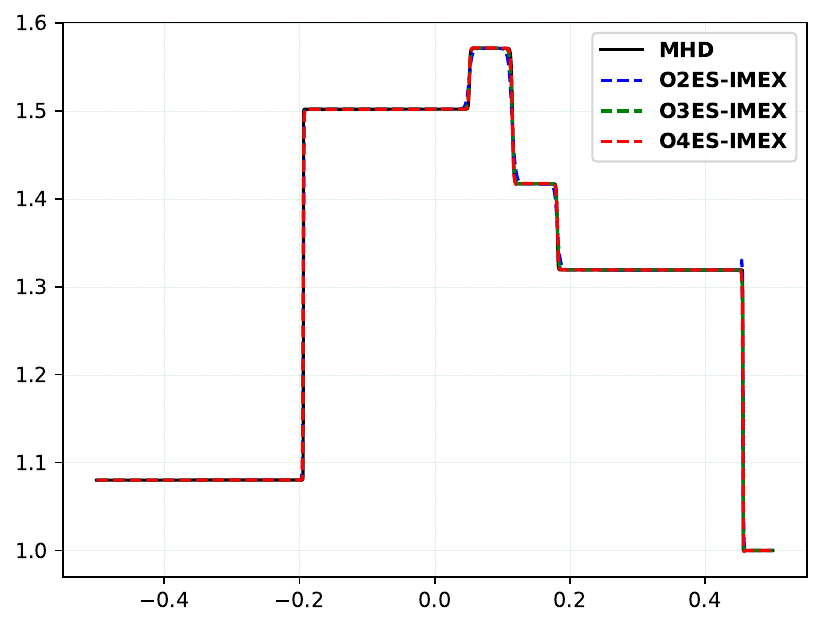}\label{fig:ryu_imp_rho}}~
		\subfloat[Isotropic case: Pressure component $\pll$ for IMEX schemes with source term]{\includegraphics[width=1.5in, height=1.5in]{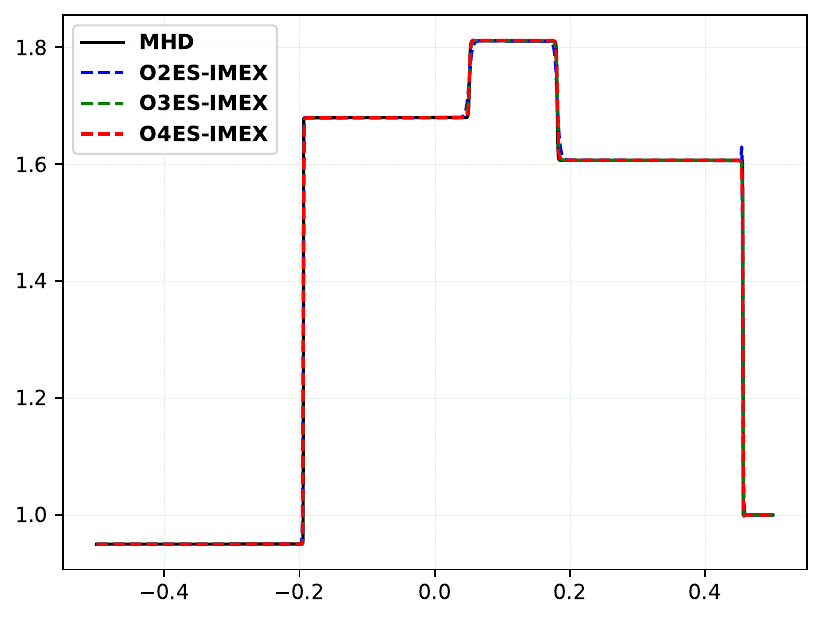}\label{fig:ryu_imp_pll}}~
		\subfloat[Isotropic case: Pressure component $\per$ for IMEX schemes with source term]{\includegraphics[width=1.5in, height=1.5in]{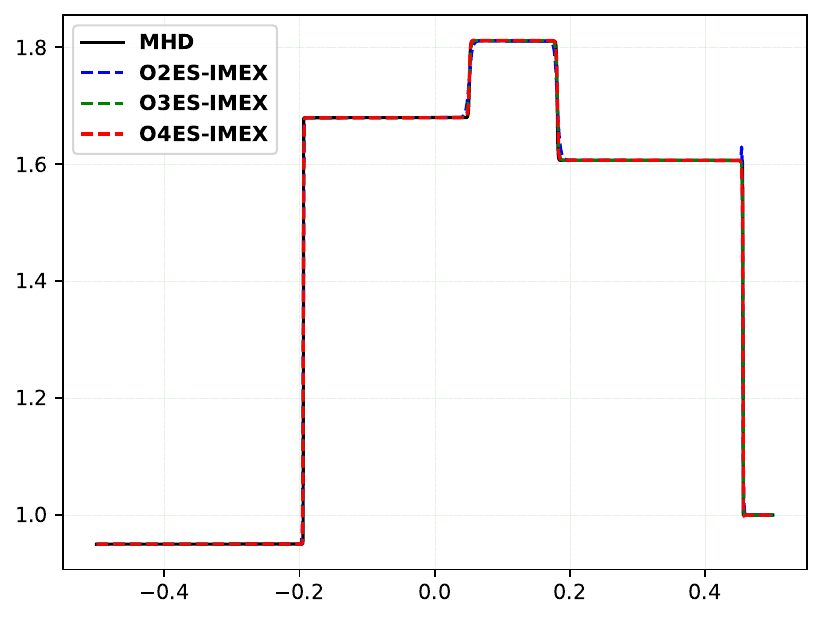}\label{fig:ryu_imp_perp}}\\
		\subfloat[Anisotropic case: Total entropy decay at each time step for explicit schemes]{\includegraphics[width=2.0in, height=1.5in]{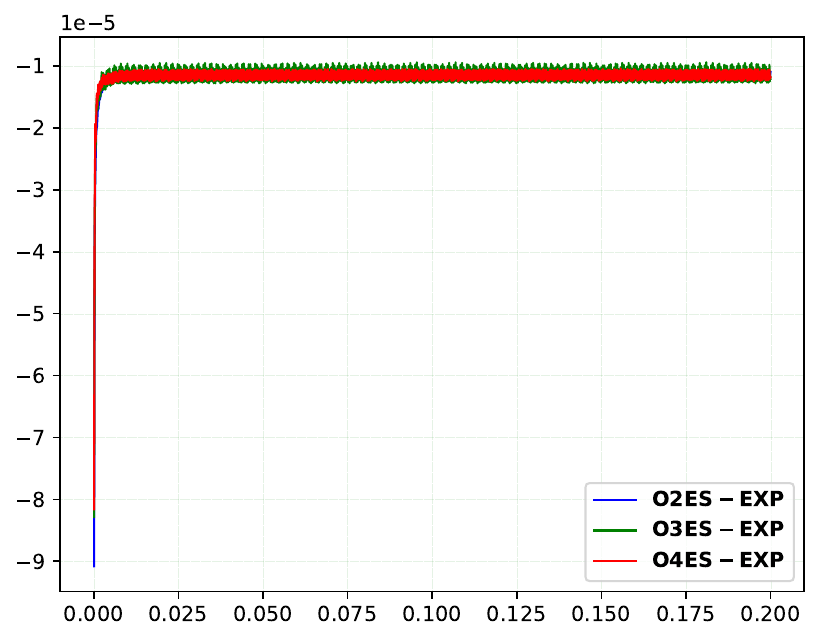}\label{fig:ryu_exp_ent}}~
		\subfloat[Isotropic case: Total entropy decay at each time step for IMEX schemes with source term]{\includegraphics[width=2.0in, height=1.5in]{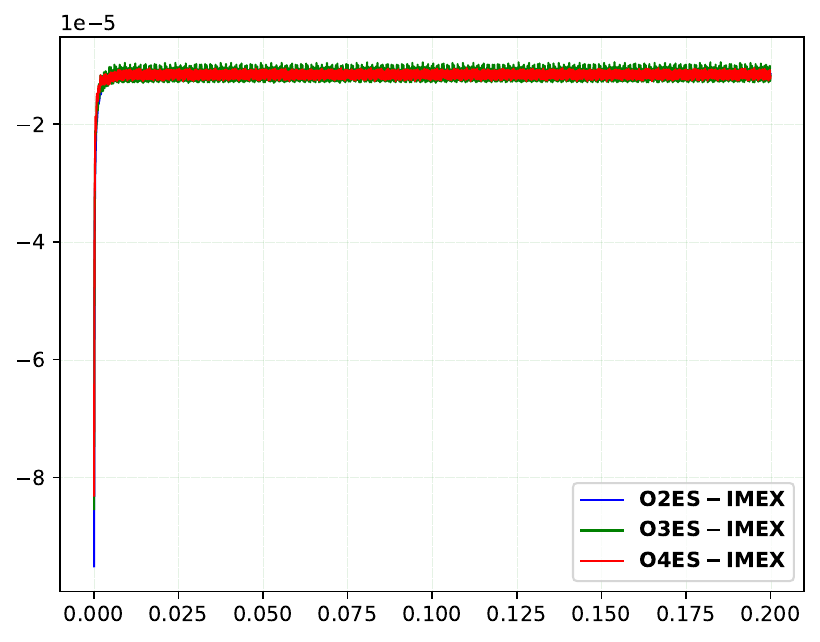}\label{fig:ryu_imp_ent}}
		\caption{\textbf{\nameref{test:ryu}}: Plots of density, parallel and perpendicular pressure components, and total entropy decay at each time step for explicit schemes without source term and IMEX scheme with source term using $2000$ cells at final time $t = 0.2$.}
		\label{fig:ryu}
	\end{center}
\end{figure}
Numerical results are presented in Figure \eqref{fig:ryu}. Anisotropic solutions (without source term) using explicit schemes are presented in Figures \eqref{fig:ryu_exp_rho}, \eqref{fig:ryu_exp_pll} and \eqref{fig:ryu_exp_perp}. We observe that the profiles of $\pll$ and $\per$ are different. We see that all the schemes are able to resolve different waves. We do see some small oscillations near some discontinuities. However, these oscillations are stable and do not increase when we refine the mesh. In Figure \eqref{fig:ryu_exp_ent}, we have plotted the total entropy decay for each scheme. We again see that \ote~is most diffusive when compared with \othe~and \ofe~schemes.

Isotropic solutions (with source term) using IMEX schemes are presented in \eqref{fig:ryu_imp_rho}, \eqref{fig:ryu_imp_pll} and \eqref{fig:ryu_imp_perp}. We see that all the variables have converged to the MHD solution. Furthermore, both components of pressure have the same profile, and the solution is isotropic. We again observe that \oti~is the most diffusive, followed by \othi~and \ofi~schemes. In Figure \eqref{fig:ryu_imp_ent}, we have plotted the total entropy decay at each time step. We see that the entropy decays are comparable to the explicit schemes. 
\subsubsection{Riemann problem 3: Super-fast expansion}
\label{test:sf}
In this test, we consider the generalization of the MHD super-fast expansion test case, considered in~\cite{bouchut2007multiwave,fuchs2009splitting,fuchs2011approximate}. The solution contains a low-density, low-pressure area in the middle as the fluid is pushed outward.  We consider the computational domain of $[0,1]$ with outflow boundary conditions. The initial conditions are given by,
\[(\rho, \bu, \pll, \per,B_x, B_{y}, B_{z}) = \begin{cases}
	(1, -3.1, 0, 0, 1, 1, 0, 0.5, 0), & \textrm{if } x\leq 0.5\\
	(1, 3.1, 0, 0, 1, 1, 0, 0.5, 0), & \textrm{otherwise.}
\end{cases}\]
\begin{figure}[!htbp]
	\begin{center}
		\subfloat[Anisotropic case: Density for explicit schemes without source term]{\includegraphics[width=1.5in, height=1.5in]{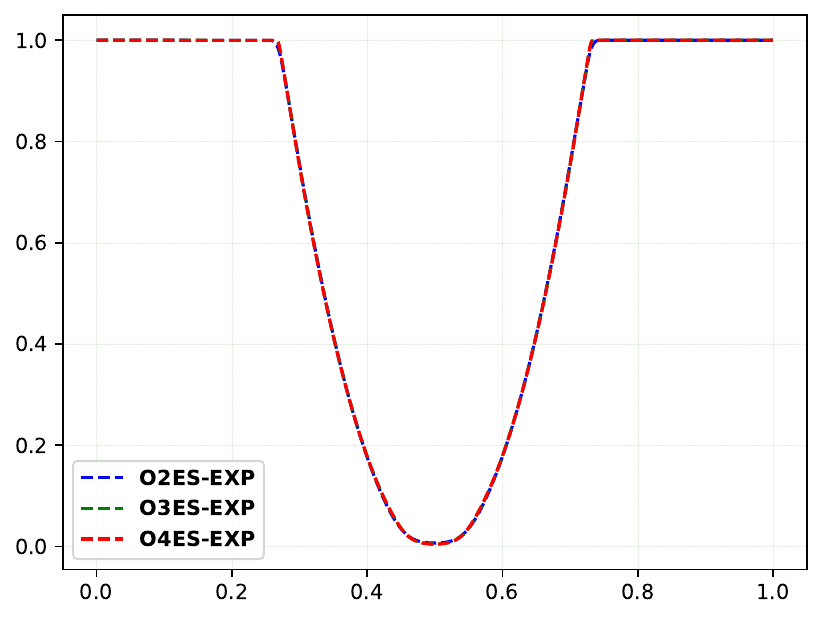} \label{fig:sf_exp_rho}}~
		\subfloat[Anisotropic case: Pressure component $\pll$ for explicit schemes without source term]{\includegraphics[width=1.5in, height=1.5in]{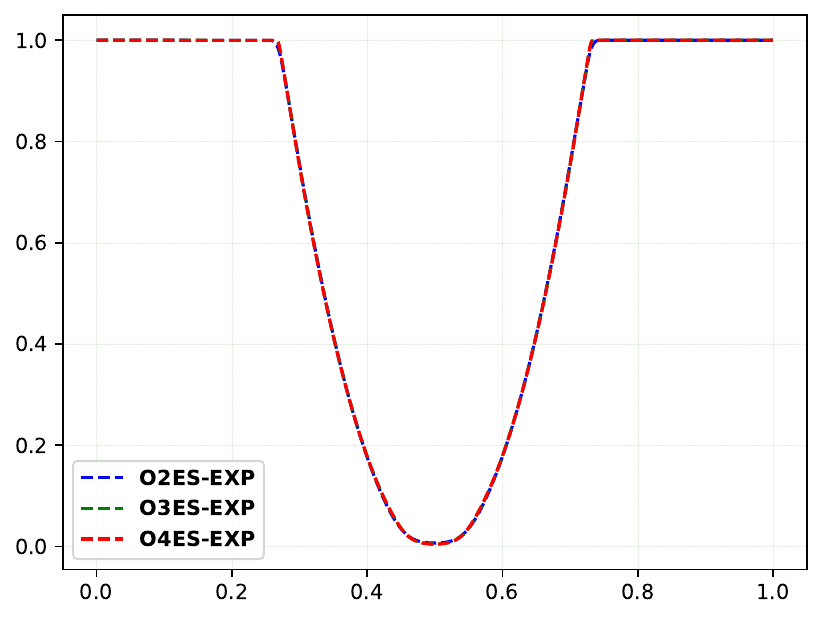}\label{fig:sf_exp_pll}}~
		\subfloat[Anisotropic case: Pressure component $\per$ for explicit schemes without source term]{\includegraphics[width=1.5in, height=1.5in]{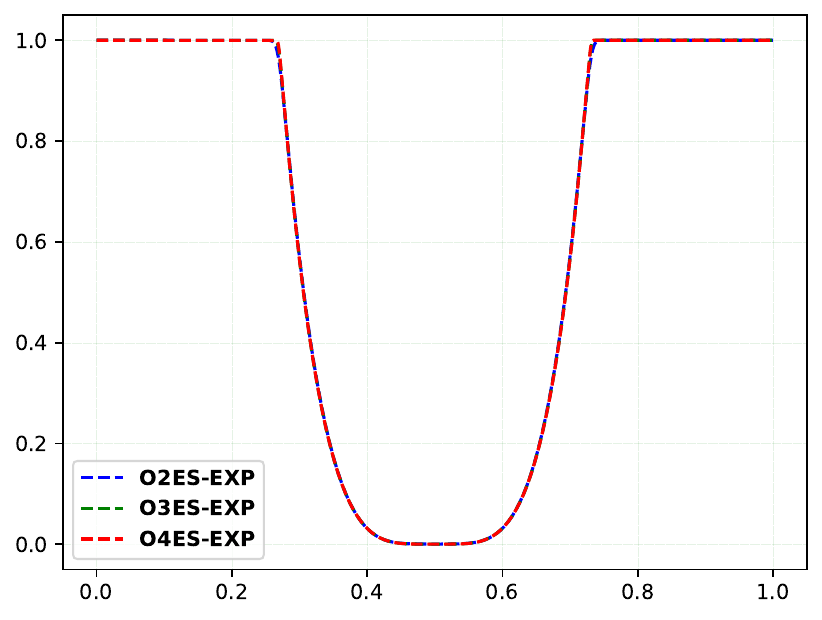}\label{fig:sf_exp_perp}}\\
		\subfloat[Isotropic case: Density for IMEX schemes with source term]{\includegraphics[width=1.5in, height=1.5in]{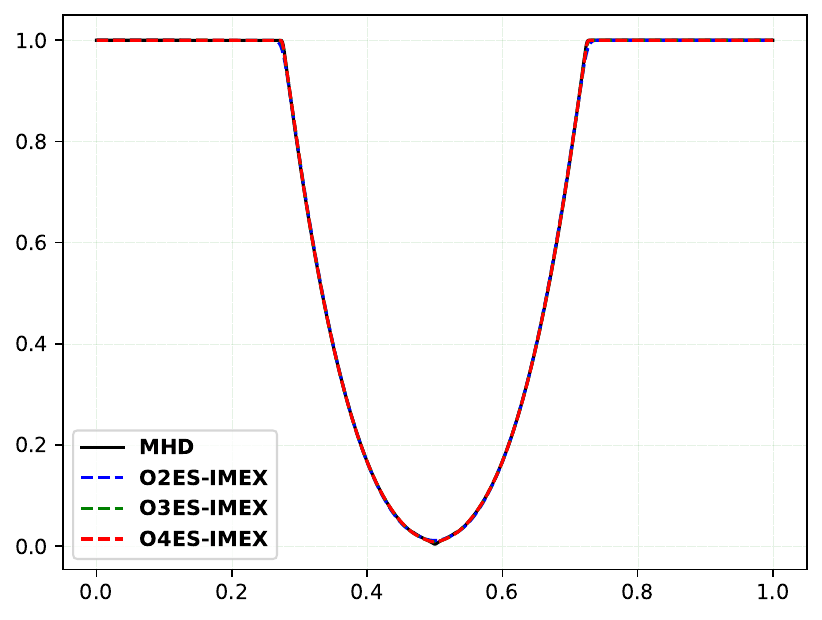}\label{fig:sf_imp_rho}}~
		\subfloat[Isotropic case: Pressure component $\pll$ for IMEX schemes with source term]{\includegraphics[width=1.5in, height=1.5in]{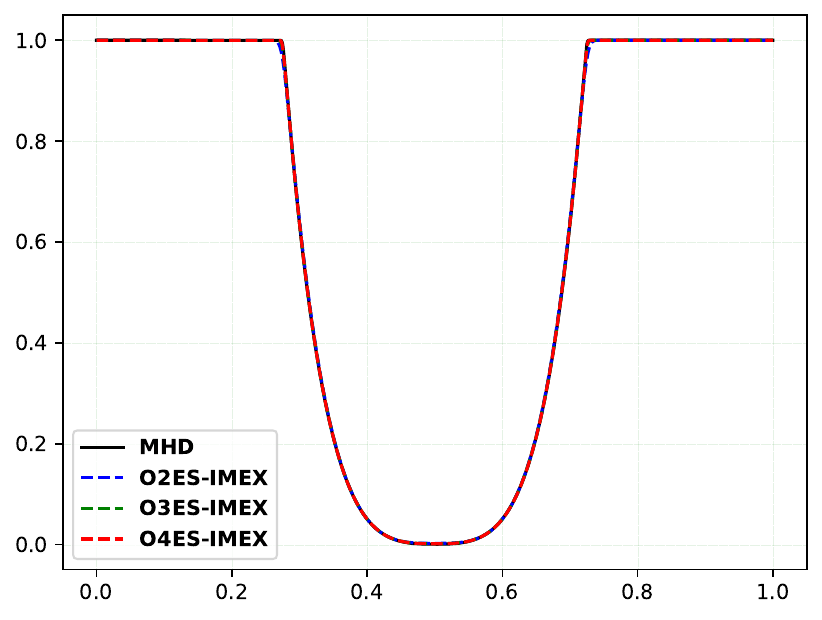}\label{fig:sf_imp_pll}}~
		\subfloat[Isotropic case: Pressure component $\per$ for IMEX schemes with source term]{\includegraphics[width=1.5in, height=1.5in]{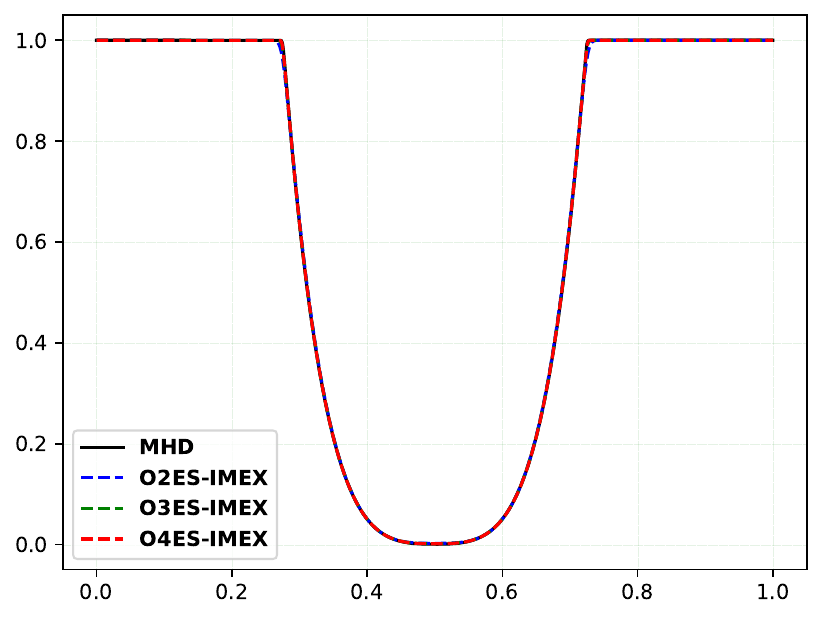}\label{fig:sf_imp_perp}}\\
		\subfloat[Anisotropic case: Total entropy decay at each time step for explicit schemes]{\includegraphics[width=2.0in, height=1.5in]{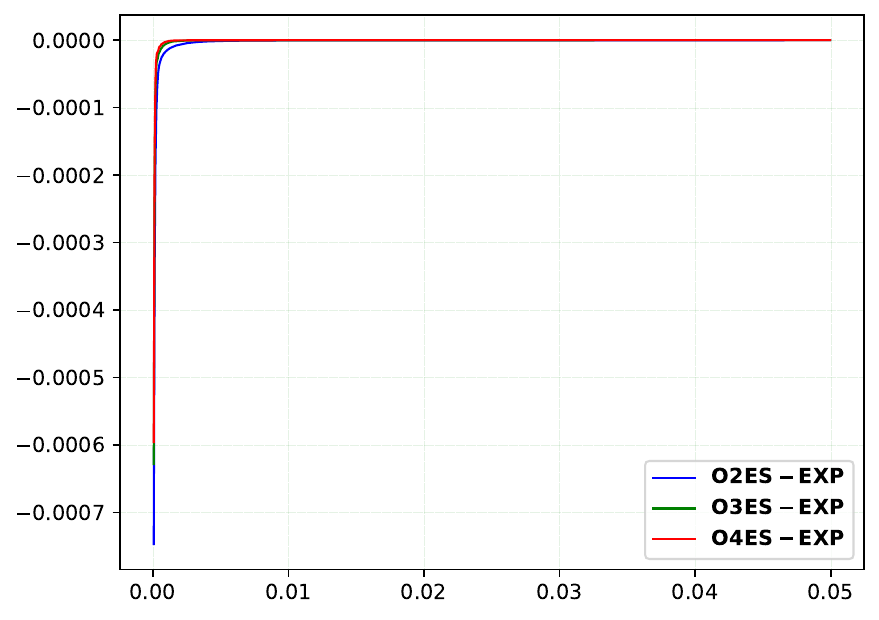}\label{fig:sf_exp_ent}}~
		\subfloat[Isotropic case: Total entropy decay at each time step for IMEX schemes with source term]{\includegraphics[width=2.0in, height=1.5in]{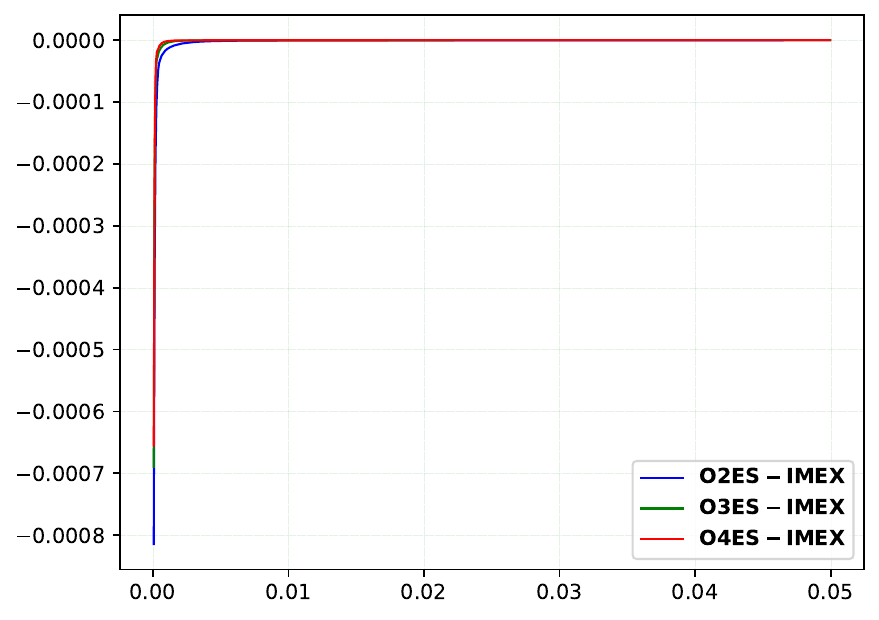}\label{fig:sf_imp_ent}}
		\caption{\textbf{\nameref{test:sf}}: Plots of density, parallel and perpendicular pressure components, and total entropy decay at each time step for explicit schemes without source term and IMEX scheme with source term using $2000$ cells at final time $t = 0.05$.}
		\label{fig:sf}
	\end{center}
\end{figure}
The numerical results are plotted in Figure \eqref{fig:sf} at the final time of $t=0.05$ using $2000$ cells. In Figures \eqref{fig:sf_exp_rho}, \eqref{fig:sf_exp_pll} and \eqref{fig:sf_exp_perp}, we have plotted density, $\pll$ and $\per$ for anisotropic case using explicit schemes. We again see the low-density, low-pressure area is resolved and both components of the pressure are not identical. Similarly, in Figures  \eqref{fig:sf_imp_rho}, \eqref{fig:sf_imp_pll} and \eqref{fig:sf_imp_perp}, we have plotted isotropic solutions using IMEX schemes. We observe that the profiles of pressure components are the same and match the MHD solution. Entropy decays for explicit and IMEX schemes are plotted in Figures \eqref{fig:sf_exp_ent} and \eqref{fig:sf_imp_ent}, respectively. We note that the second order schemes \ote~and \oti~are most diffusive compared to the higher order schemes \othe, \othi, \ofe~and \ofi. Furthermore, both explicit and IMEX schemes of corresponding order have similar entropy decay performance.
\subsubsection{Riemann problem 4}
\label{test:rp4}
In this test, we generalize a Riemann problem described in \cite{dumbser2016new}. The computational domain is $[-0.5,0.5]$, with outflow boundary conditions. The initial profile is given by,
\[(\rho, \bu, \pll, \per, B_{y}, B_{z}) = \begin{cases}
	(1, 0, 0, 0, 1, 1, 1, 0), & \textrm{if } x\leq 0\\
	(0.4, 0, 0, 0, 0.4, 0.4, -1, 0), & \textrm{otherwise}
\end{cases}\]
with $B_{x}=1.3$. We compute the solution using $2000$ cells till the final time $t=0.15$. The numerical results are presented in Figure \eqref{fig:rp4}. Anisotropic density, $\pll$, $\per$ and $B_y$ using explicit schemes are plotted in Figure \eqref{fig:rp4_exp_rho}, \eqref{fig:rp4_exp_pll}, \eqref{fig:rp4_exp_perp} and \eqref{fig:rp4_exp_by}, respectively. We observe that both pressure components have significantly different profiles. Isotropic density, $\pll$, $\perp$ and $B_y$ using explicit schemes are plotted in Figure \eqref{fig:rp4_imp_rho}, \eqref{fig:rp4_imp_pll}, \eqref{fig:rp4_imp_perp} and \eqref{fig:rp4_imp_by}, respectively. We see that all the variables are matching the MHD solution. We observe that the second-order schemes are less accurate when compared with higher-order schemes. We also observe some small-scale oscillations near some of the discontinuities. However, they are stable and do not increase with further mesh refinement.

Total entropy decays at each time step are plotted in Figures \eqref{fig:rp4_exp_ent} and \eqref{fig:rp4_imp_ent} for explicit and IMEX schemes, respectively. We again see that both explicit and IMEX schemes have similar entropy decay performance.

\begin{figure}[!htbp]
	\begin{center}
		\subfloat[Anisotropic case: Density for explicit schemes without source term]{\includegraphics[width=1.5in, height=1.5in]{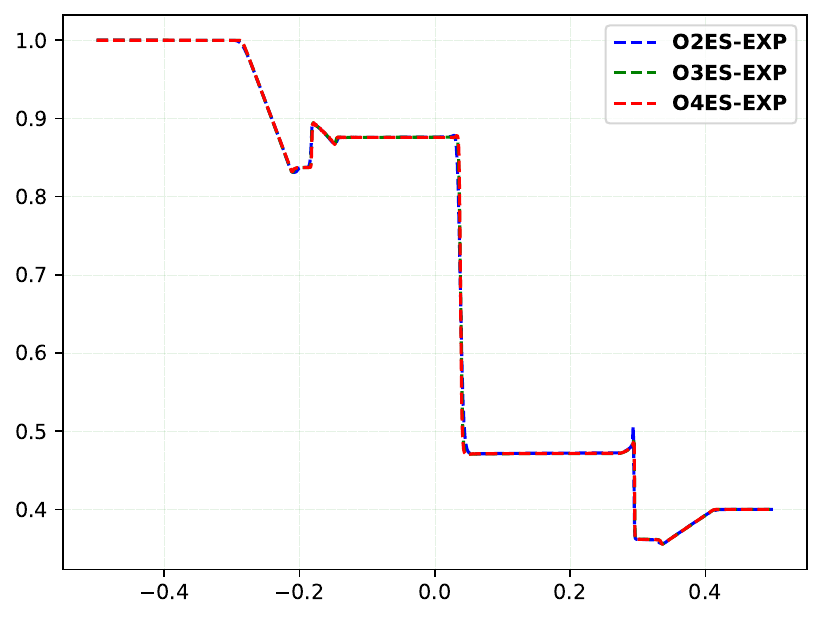} \label{fig:rp4_exp_rho}}~
		\subfloat[Anisotropic case: Pressure component $\pll$ for explicit schemes without source term]{\includegraphics[width=1.5in, height=1.5in]{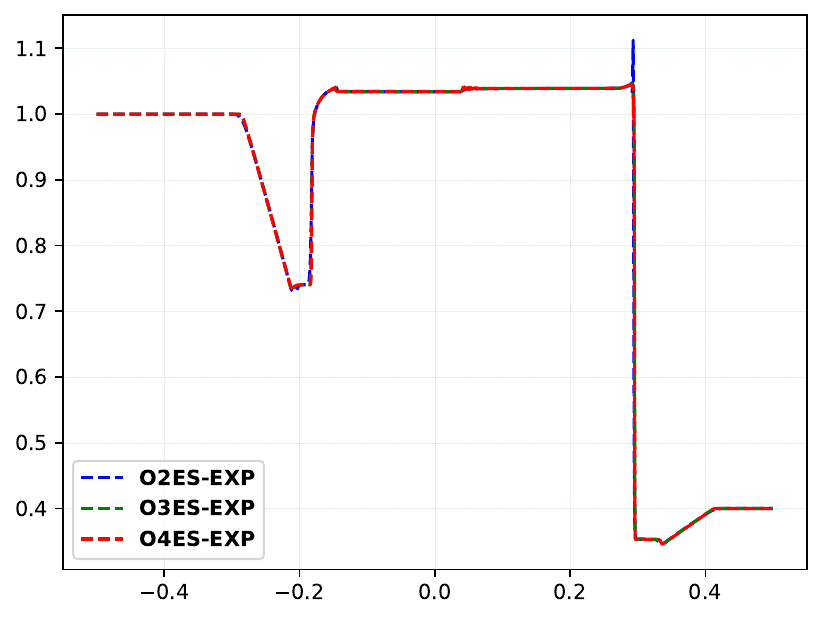}\label{fig:rp4_exp_pll}}~
		\subfloat[Anisotropic case: Pressure component $\per$ for explicit schemes without source term]{\includegraphics[width=1.5in, height=1.5in]{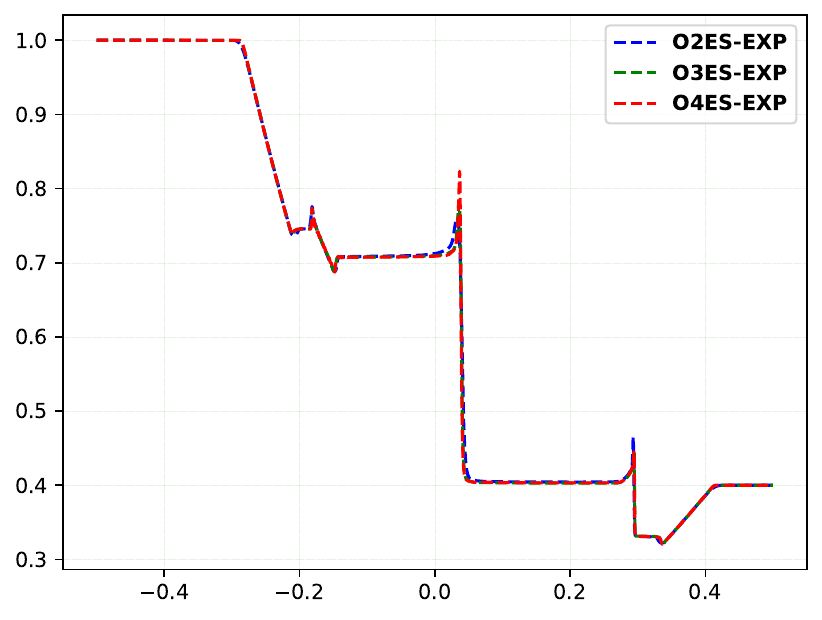}\label{fig:rp4_exp_perp}}\\
		\subfloat[Anisotropic case: Magnetic field  component $B_y$ for explicit schemes without source term]{\includegraphics[width=1.5in, height=1.5in]{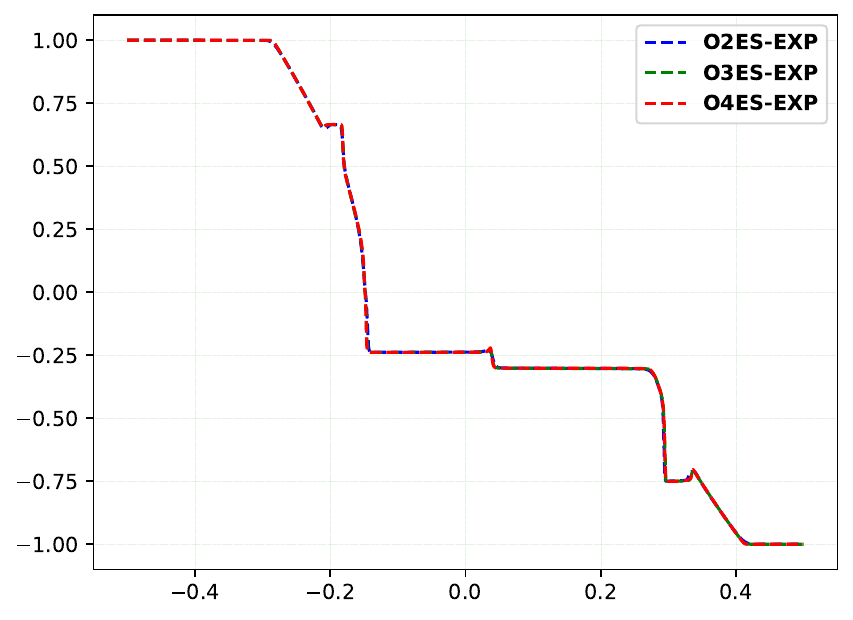}\label{fig:rp4_exp_by}}~
		\subfloat[Isotropic case: Density for IMEX schemes with source term]{\includegraphics[width=1.5in, height=1.5in]{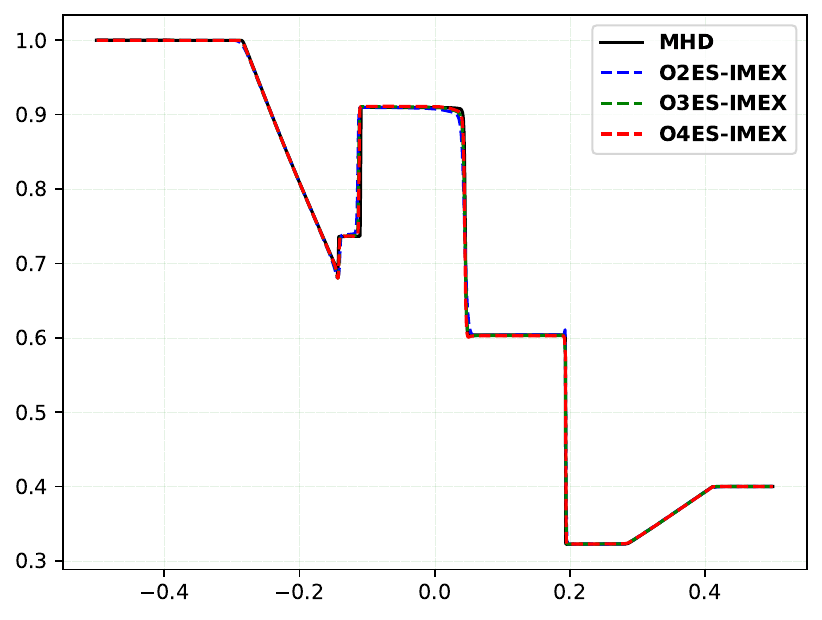}\label{fig:rp4_imp_rho}}~
		\subfloat[Isotropic case: Pressure component $\pll$ for IMEX schemes with source term]{\includegraphics[width=1.5in, height=1.5in]{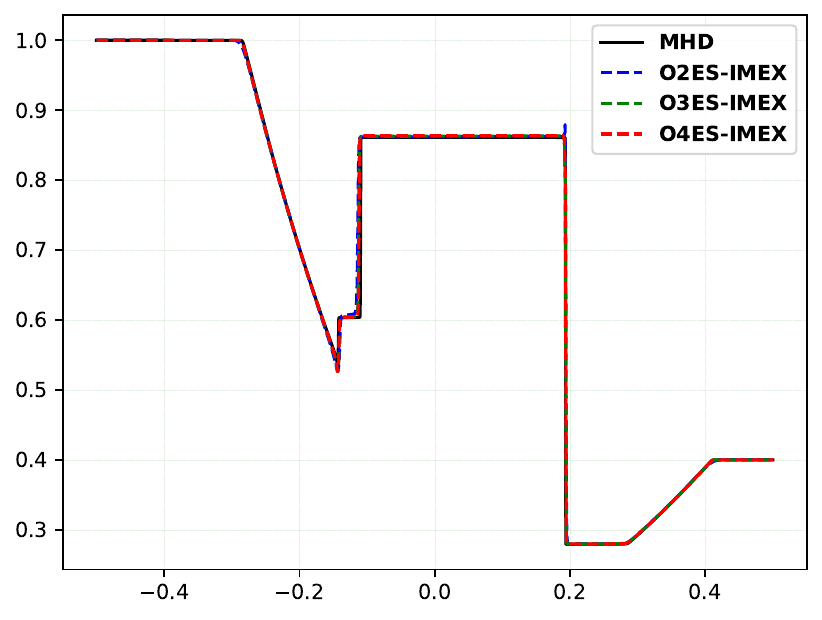}\label{fig:rp4_imp_pll}}\\
		\subfloat[Isotropic case: Pressure component $\per$ for IMEX schemes with source term]{\includegraphics[width=1.5in, height=1.5in]{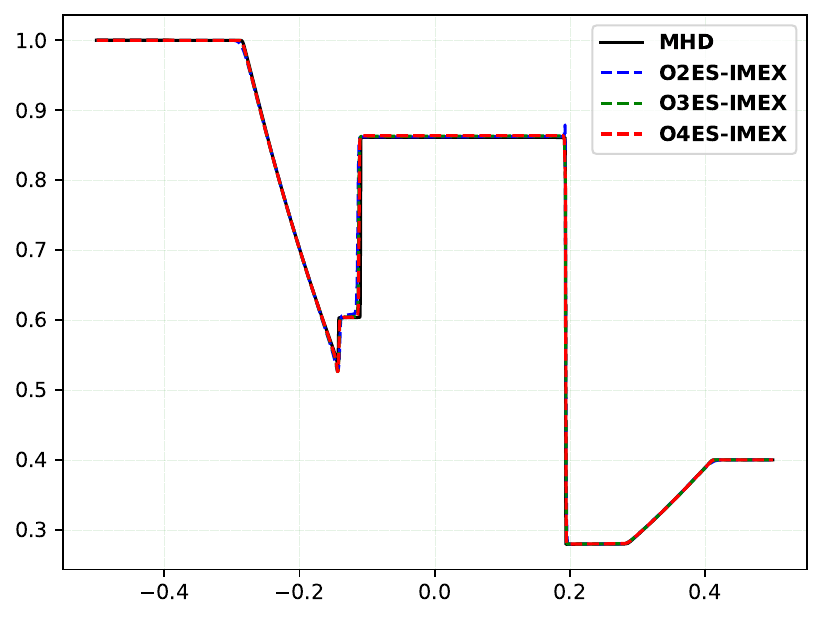}\label{fig:rp4_imp_perp}}~
		\subfloat[Isotropic case: Magnetic field  component $B_y$ for explicit schemes without source term]{\includegraphics[width=1.5in, height=1.5in]{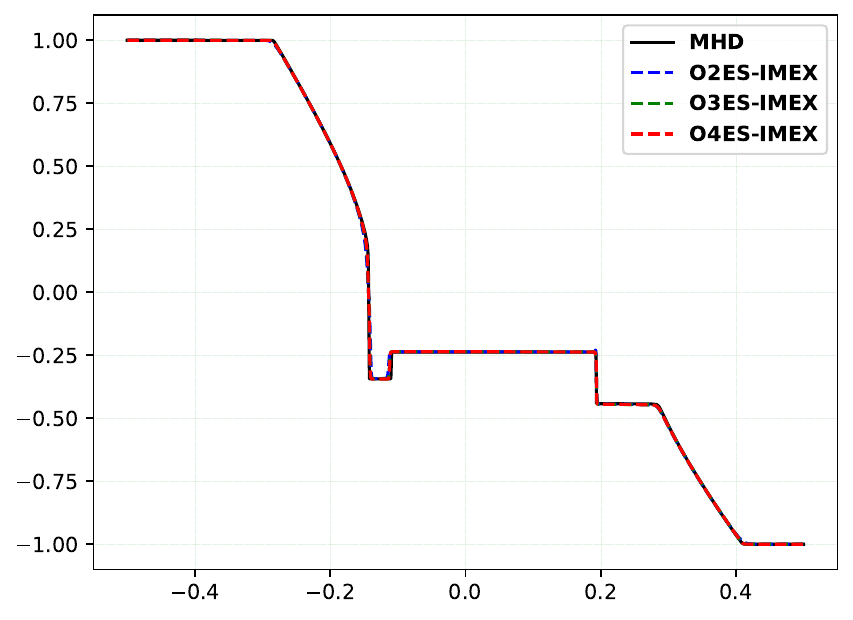}\label{fig:rp4_imp_by}}\\
		\subfloat[Anisotropic case: Total entropy decay at each time step for explicit schemes]{\includegraphics[width=2.0in, height=1.5in]{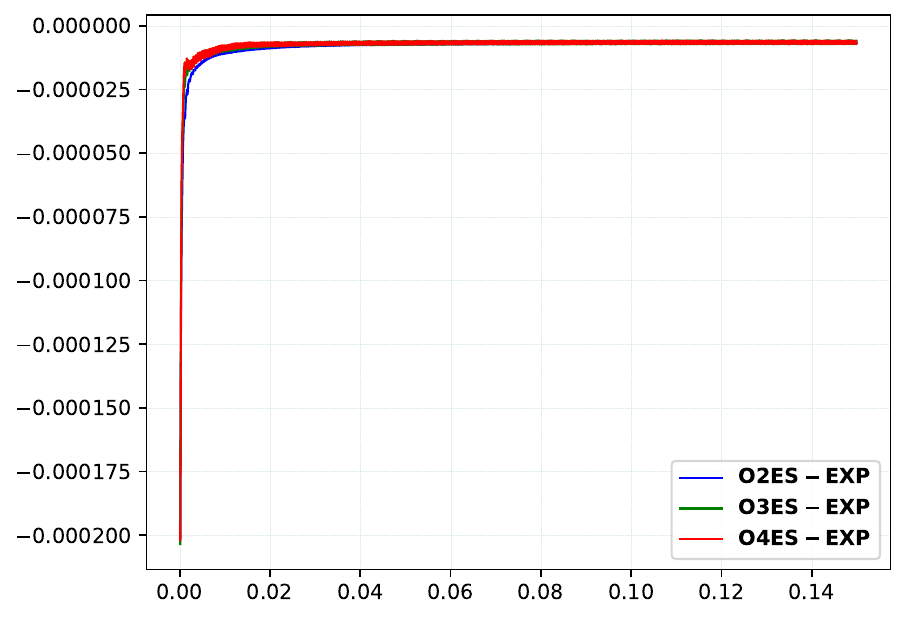}\label{fig:rp4_exp_ent}}~
		\subfloat[Isotropic case: Total entropy decay at each time step for IMEX schemes with source term]{\includegraphics[width=2.0in, height=1.5in]{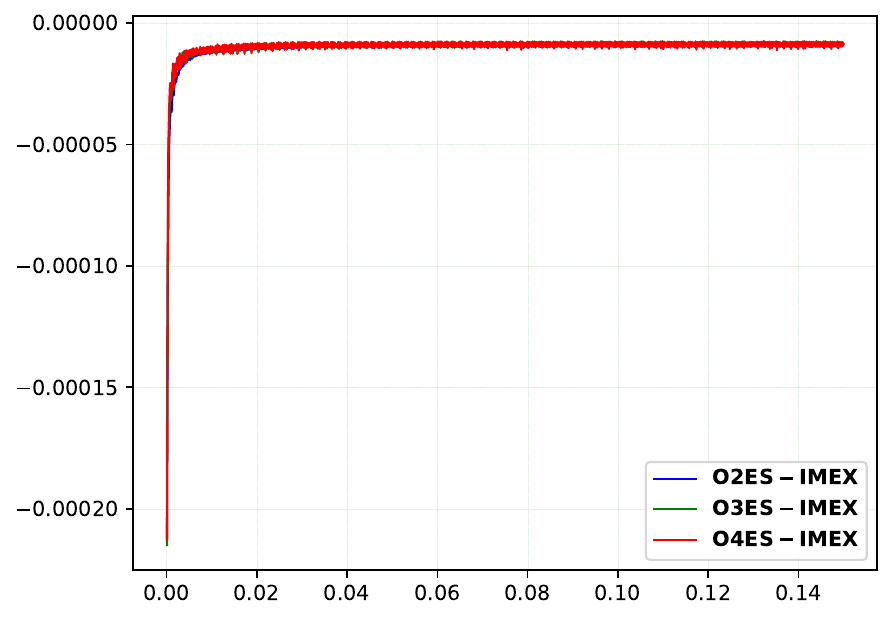}\label{fig:rp4_imp_ent}}
		\caption{\textbf{\nameref{test:rp4}}: Plots of density, parallel and perpendicular pressure components, and total entropy decay at each time step for explicit schemes without source term and IMEX scheme with source term using $2000$ cells at final time $t = 0.15$.}
		\label{fig:rp4}
	\end{center}
\end{figure}

\subsubsection{Riemann problem 5}
\label{test:rp5}
In this test, we consider another Riemann problem from \cite{dumbser2016new}. The initial conditions are given by,
\[(\rho, \bu, \pll, \per, B_{y}, B_{z}) = \begin{cases}
	(1.7, 0, 0, 0, 1.7, 1.7, \frac{3.544908}{\sqrt{4\pi}}, 0),&\textrm{if } x\leq 0\\
	(0.2, 0, 0, -1.496891, 0.2, 0.2, \frac{2.785898}{\sqrt{4\pi}}, \frac{2.192064}{\sqrt{4\pi}}),&\textrm{otherwise}
\end{cases}\]
with $B_{x}=\frac{3.899398}{\sqrt{4\pi}}$. We consider the computational domain $[-0.5,0.5]$ with outflow boundary conditions. We use $2000$ cells and compute the solutions till the final time $t=0.15$.
\begin{figure}[!htbp]
	\begin{center}
		\subfloat[Anisotropic case: Density for explicit schemes without source term]{\includegraphics[width=1.5in, height=1.5in]{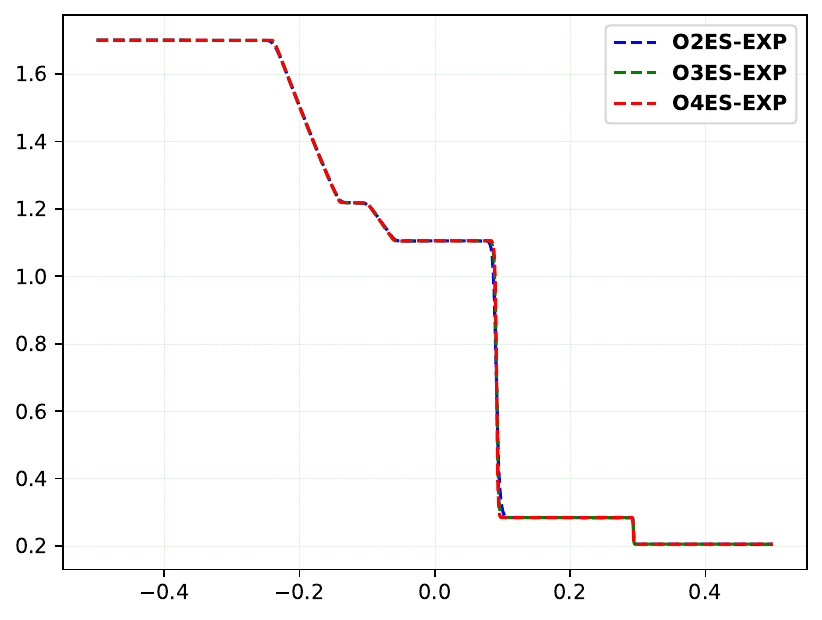} \label{fig:rp5_exp_rho}}~
		\subfloat[Anisotropic case: Pressure component $\pll$ for explicit schemes without source term]{\includegraphics[width=1.5in, height=1.5in]{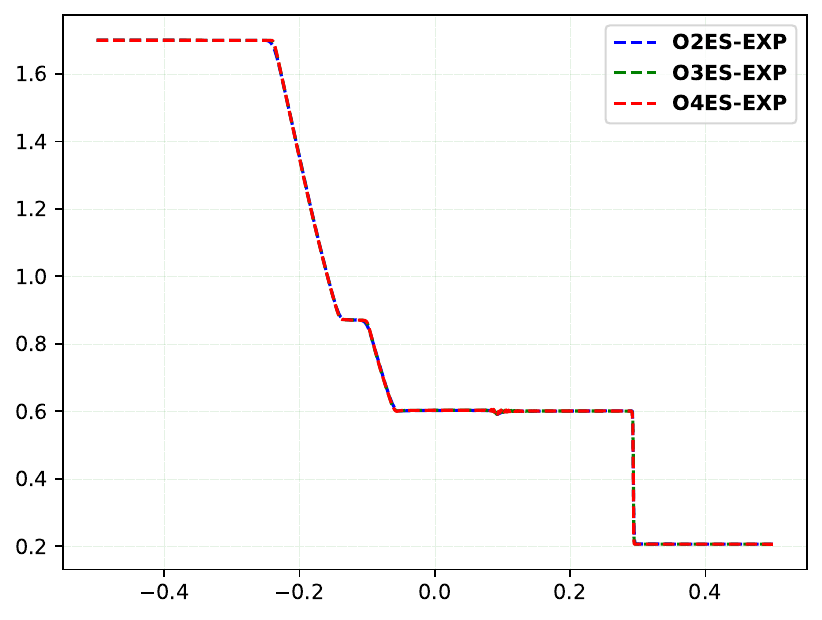}\label{fig:rp5_exp_pll}}~
		\subfloat[Anisotropic case: Pressure component $\per$ for explicit schemes without source term]{\includegraphics[width=1.5in, height=1.5in]{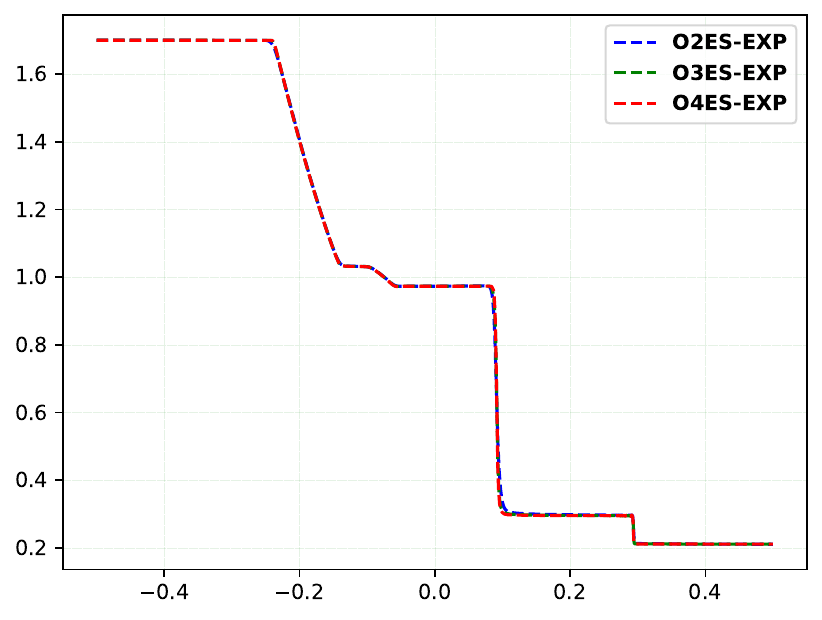}\label{fig:rp5_exp_perp}}\\
		\subfloat[Anisotropic case: Magnetic field  component $B_y$ for explicit schemes without source term]{\includegraphics[width=1.5in, height=1.5in]{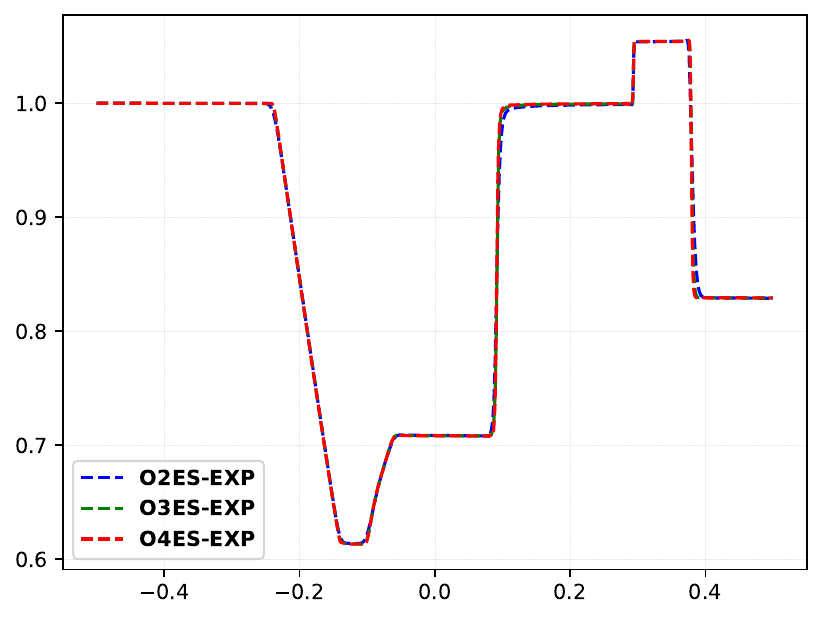}\label{fig:rp5_exp_by}}~
		\subfloat[Isotropic case: Density for IMEX schemes with source term]{\includegraphics[width=1.5in, height=1.5in]{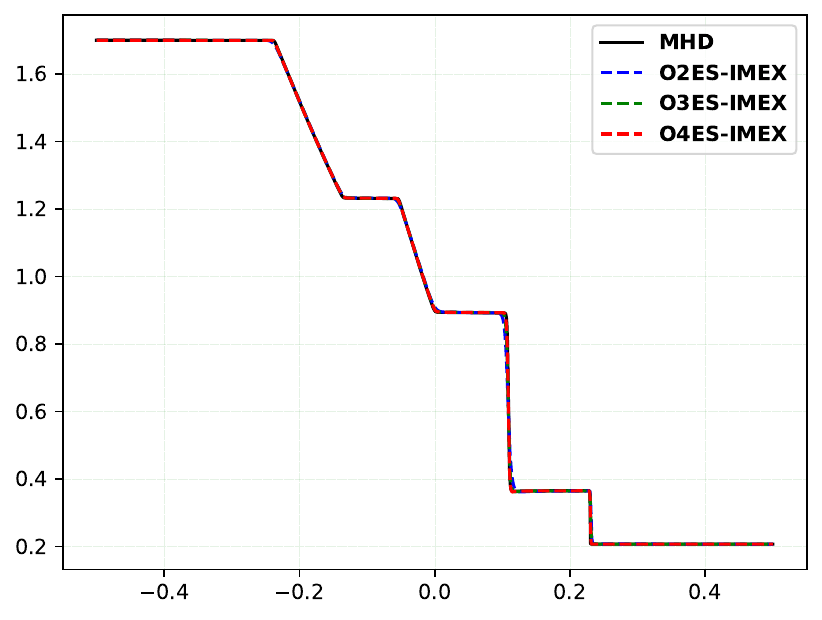}\label{fig:rp5_imp_rho}}~
		\subfloat[Isotropic case: Pressure component $\pll$ for IMEX schemes with source term]{\includegraphics[width=1.5in, height=1.5in]{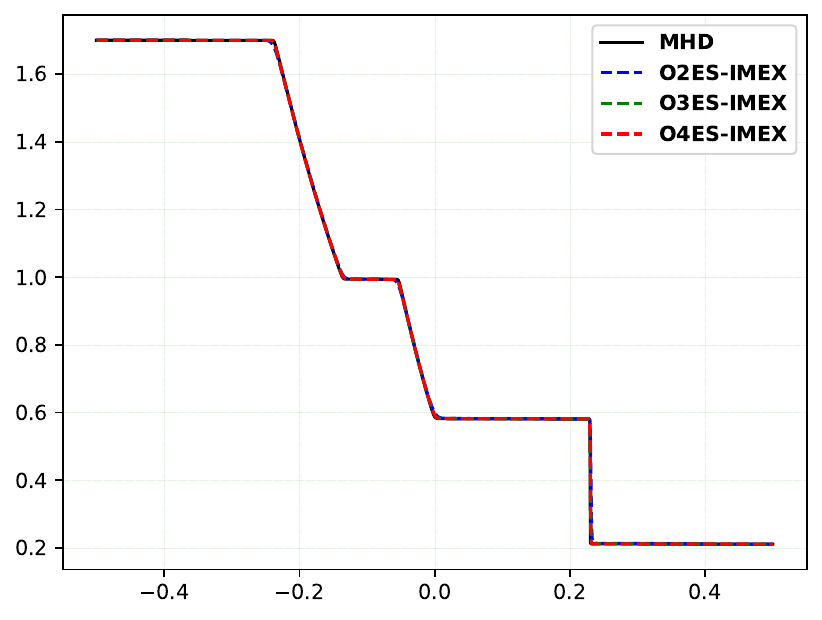}\label{fig:rp5_imp_pll}}\\
		\subfloat[Isotropic case: Pressure component $\per$ for IMEX schemes with source term]{\includegraphics[width=1.5in, height=1.5in]{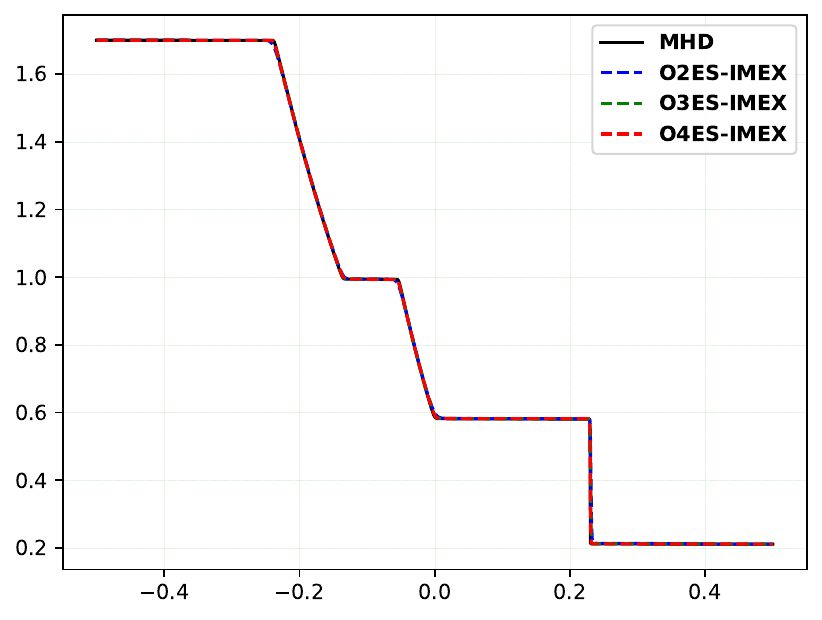}\label{fig:rp5_imp_perp}}~
		\subfloat[Isotropic case: Magnetic field  component $B_y$ for explicit schemes without source term]{\includegraphics[width=1.5in, height=1.5in]{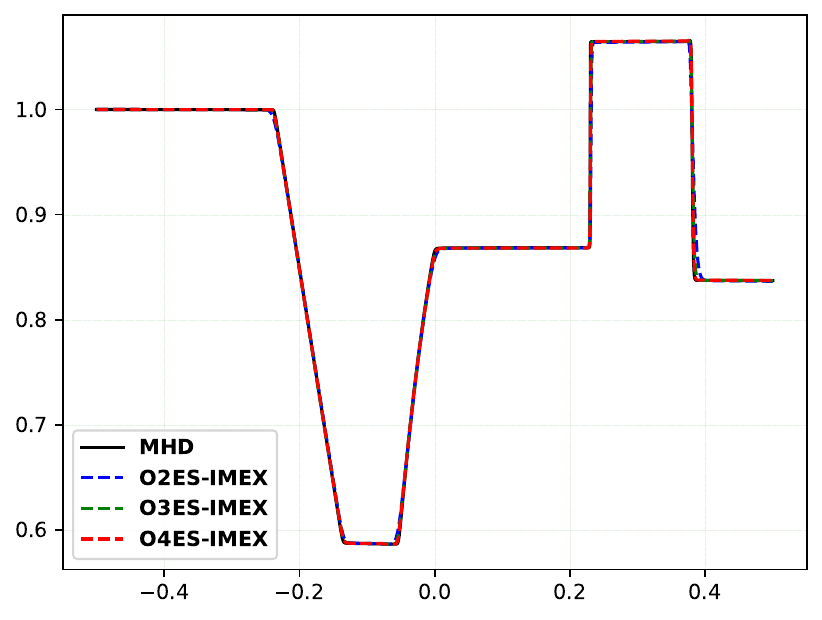}\label{fig:rp5_imp_by}}\\
		\subfloat[Anisotropic case: Total entropy decay at each time step for explicit schemes]{\includegraphics[width=2.0in, height=1.5in]{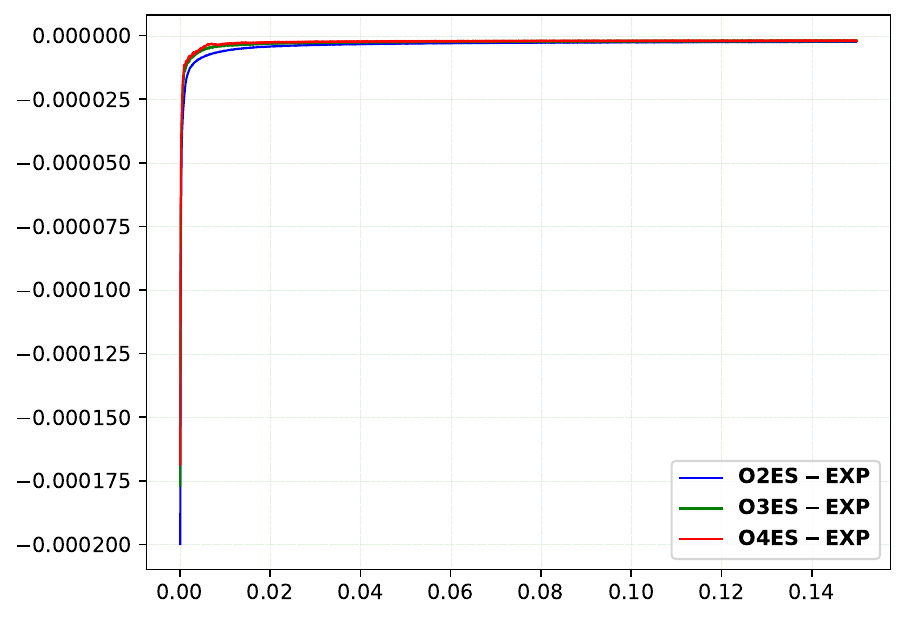}\label{fig:rp5_exp_ent}}~
		\subfloat[Isotropic case: Total entropy decay at each time step for IMEX schemes with source term]{\includegraphics[width=2.0in, height=1.5in]{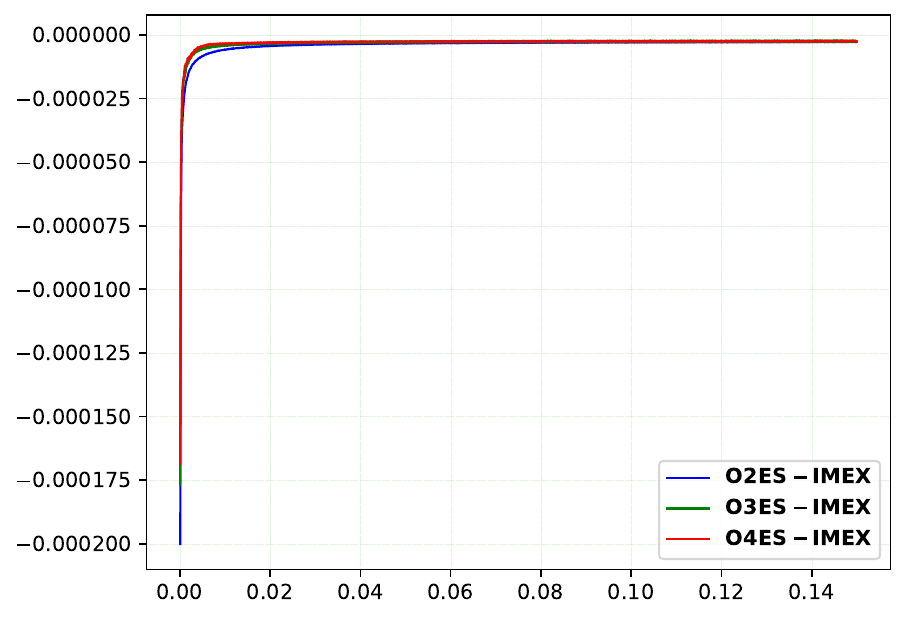}\label{fig:rp5_imp_ent}}
		\caption{\textbf{\nameref{test:rp5}}: Plots of density, parallel and perpendicular pressure components, and total entropy decay at each time step for explicit schemes without source term and IMEX scheme with source term using $2000$ cells at final time $t = 0.15$.}
		\label{fig:rp5}
	\end{center}
\end{figure}
Numerical results are presented in Figure \eqref{fig:rp5} for anisotrpic solutions using IMEX schemes \ote, \othe~and \ofe schemes (see Figures \eqref{fig:rp5_exp_rho}, \eqref{fig:rp5_exp_pll}, \eqref{fig:rp5_exp_perp}, and \eqref{fig:rp5_exp_by} )and for isotropic solutions using \oti, \othi~and \ofi~scheme (see Figures \eqref{fig:rp5_imp_rho}, \eqref{fig:rp5_imp_pll}, \eqref{fig:rp5_imp_perp}, and \eqref{fig:rp5_imp_by}). We have plotted density, pressure components $\pll$ and $\per$ and magnetic filed component $B_y$. We observe that all the schemes are able to resolve different waves. Furthermore, isotropic solutions match MHD solutions for all the variables. From the entropy decay in Figures \eqref{fig:rp5_exp_ent} and \eqref{fig:rp5_imp_ent}, we observe that the second-order schemes are more diffusive than the third and fourth-order schemes.

\subsubsection{Riemann problem 6}
\label{test:rp6}
In this test, we consider the Riemann problem presented in  \cite{Balsara2018efficient} for MHD equations. The computational domain is considered to be $[-0.5,0.5]$ with outflow boundary conditions. The initial conditions are,
\[(\rho, \bu, \pll, \per, B_{y}, B_{z}) = \begin{cases}
	(\frac{1}{4\pi}, -1, 1, -1, 1, 1, -\frac{1}{\sqrt{4\pi}}, \frac{1}{\sqrt{4\pi}}), & \textrm{if } x\leq 0\\
	(\frac{1}{4\pi}, -1, -1, -1, 1, 1, \frac{1}{\sqrt{4\pi}}, \frac{1}{\sqrt{4\pi}}), & \textrm{otherwise}
\end{cases}\]
with magnetic field component $B_{x}=\frac{1}{\sqrt{4\pi}}$. We compute anisotropic and isotropic solutions using explicit (\ote, \othe~and \ofe) and IMEX (\oti, \othi~and \ofi) schemes, respectively with $2000$ cells. The numerical results for variable $\pll$, $\per$ and $B_y$ are presented in Figure \eqref{fig:rp6}. We observe that the wave structures are much more dominant in the anisotropic case than in the isotropic case. Furthermore, in the case of anisotropic solutions, we clearly observe that \ofe~is most accurate, followed by \othe~and \ote~schemes (see Figures  \eqref{fig:rp6_exp_pll},  \eqref{fig:rp6_exp_perp} and  \eqref{fig:rp6_exp_by} ). Furthermore, the isotropic solutions matched the MHD solution for IMEX schemes see Figures  \eqref{fig:rp6_imp_pll},  \eqref{fig:rp6_imp_perp} and  \eqref{fig:rp6_imp_by}. We also note that entropy decays are more oscillatory (see Figures \eqref{fig:rp6_exp_ent} and \eqref{fig:rp6_imp_ent}).

\begin{figure}[!htbp]
	\begin{center}
		\subfloat[Anisotropic case: Pressure component $\pll$ for explicit schemes without source term]{\includegraphics[width=1.5in, height=1.5in]{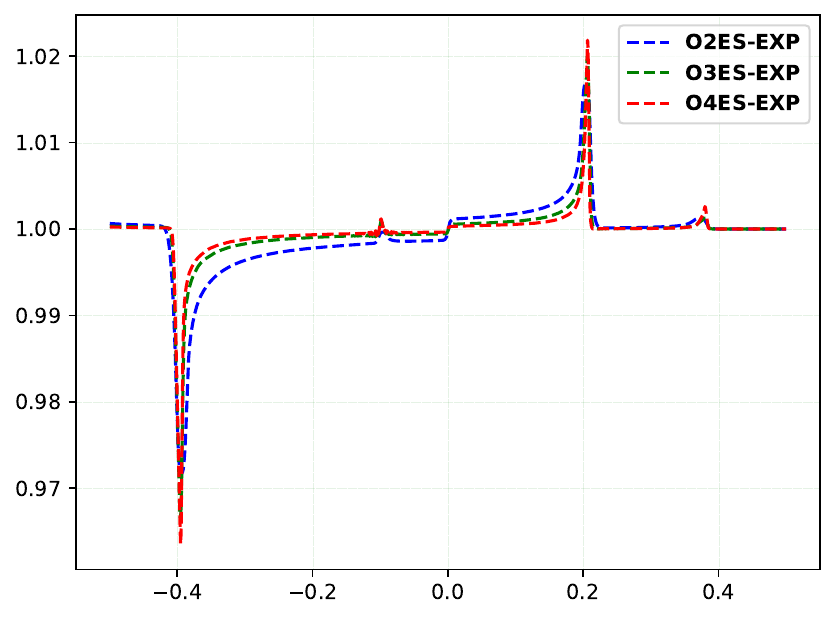}\label{fig:rp6_exp_pll}}~
		\subfloat[Anisotropic case: Pressure component $\per$ for explicit schemes without source term]{\includegraphics[width=1.5in, height=1.5in]{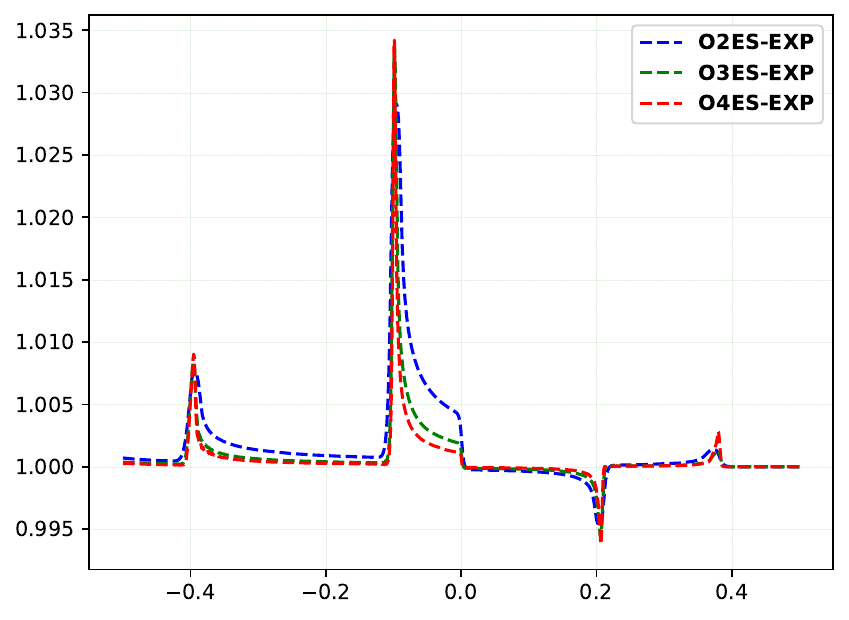}\label{fig:rp6_exp_perp}}~
		\subfloat[Anisotropic case: Magnetic field  component $B_y$ for explicit schemes without source term]{\includegraphics[width=1.5in, height=1.5in]{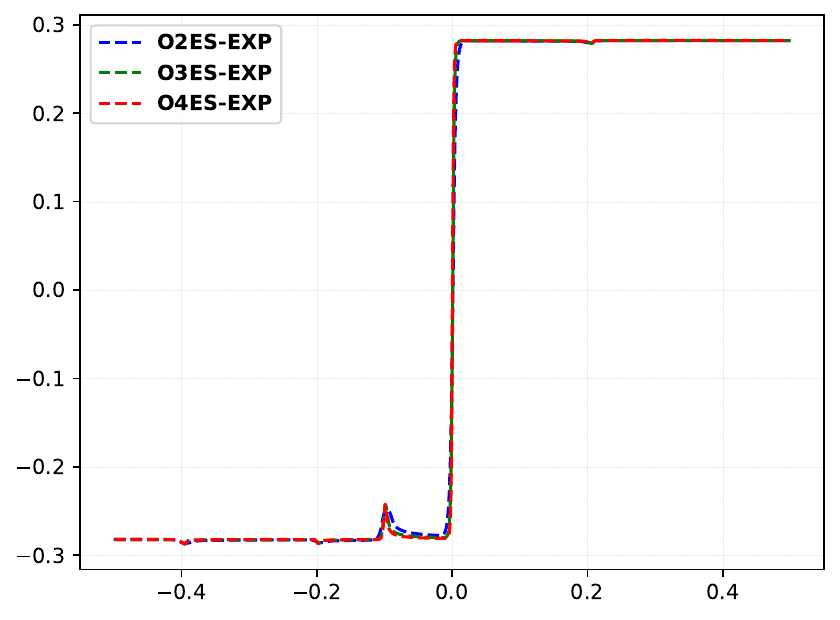}\label{fig:rp6_exp_by}}\\
		\subfloat[Isotropic case: Pressure component $\pll$ for IMEX schemes with source term]{\includegraphics[width=1.5in, height=1.5in]{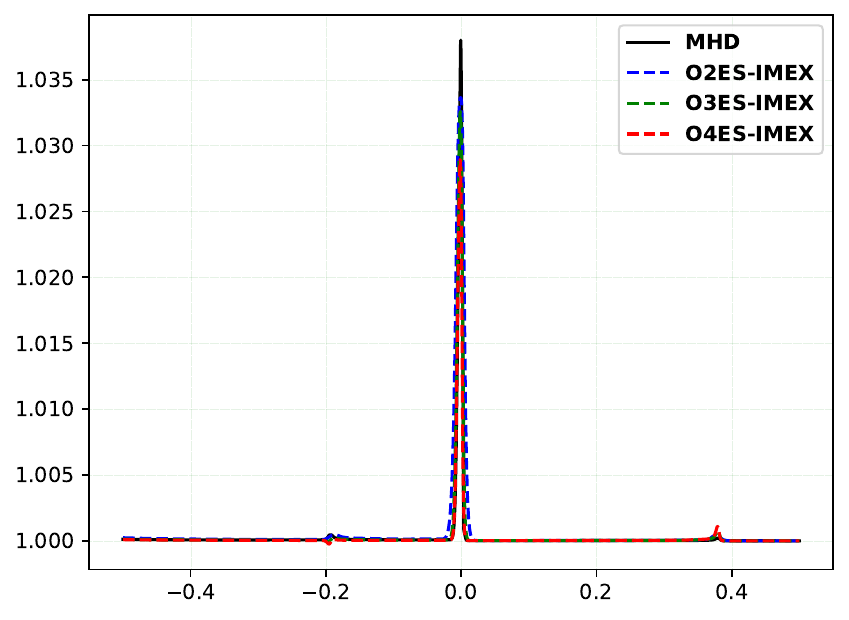}\label{fig:rp6_imp_pll}}~
		\subfloat[Isotropic case: Pressure component $\per$ for IMEX schemes with source term]{\includegraphics[width=1.5in, height=1.5in]{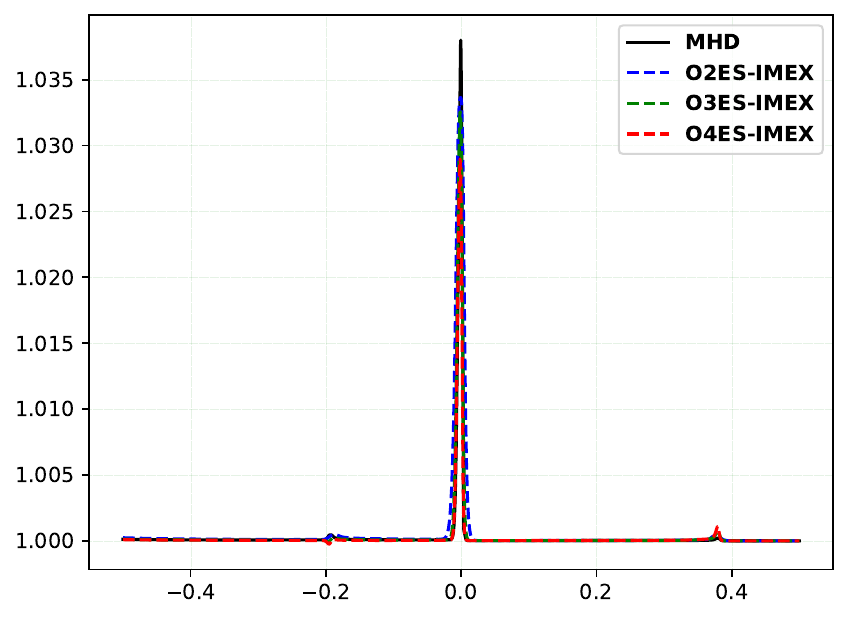}\label{fig:rp6_imp_perp}}~
		\subfloat[Isotropic case: Magnetic field  component $B_y$ for explicit schemes without source term]{\includegraphics[width=1.5in, height=1.5in]{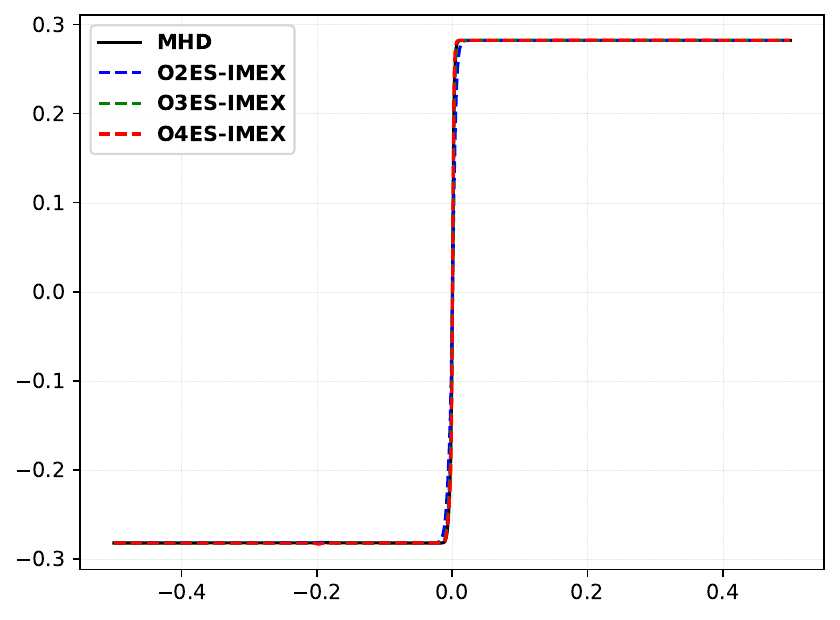}\label{fig:rp6_imp_by}}\\
		\subfloat[Anisotropic case: Total entropy decay at each time step for explicit schemes]{\includegraphics[width=2.0in, height=1.5in]{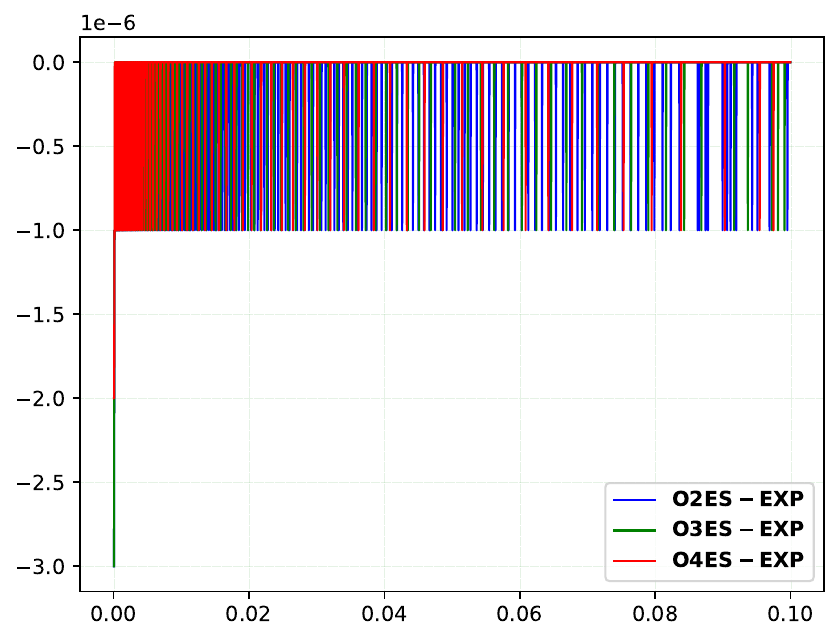}\label{fig:rp6_exp_ent}}~
		\subfloat[Isotropic case: Total entropy decay at each time step for IMEX schemes with source term]{\includegraphics[width=2.0in, height=1.5in]{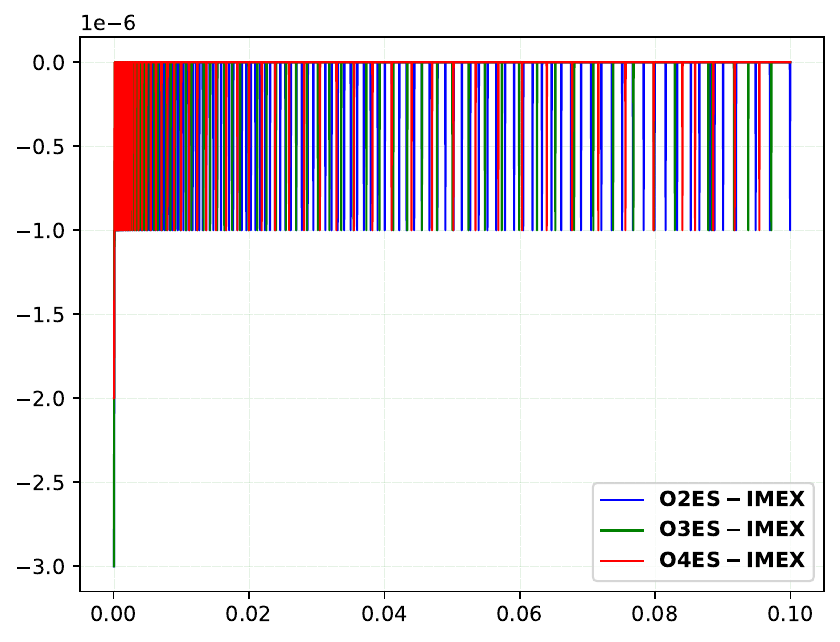}\label{fig:rp6_imp_ent}}
		\caption{\textbf{\nameref{test:rp6}}: Plots of density, parallel and perpendicular pressure components, and total entropy decay at each time step for explicit schemes without source term and IMEX scheme with source term using $2000$ cells at final time $t = 0.15$.}
		\label{fig:rp6}
	\end{center}
\end{figure}

\subsubsection{Riemann problem 7}
\label{test:rp7}
The last one-dimensional test case for CGL system is designed using MHD Riemann problem from \cite{Balsara2018efficient}. The initial conditions are given by,
\[(\rho, \bu, \pll, \per, B_{y}, B_{z}) = \begin{cases}
	(1, 0, 0, 0, 1, 1, 1, 0), & \textrm{if } x\leq 0\\
	(0.2, 0, 0, 0, 0.1, 0.1, 0, 0), & \textrm{otherwise},
\end{cases}\]
with $B_{x}=1$. We take the computational domain of $[-0.5,0.5]$ with outflow boundary conditions. We compute the solution till the final time $t=0.15$. The numerical solutions using explicit and IMEX schemes for anisotropic and isotropic cases are presented in Figure \eqref{fig:rp7} using $2000$ cells. We again observe that all the schemes are able to resolve wave structures for both anisotropic and isotropic cases. We also observed a significant change in the wave structure when comparing anisotropic and isotropic cases. For the isotropic case, the solutions match with the MHD solution.  For the entropy decay, we observe that explicit and IMEX schemes have similar decays.

\begin{figure}[!htbp]
	\begin{center}
		\subfloat[Anisotropic case: Density for explicit schemes without source term]{\includegraphics[width=1.5in, height=1.5in]{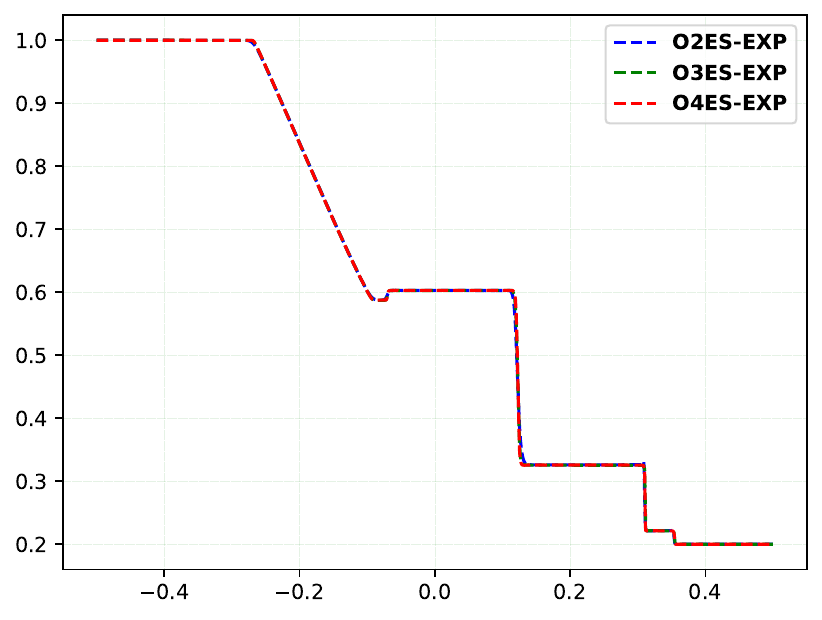} \label{fig:rp7_exp_rho}}~
		\subfloat[Anisotropic case: Pressure component $\pll$ for explicit schemes without source term]{\includegraphics[width=1.5in, height=1.5in]{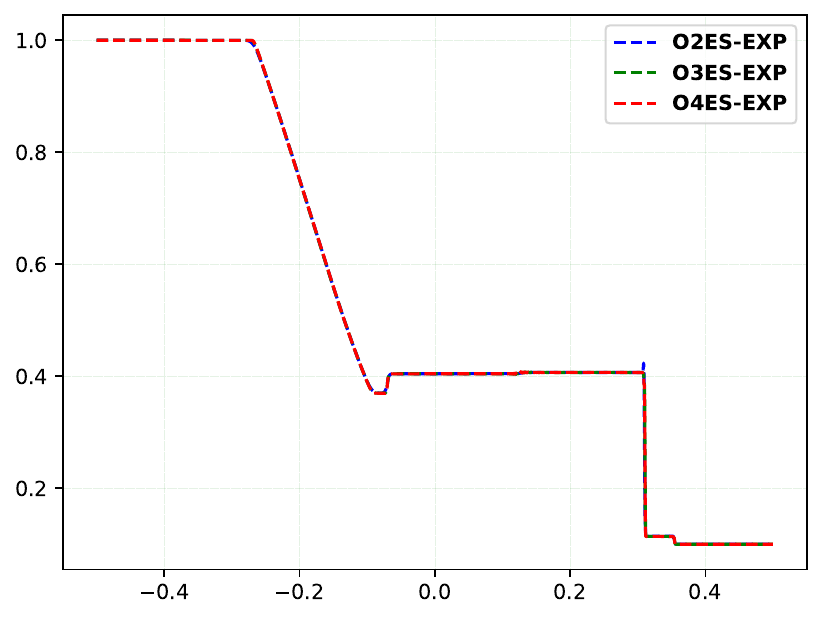}\label{fig:rp7_exp_pll}}~
		\subfloat[Anisotropic case: Pressure component $\per$ for explicit schemes without source term]{\includegraphics[width=1.5in, height=1.5in]{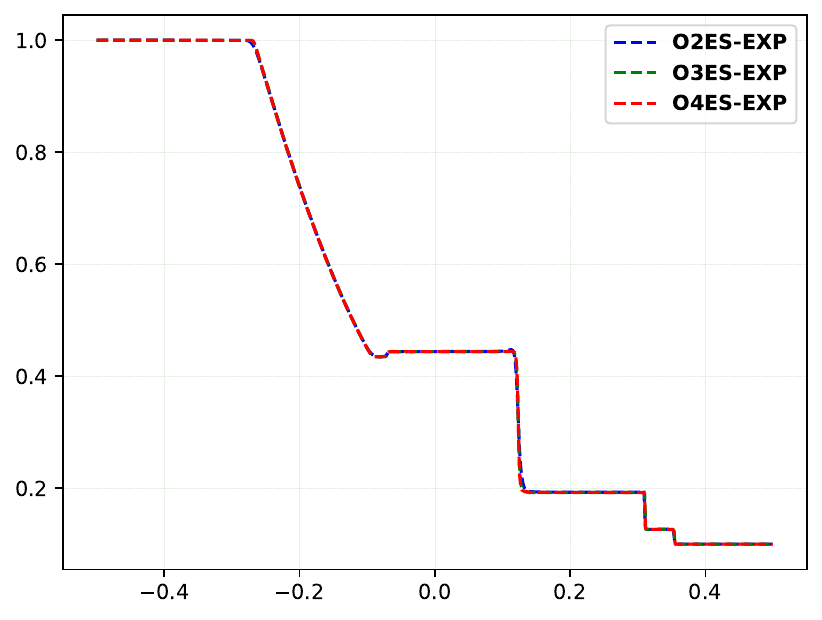}\label{fig:rp7_exp_perp}}\\
		\subfloat[Anisotropic case: Magnetic field  component $B_y$ for explicit schemes without source term]{\includegraphics[width=1.5in, height=1.5in]{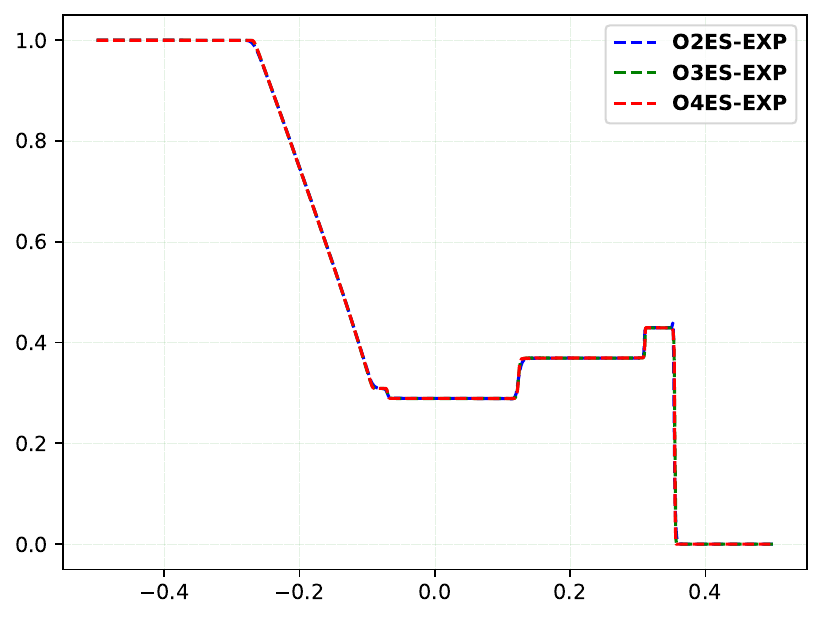}\label{fig:rp7_exp_by}}~
		\subfloat[Isotropic case: Density for IMEX schemes with source term]{\includegraphics[width=1.5in, height=1.5in]{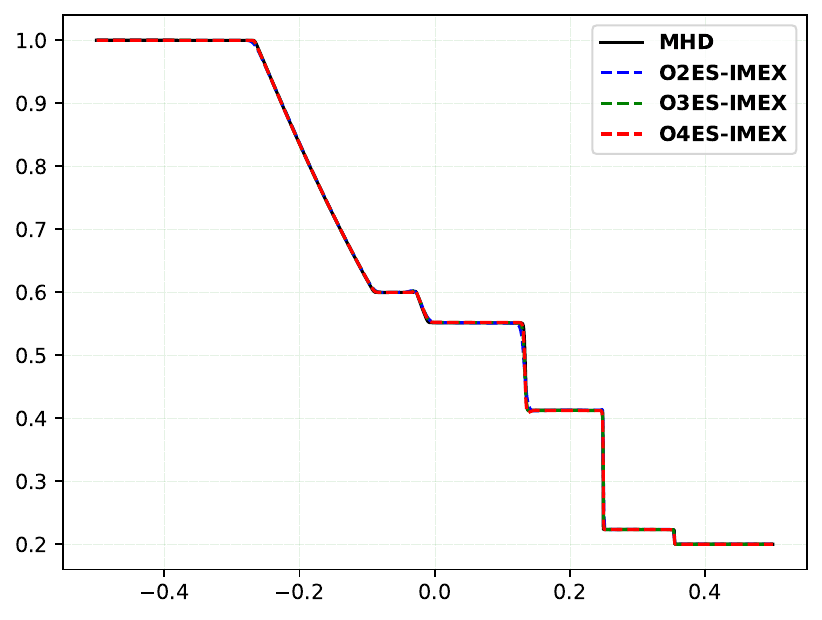}\label{fig:rp7_imp_rho}}~
		\subfloat[Isotropic case: Pressure component $\pll$ for IMEX schemes with source term]{\includegraphics[width=1.5in, height=1.5in]{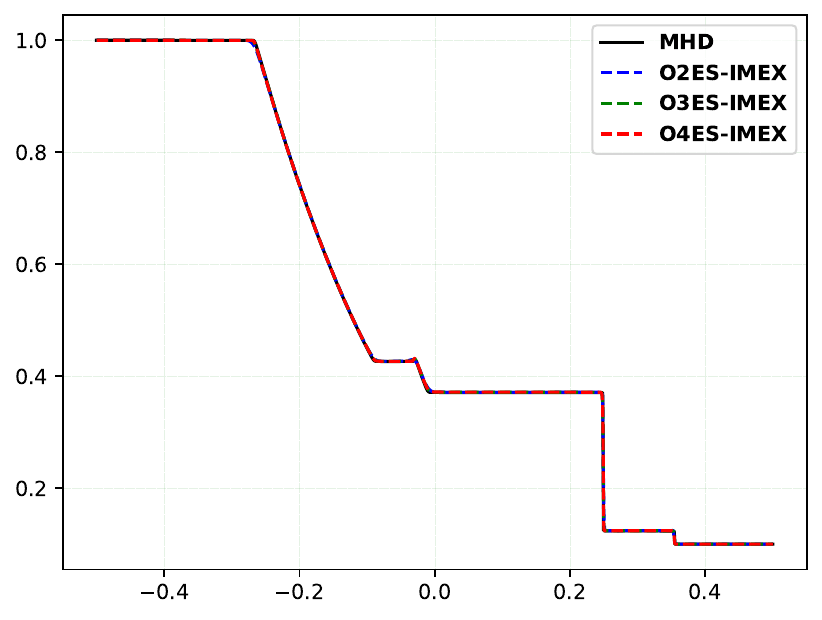}\label{fig:rp7_imp_pll}}\\
		\subfloat[Isotropic case: Pressure component $\per$ for IMEX schemes with source term]{\includegraphics[width=1.5in, height=1.5in]{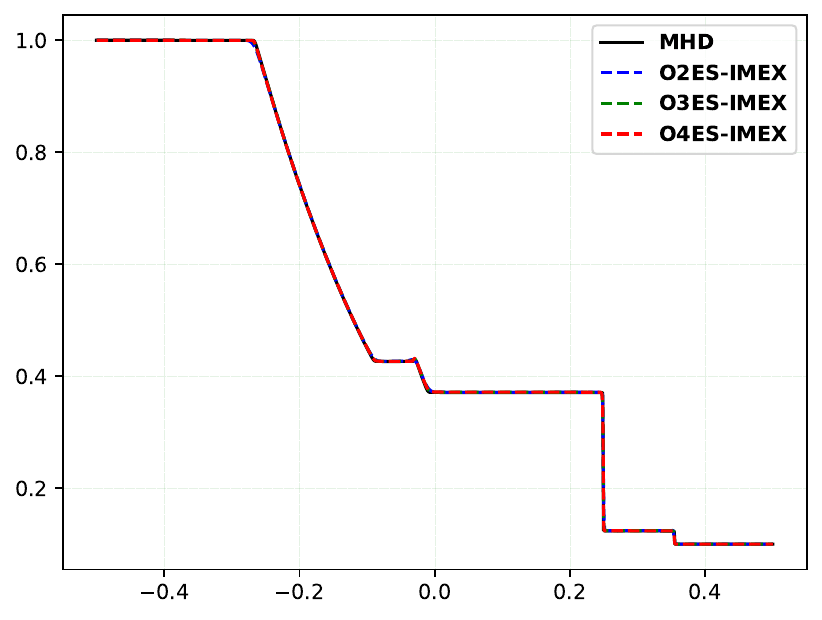}\label{fig:rp7_imp_perp}}~
		\subfloat[Isotropic case: Magnetic field  component $B_y$ for explicit schemes without source term]{\includegraphics[width=1.5in, height=1.5in]{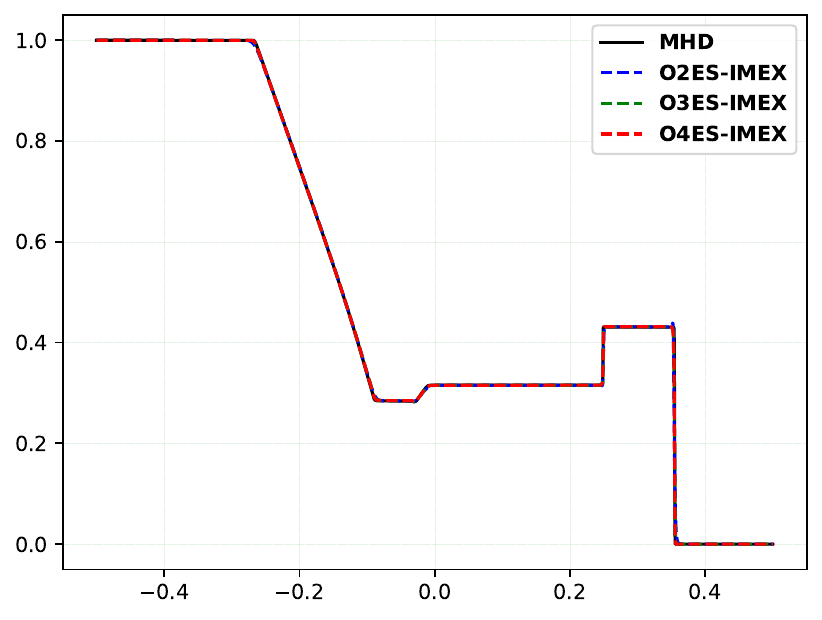}\label{fig:rp7_imp_by}}\\
		\subfloat[Anisotropic case: Total entropy decay at each time step for explicit schemes]{\includegraphics[width=2.0in, height=1.5in]{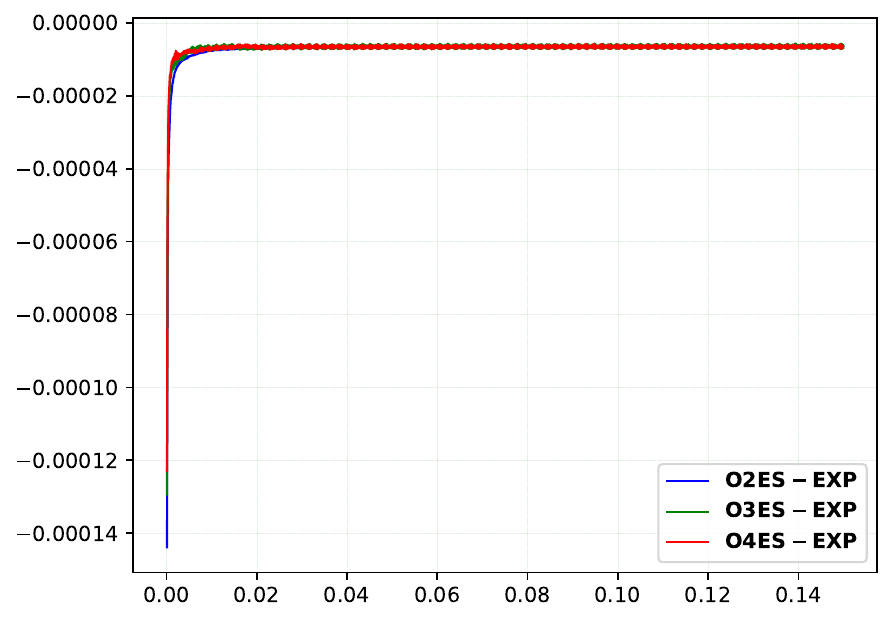}\label{fig:rp7_exp_ent}}~
		\subfloat[Isotropic case: Total entropy decay at each time step for IMEX schemes with source term]{\includegraphics[width=2.0in, height=1.5in]{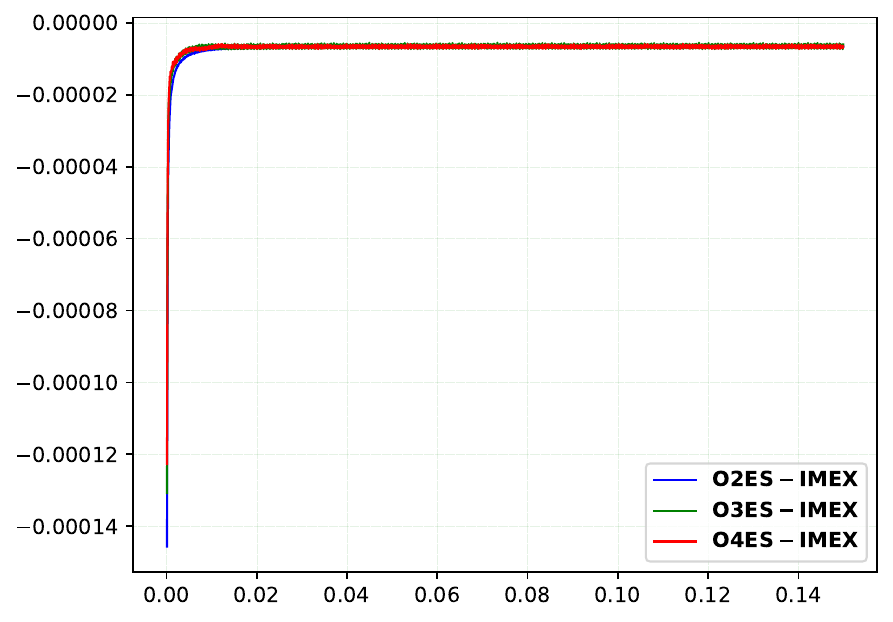}\label{fig:rp7_imp_ent}}
		\caption{\textbf{\nameref{test:rp7}}: Plots of density, parallel and perpendicular pressure components, and total entropy decay at each time step for explicit schemes without source term and IMEX scheme with source term using $2000$ cells at final time $t = 0.15$.}
		\label{fig:rp7}
	\end{center}
\end{figure}
\subsection{Two-dimensional test problems} 
\subsubsection{Orszag–Tang vortex}
\label{test:ot}
To test the stability of the proposed schemes in two dimensions, we generalize the Orzag-Tang vortex test case of MHD equations, which was first considered in \cite{orszag1979small}. Following \cite{toth2000b,chandrashekar2016entropy}, we consider the computational domain $[0, 1] \times [0, 1]$ with periodic boundary conditions. The initial conditions are specified as follows: 
\begin{align*}
	\rho = \frac{25}{36\pi},~~~~\bu &= (-\sin{(2\pi y)}, \sin{(2\pi x)}, 0),~~~~\per = \pll = \frac{5}{12\pi},\\
	\B &= \frac{1}{\sqrt{4\pi}}(-\sin{(2\pi y)}, \sin{(4\pi x)}, 0).
\end{align*}
We present the solutions for the isotropic case using IMEX schemes only, as we do not have any comparable results for the anisotropic case.  The numerical results for \oti, \othi~and \ofi~schemes at $200\times 200$, $400 \times 400$ and $800 \times 800$ cells are presented in Figure \eqref{fig:ot}. We observe that all the schemes are able to resolve the complicated flow features. To compare the solution with MHD solution, we have plotted the cut of variables $\pll$ and $\per$ for $200 \times 200$ cells and compared the solutions with MHD pressure (see Figures \eqref{fig:ot_cut_pll} and \eqref{fig:ot_cut_perp} ). We observe that all the schemes produce solutions similar to the MHD solution, with \ofi~being the most accurate and \oti~being the most diffusive.

The solution is initially smooth and, over time, develops discontinuities. Hence, low entropy decay should be expected initially, and as the discontinuities form, entropy decay should increase. We observe the same behaviour of the schemes in Figure \eqref{fig:ot_entropy}, \eqref{fig:ot_entropy_400} and \eqref{fig:ot_entropy_800}. We also note that the second-order scheme has the most entropy decay.

\begin{figure}[!htbp]
	\begin{center}
		\subfloat[Density using $200 \times 200$ cells for \oti scheme]{\includegraphics[width=1.5in, height=1.5in]{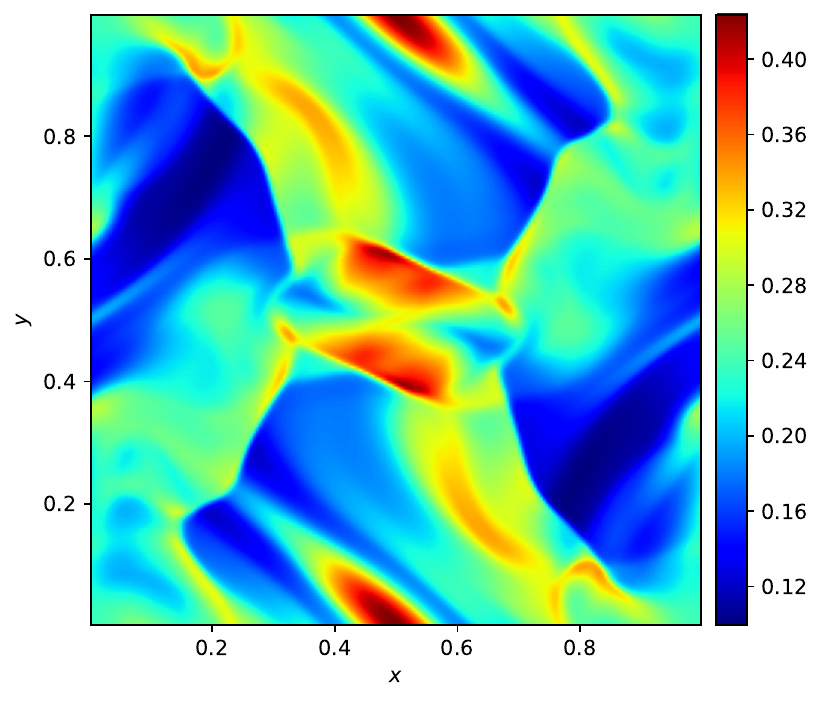}\label{fig:ot_o2i_200}}~
		\subfloat[Density using $400 \times 400$ cells for \oti scheme]{\includegraphics[width=1.5in, height=1.5in]{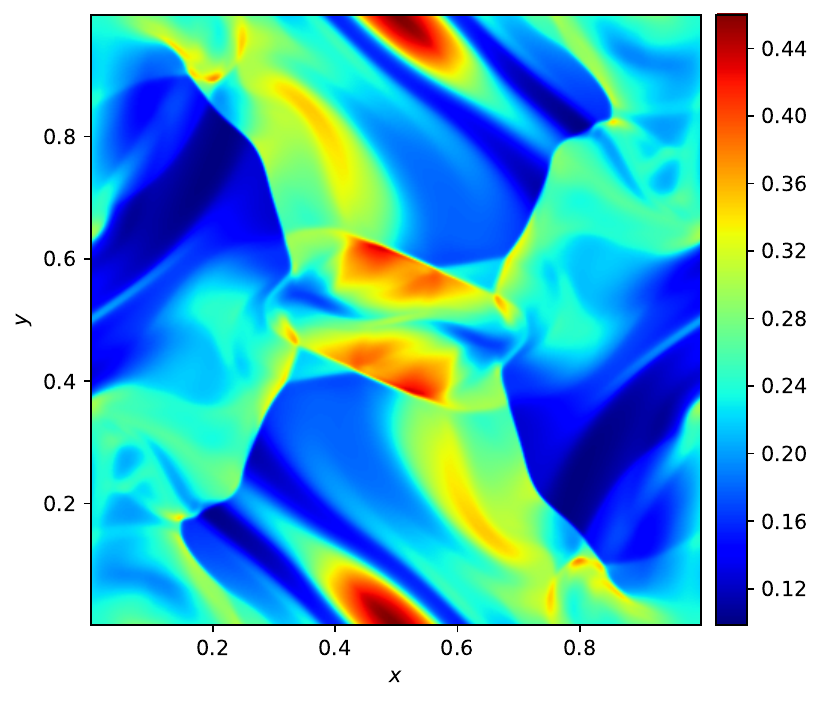}\label{fig:ot_o2i_400}}~
		\subfloat[Density using $800 \times 800$ cells for \oti scheme]{\includegraphics[width=1.5in, height=1.5in]{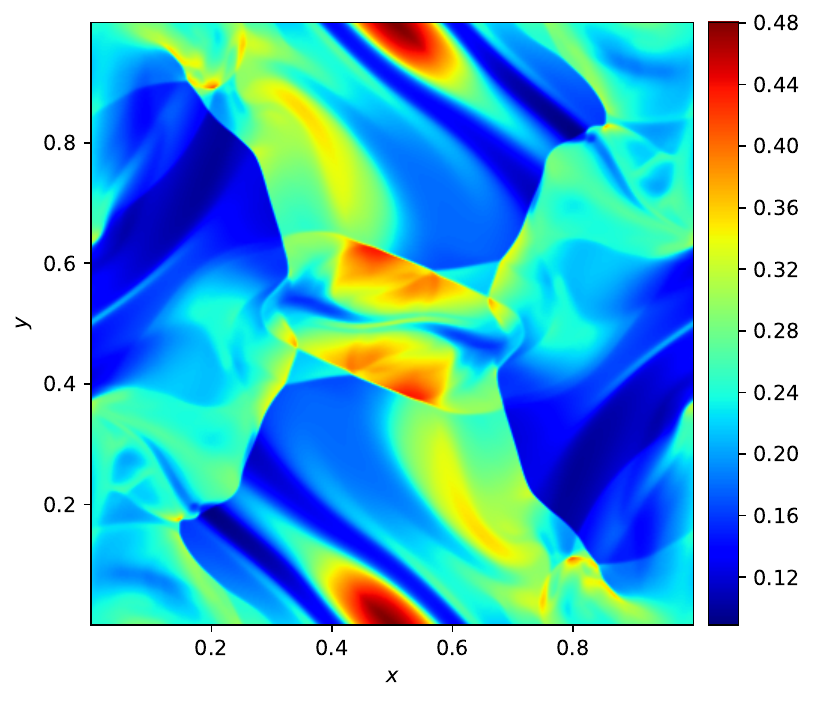}\label{fig:ot_o2i_800}}\\
		\subfloat[Density using $200 \times 200$ cells for \othi scheme]{\includegraphics[width=1.5in, height=1.5in]{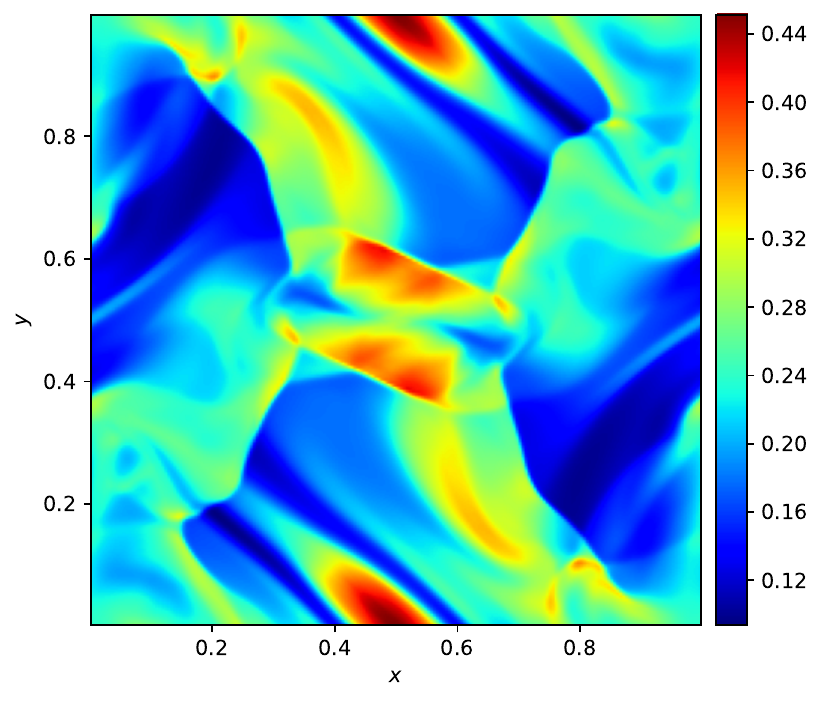}\label{fig:ot_o3i_200}}~
		\subfloat[Density using $400 \times 400$ cells for \othi scheme]{\includegraphics[width=1.5in, height=1.5in]{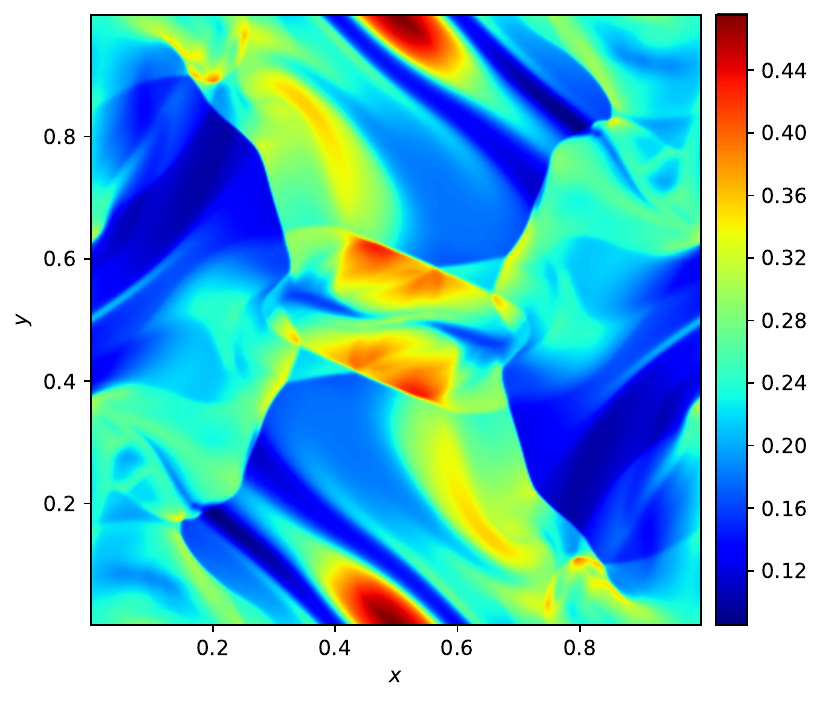}\label{fig:ot_o3i_400}}~
		\subfloat[Density using $800 \times 800$ cells for \othi scheme]{\includegraphics[width=1.5in, height=1.5in]{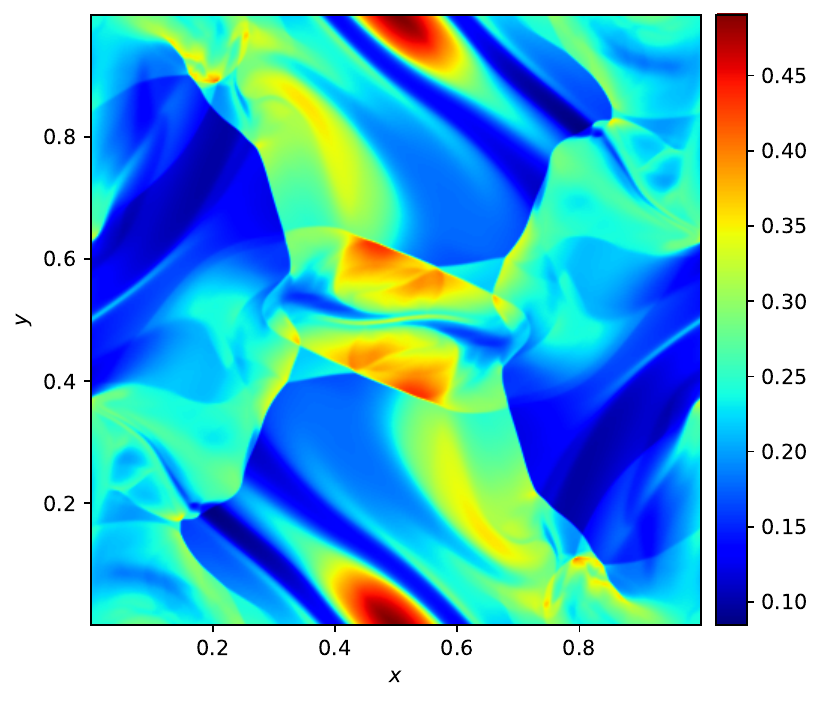}\label{fig:ot_o3i_800}}\\
		\subfloat[Density using $200 \times 200$ cells for \ofi scheme]{\includegraphics[width=1.5in, height=1.5in]{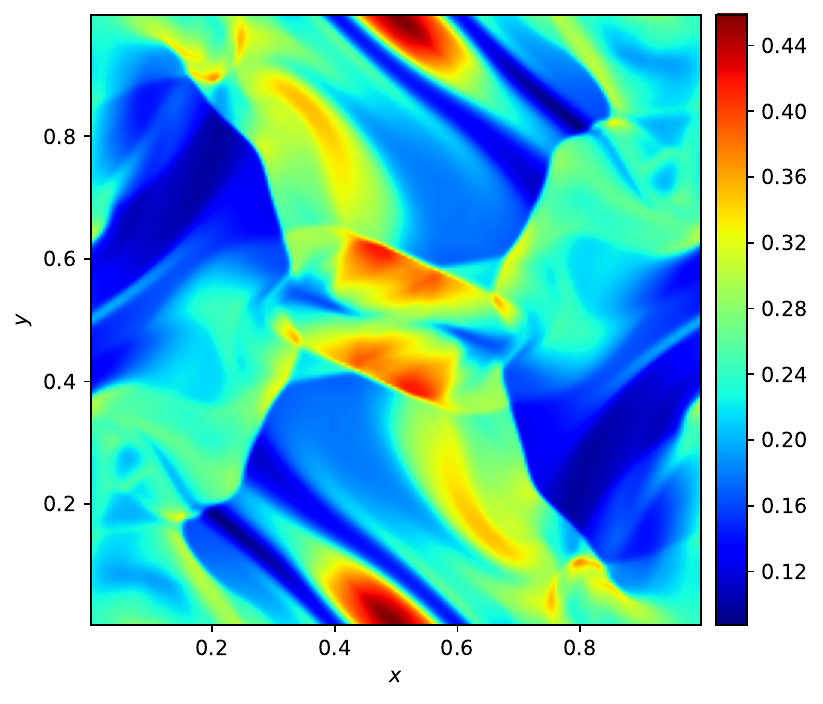}\label{fig:ot_o4i_200}}~
		\subfloat[Density using $400 \times 400$ cells for \ofi 4cheme]{\includegraphics[width=1.5in, height=1.5in]{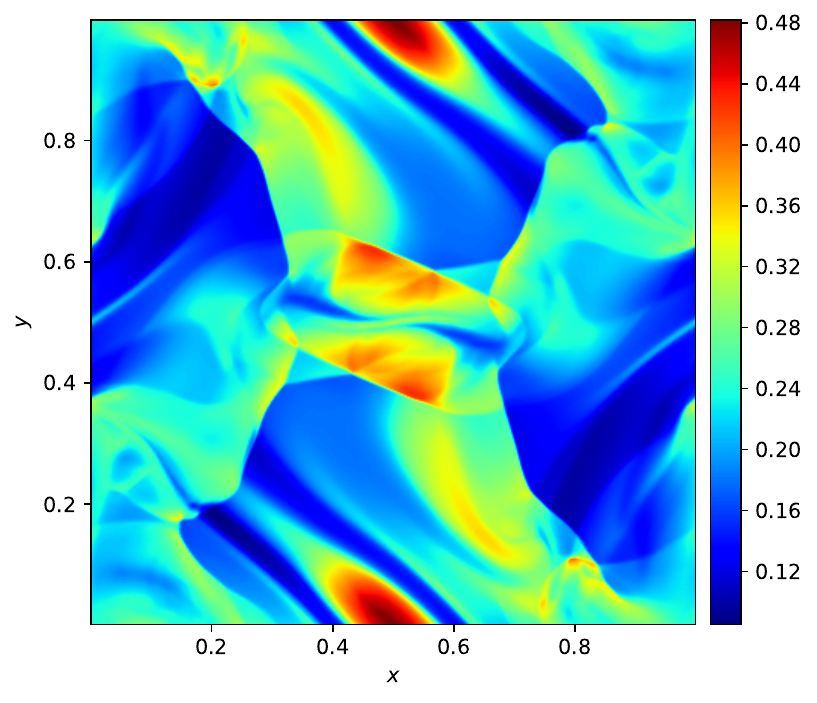}\label{fig:ot_o4i_400}}~
		\subfloat[Density using $800 \times 800$ cells for \ofi scheme]{\includegraphics[width=1.5in, height=1.5in]{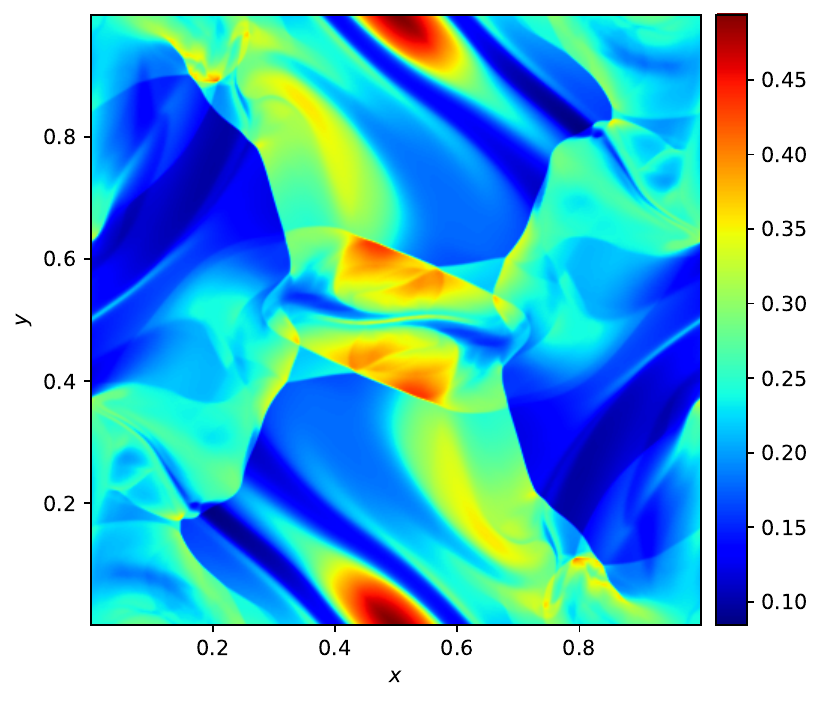}\label{fig:ot_o4i_800}}
		\caption{\textbf{\nameref{test:ot}}: Isotropic case:  Plots of density at different meshes for IMEX schemes \oti, \othi~and \ofi~at time $t=0.5$.}
		\label{fig:ot}
	\end{center}
\end{figure}
\begin{figure}[!htbp]
	\begin{center}
		\subfloat[Isotropic case: Cut of pressure component $\pll$ along $y = 0.3125$ with $200\times 200$ cells using \oti, \othi and \ofi schemes ]{\includegraphics[width=5.0in, height=1.5in]{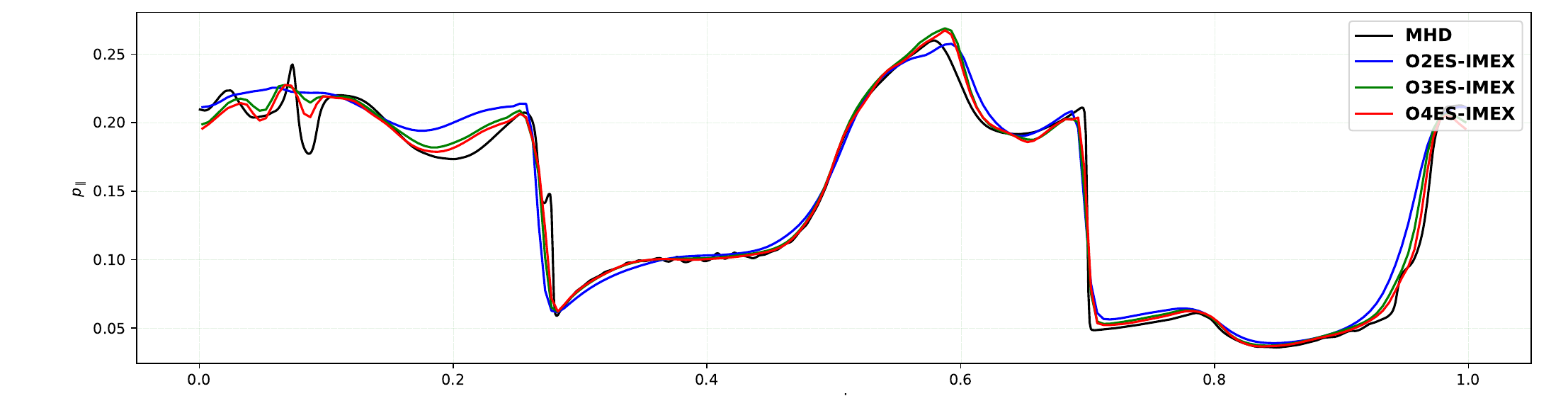}\label{fig:ot_cut_pll}}\\
		\subfloat[Isotropic case: Cut of pressure component $\per$ along $y = 0.3125$ with $200\times 200$ cells using \oti, \othi and \ofi schemes ]{\includegraphics[width=5.0in, height=1.5in]{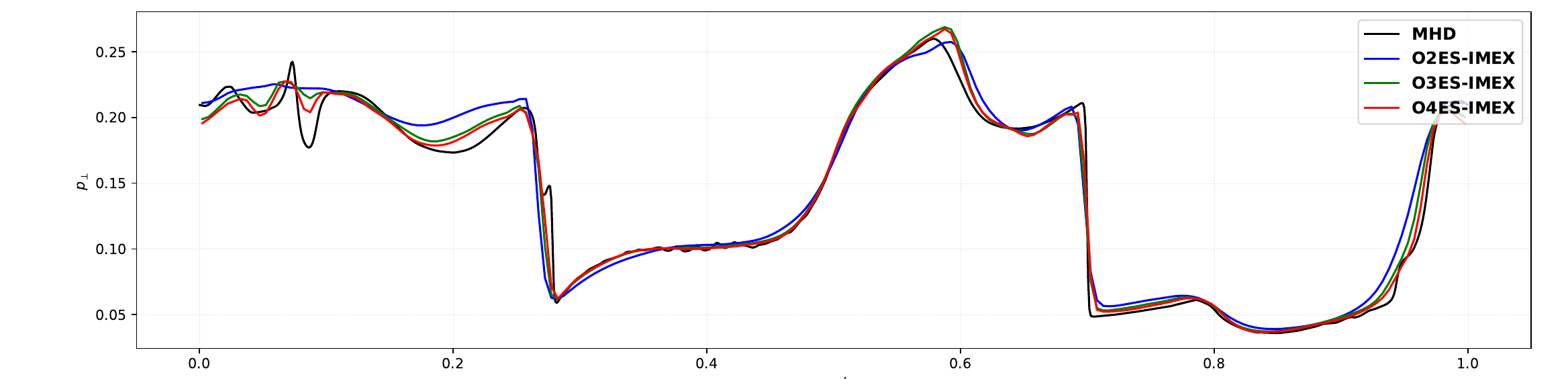}\label{fig:ot_cut_perp}}\\
		\subfloat[Isotropic case: Total entropy decay at each time step with $200\times 200$ cells using \oti, \othi and \ofi schemes ]{\includegraphics[width=1.5in, height=1.5in]{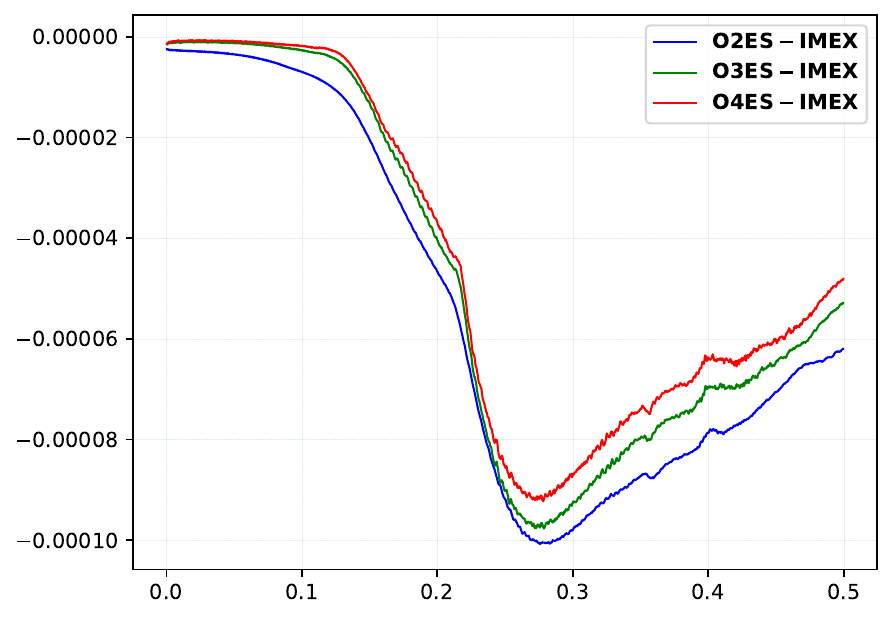}\label{fig:ot_entropy}}~
		\subfloat[Isotropic case: Total entropy decay at each time step with $400\times 400$ cells using \oti, \othi and \ofi schemes ]{\includegraphics[width=1.5in, height=1.5in]{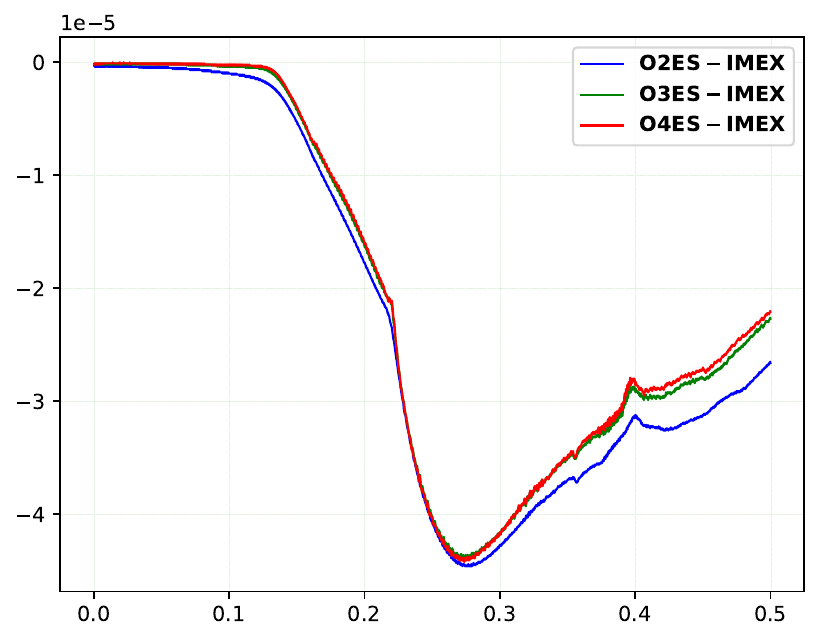}\label{fig:ot_entropy_400}}~
		\subfloat[Isotropic case: Total entropy decay at each time step with $800\times 800$ cells using \oti, \othi and \ofi schemes ]{\includegraphics[width=1.5in, height=1.5in]{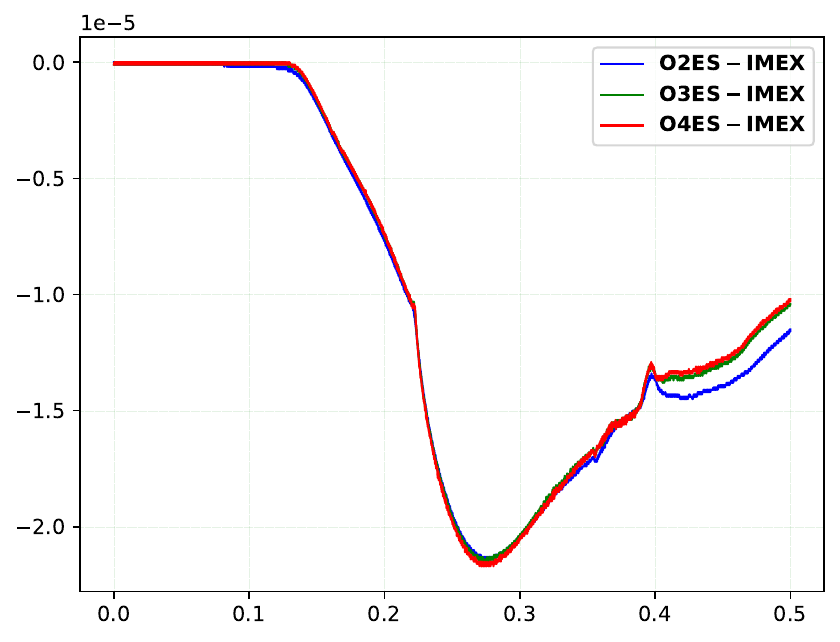}\label{fig:ot_entropy_800}}
		\caption{\textbf{\nameref{test:ot}}: Plots of pressure components and total entropy decay at each time step for IMEX schemes at time $t=0.5$.}
		\label{fig:ot_cut_and_ent}
	\end{center}
\end{figure}
\section{Conclusion}
\label{sec:conc}
In this article, we have proposed higher-order finite difference entropy stable numerical schemes for the CGL equations, which form a system of non-conservative hyperbolic equations. The proposed schemes are based on the novel reformulation of the equations, such that the non-conservative terms do not contribute to the entropy evolution. We then present the Godunov symmetrization of the conservative part. To design entropy stable schemes for the complete system, we first consider the entropy conservative numerical schemes for the conservative part and then extend the same to include non-conservative terms. To ensure entropy stability, we have considered the diffusion operator based on the entropy scaled right eigenvector of the symmetrized conservative part. The time discretization is performed using SSP-RK schemes for explicit schemes, which are used to calculate anisotropic solutions. For the isotropic case, we have considered a still source term. Hence, we have used ARK-IMEX schemes for time discretization. To test the entropy stability of the proposed numerical schemes, we have designed several test cases in one dimension by generalizing the one-dimensional MHD Riemann problems. We demonstrate that for a wide variety of test cases, the numerical entropy decay holds for each time step for both anisotropic and isotropic limits. For the isotropic limit, we show that solutions matched the MHD solution, as expected. For the two-dimensional test case, we have considered the Orzag-Tang vortex test case and present the isotropic case. We demonstrate that the schemes decay entropy only when discontinuities start to form in the solution. Furthermore, the solutions match the MHD solutions.

\section*{Acknowledgments}
The work of Harish Kumar is supported in parts by VAJRA grant No. VJR/2018/000129 by the Dept. of Science and Technology, Govt. of India
\section*{Conflict of interest}
The authors declare that they have no conflict of interest.
\section*{Data Availability Declaration}
Data will be made available on reasonable request.

\bibliographystyle{spmpsci}      
\bibliography{references}

\appendix
\section{Entropy scaled right eigenvectors for conservative part of the reformulated CGL system}\label{scaledrev}

In this section, we present expressions for the right eigenvectors and scaled right eigenvectors for the conservative system
\begin{equation}
	\df{\con}{t} + \mathcal{A}_x(\con)\df{\con}{x} = 0,~~~~~~~~\text{where,}~ \mathcal{A}_x(\con)=\frac{\p \textbf{f}_x}{\p \con} + \phi'(\evar)^\top B_x'(\con) \label{eq:jacobi}
\end{equation}
The right eigenvectors are described in Section A.1, while the expressions of scaled eigenvectors have been demonstrated in Section A.2. We will do all the analysis in the term of primitive variable $\textbf{W} = \{\rho,u_x,u_y,u_z,\pll,\per,B_x,B_y,B_z\}$.

\subsection{Eigenvalues and Right eigenvectors}\label{A.1}
To derive the eigenvalues and right eigenvectors, it is useful to transform the system \eqref{eq:jacobi} in terms of the primitive variables $\textbf{W}$. The set of eigenvalues $\Lambda_1$ of the jacobian matrix $\mathcal{A}_x$ are given as
\begin{align*}
	& \lambda_{1,2,3,4,5,6,7,8}=u_x,~ u_x, ~u_x\pm v_{ax},~u_x\pm c_f,~u_x\pm c_s,
\end{align*}
where,
\begin{align*}
	& v^2_{ax}= \frac{B_x^2}{\rho},~v^2_a=\frac{B^2}{\rho}, ~a^2=\frac{2P_\perp}{\rho},&\\&
	c^2_{f,s}=\frac{1}{2}\bigg[(v_a^2+a^2)\pm\sqrt{(v_a^2+a^2)^2-4v^2_{ax}a^2}\bigg].
\end{align*}
The corresponding right eigenvectors are 
\begin{gather*}R_{\lambda_1}=
	\begin{pmatrix}
		1\\
		0\\
		0\\
		0\\
		0\\
		0\\
		0\\
		0\\
	\end{pmatrix},R_{\lambda_2}=
	\begin{pmatrix}
		0\\
		0\\
		0\\
		0\\
		1\\
		0\\
		0\\
		0
	\end{pmatrix},R_{\lambda_{3,4}}=
	\begin{pmatrix}
		0\\
		0\\
		\pm{\beta_z}\\
		\mp{\beta_y}\\
		0\\
		0\\
		-{\beta_z}sgn(B_x)\sqrt{\rho}\\
		{\beta_y}sgn(B_x)\sqrt{\rho}\\
	\end{pmatrix},
\end{gather*}
\begin{gather*}R_{\lambda_{5,6}}=
	\begin{pmatrix}
		\alpha_f \rho\\
		\pm \alpha_f c_f\\
		\mp \alpha_s c_s \beta_y sgn(B_x)\\
		\mp \alpha_s c_s \beta_z sgn(B_x)\\
		\alpha_f \pll\\
		\alpha_f \rho a^2\\
		\alpha_s a\beta_y\sqrt{\rho}\\
		\alpha_s a\beta_z\sqrt{\rho}
	\end{pmatrix},R_{\lambda_{7,8}}=
	\begin{pmatrix}
		\alpha_s \rho\\
		\pm \alpha_s c_s\\
		\pm \alpha_f c_f \beta_y sgn(B_x)\\
		\pm \alpha_fc_f \beta_z sgn(B_x)\\
		\alpha_s \pll\\
		\alpha_s \rho a^2\\
		-\alpha_f a\beta_y\sqrt{\rho} \\
		-\alpha_f a\beta_z\sqrt{\rho} 
	\end{pmatrix},
\end{gather*}
where,
\begin{align*}
	\alpha_f^2=\frac{a^2-c_s^2}{c_f^2-c_s^2},~\alpha_s^2=\frac{c_f^2-a^2}{c_f^2-c_s^2},~
	\beta_y=\frac{B_y}{\sqrt{B_y^2+B_z^2}},~\beta_z=\frac{B_z}{\sqrt{B_y^2+B_z^2}}.
\end{align*}
Above all, eigenvectors are linearly independent. The above-defined right eigenvectors are singular in a variety of cases as listed by Balsara and Roe in ~\cite{balsara1999eigenstructure}. 
Similarly, in $y$-direction, the set of eigenvalues $\Lambda_2$ of the jacobian matrix $\frac{\p \f_{y}}{\p \con}$ are given as
\begin{align*}
	& \lambda_{1,2,3,4,5,6,7,8}=u_y,~ u_y, ~u_y\pm v_{ay},~u_y\pm c_f,~u_x\pm c_s,
\end{align*}
where,
\begin{align*}
	& v^2_{ay}= \frac{B_y^2}{\rho},~v^2_a=\frac{B^2}{\rho}, ~a^2=\frac{2P_\perp}{\rho},&\\&
	c^2_{f,s}=\frac{1}{2}\bigg[(v_a^2+a^2)\pm\sqrt{(v_a^2+a^2)^2-4v^2_{ay}a^2}\bigg].
\end{align*}
the corresponding right eigenvectors are 
\begin{gather*}R_{\lambda_1}=
	\begin{pmatrix}
		1\\
		0\\
		0\\
		0\\
		0\\
		0\\
		0\\
		0\\
	\end{pmatrix},R_{\lambda_2}=
	\begin{pmatrix}
		0\\
		0\\
		0\\
		0\\
		1\\
		0\\
		0\\
		0
	\end{pmatrix},R_{\lambda_{3,4}}=
	\begin{pmatrix}
		0\\
		\pm{\beta_z}\\
		0\\
		\mp{\beta_x}\\
		0\\
		0\\
		-{\beta_z}sgn(B_y)\sqrt{\rho}\\
		{\beta_x}sgn(B_y)\sqrt{\rho}\\
	\end{pmatrix},
\end{gather*}
\begin{gather*}R_{\lambda_{5,6}}=
	\begin{pmatrix}
		\alpha_f \rho\\
		\mp \alpha_s c_s \beta_x sgn(B_y)\\
		\pm \alpha_f c_f\\
		\mp \alpha_s c_s \beta_z sgn(B_y)\\
		\alpha_f \pll\\
		\alpha_f \rho a^2\\
		\alpha_s a\beta_x\sqrt{\rho}\\
		\alpha_s a\beta_z\sqrt{\rho}
	\end{pmatrix},R_{\lambda_{7,8}}=
	\begin{pmatrix}
		\alpha_s \rho\\
		\pm \alpha_f c_f \beta_x sgn(B_y)\\
		\pm \alpha_s c_s\\
		\pm \alpha_fc_f \beta_z sgn(B_y)\\
		\alpha_s \pll\\
		\alpha_s \rho a^2\\
		-\alpha_f a\beta_x\sqrt{\rho} \\
		-\alpha_f a\beta_z\sqrt{\rho} 
	\end{pmatrix},
\end{gather*}
where,
\begin{align*}
	\alpha_f^2=\frac{a^2-c_s^2}{c_f^2-c_s^2},~\alpha_s^2=\frac{c_f^2-a^2}{c_f^2-c_s^2},~
	\beta_x=\frac{B_x}{\sqrt{B_x^2+B_z^2}},~\beta_z=\frac{B_z}{\sqrt{B_x^2+B_z^2}}.
\end{align*}
\subsection{Entropy scaled right eigenvectors}
In this section, we have calculated the entropy scaled right eigenvectors for the case of $x$-direction. Consider the conservative part of the CGL system \eqref{eq:jacobi}. The eigenvalues and Right eigenvectors in terms of primitive of the Jacobian matrix $\mathcal{A}_1$ are described in \eqref{A.1}. The right eigenvector matrix $R^x$ for
the matrix $\mathcal{A}_1$ in terms of conservative variable is given by the relation
\begin{equation*}
	R^x = \frac{\p \con}{\p \textbf{W}} R_{\textbf{W}}^{x}
\end{equation*}
where, $\frac{\p \con}{\p \textbf{W}}$ is the Jacobian matrix for the change of variable, given by
%
\begin{align*}
	\frac{\p \con}{\p \textbf{W}}=
	\begin{pmatrix}
		1 & 0 & 0 & 0 & 0 & 0 & 0 & 0 & 0\\
		u_x & \rho & 0 & 0 & 0 & 0 & 0 & 0 & 0\\
		u_y & 0 & \rho & 0 & 0 & 0 & 0 & 0 & 0\\
		u_z & 0 & 0 & \rho & 0 & 0 & 0 & 0 & 0\\
		0 & 0 & 0 & 0 & 1 & 0 & 0 & 0 & 0\\
		\frac{\bu^2}{2} & \rho u_x & \rho u_y & \rho u_z & \frac{1}{2} & 1 & B_x & B_y & B_z\\
		0 & 0 & 0 & 0 & 0 & 0 & 1 & 0 & 0\\
		0 & 0 & 0 & 0 & 0 & 0 & 0 & 1 & 0\\
		0 & 0 & 0 & 0 & 0 & 0 & 0 & 0 & 1
	\end{pmatrix}.
\end{align*}
and the matrix $R_{\textbf{W}}^{x}$ is given by, 
\begin{align*}
	\begin{pmatrix}
		\alpha_f \rho  & \alpha_s \rho  & 0 & 1 & 0 & 0 & 0 & \alpha_s \rho  & \alpha_f \rho  \\
		-\alpha_f c_f & -\alpha_s c_s & 0 & 0 & 0 & 0 & 0 & \alpha_s c_s & \alpha_f c_f \\
		\alpha_s \beta_y c_s \mathcal{S}_{x} & -\alpha_f \beta_y c_f \mathcal{S}_{x} & -\beta_z & 0 & 0 & 0 & \beta_z & \alpha_f \beta_y c_f \mathcal{S}_{x}& -\alpha_s\beta_y c_s \mathcal{S}_{x}\\
		\alpha_s \beta_z c_s \mathcal{S}_{x} & -\alpha_f \beta_z c_f \mathcal{S}_{x} & \beta_y & 0 & 0 & 0 & -\beta_y & \alpha_f \beta_z c_f \mathcal{S}_{x}& -\alpha_s \beta_z c_s \mathcal{S}_{x} \\
		\alpha_f \pll & \alpha_s \pll & 0 & 0 & 0 & 1 & 0 & \alpha_s \pll & \alpha_f \pll \\
		a^2 \alpha_f \rho  & a^2 \alpha_s \rho  & 0 & 0 & 0 & 0 & 0 & a^2 \alpha_s \rho  & a^2 \alpha_f \rho  \\
		0 & 0 & 0 & 0 & 1 & 0 & 0 & 0 & 0 \\
		a \alpha_s \beta_y \sqrt{\rho} & -a \alpha_f \beta_y \sqrt{\rho} & -\beta_z \sqrt{\rho} \mathcal{S}_{x} & 0 & 0 & 0 & -\beta_z \sqrt{\rho} \mathcal{S}_{x} & -a \alpha_f \beta_y \sqrt{\rho} & a \alpha_s \beta_y \sqrt{\rho} \\
		a \alpha_s \beta_z \sqrt{\rho} & -a \alpha_f \beta_z \sqrt{\rho} & \beta_y \sqrt{\rho} \mathcal{S}_{x} & 0 & 0 & 0 & \beta_y \sqrt{\rho} \mathcal{S}_{x} & -a \alpha_f \beta_z \sqrt{\rho} & a \alpha_s \beta_z \sqrt{\rho} \\
	\end{pmatrix}.
\end{align*}
where, $\mathcal{S}_{x} = sgn(B_x)$. We need to find a scaling matrix $T^x$ such that the scaled right eigenvector matrix $\Tilde{R}^x=R^x T^x$ satisfies
\begin{equation}
	\frac{\p \con}{\p \evar} = \Tilde{R}^x \Tilde{R}^{x\top}
\end{equation}
where $\evar$ is the entropy variable vector as in Eqn. \eqref{eq:envar}. We follow the Barth scaling process \cite{barth1999numerical} to scale the right eigenvectors. The scaling matrix $T^x$ is the square root of $Y^x$, where $Y^x$ has the expression
\begin{equation*}
	Y^x = (R_{\textbf{W}}^{x})^{-1} \frac{\p \textbf{W}}{\p \evar}\left(\frac{\p \con}{\p \textbf{W}}\right)^{-\top}(R_{\textbf{W}}^{x}) ^{-\top}
\end{equation*}
which results in
\begin{align*}
	Y^x = \begin{pmatrix}
		\frac{1}{8\rho} & 0 & 0 & 0 & 0 & 0 & 0 & 0 & 0\\
		0 & \frac{1}{8\rho} & 0 & 0 & 0 & 0 & 0 & 0 & 0\\
		0 & 0 & \frac{\per}{4\rho^2} & 0 & 0 & 0 & 0 & 0 & 0\\
		0 & 0 & 0 & \frac{\rho}{4} &0 & \frac{\pll}{4} & 0 & 0 & 0 \\
		0 & 0 & 0 & 0 & \frac{\per}{2\rho} & 0 & 0 & 0 & 0\\
		0 & 0 & 0 & \frac{\pll}{4} & 0&\frac{5 \pll^2}{4\rho} & 0 & 0 & 0 \\
		0 & 0 & 0 & 0 & 0 & 0 & \frac{\per}{4\rho^2}& 0 & 0\\
		0 & 0 & 0 & 0 & 0 & 0 & 0 & \frac{1}{8\rho} & 0\\ 
		0 & 0 & 0 & 0 & 0 & 0 & 0 & 0 & \frac{1}{8\rho}
	\end{pmatrix}.
\end{align*}
Then matrix $T^x$ is,
\begin{align*}
	\begin{pmatrix}
		\frac{1}{2\sqrt{2\rho}} & 0 & 0 & 0 & 0 & 0 & 0 & 0 & 0\\
		0 & \frac{1}{2\sqrt{2\rho}} & 0 & 0 & 0 & 0 & 0 & 0 & 0\\
		0 & 0 & \frac{\sqrt{\per}}{2\rho} & 0 & 0 & 0 & 0 & 0 & 0\\
		0 & 0 & 0 & \frac{\sqrt{\rho}(2 \pll + \rho )}{2 \sqrt{5 \pll^{2} + 4 \pll\rho + \rho^2}} & 0 &\frac{\pll\sqrt{\rho}}{2 \sqrt{5 \pll^{2} + 4 \pll\rho + \rho^2}} &  0 & 0 & 0 \\
		0 & 0 & 0 & 0 & \sqrt{\frac{\per}{2\rho}} & 0 & 0 & 0 & 0\\
		0 & 0 & 0 & \frac{\pll\sqrt{\rho}}{2 \sqrt{5 \pll^{2} + 4 \pll\rho + \rho^2}} & 0 & \frac{\pll (5 \pll + 2 \rho)}{2 \sqrt{\rho (5 \pll^{2} + 4 \pll\rho + \rho^2)}}& 0 & 0 & 0 \\
		0 & 0 & 0 & 0 & 0 & 0 & \frac{\sqrt{\per}}{2\rho}& 0 & 0\\
		0 & 0 & 0 & 0 & 0 & 0 & 0 & \frac{1}{2\sqrt{2\rho}} & 0\\ 
		0 & 0 & 0 & 0 & 0 & 0 & 0 & 0 & \frac{1}{2\sqrt{2\rho}}
	\end{pmatrix}.
\end{align*}
\begin{remark}
	For finding the square root of $3\times 3$ matrix 
	\begin{align*}
		M = \begin{pmatrix}
			a & 0 & b\\
			0 & e & 0 \\
			c & 0 & d 
		\end{pmatrix}
	\end{align*}
	where $a~,b~,c$, $d$ and $e$ are real or complex numbers. we apply the formula,
	\begin{align*}
		R = \begin{pmatrix}
			\frac{a+s}{t} & 0 & \frac{b}{t}\\
			0 & \sqrt{e} & 0\\
			\frac{c}{t} & 0 & \frac{d+s}{t}
		\end{pmatrix}.
	\end{align*}
	where, $\delta = ad-bc$, $\tau= a+d$ and $s=\sqrt{\delta}$, $t=\sqrt{\tau+2s}$. We apply this formula if $t \neq 0$.
\end{remark}
\begin{remark}
	We proceed similarly in the $y$-direction, the eigenvalues and Right eigenvectors in terms of primitive of the Jacobian matrix $\frac{\p \f_{2}}{\p \con}$ describe in \eqref{A.1}. The right eigenvector matrix is given by the relation
	\begin{equation*}
		R^y = \frac{\p \con}{\p \textbf{W}} R_{\textbf{W}}^{y}
	\end{equation*}
	where, $R_{\textbf{W}}^{y}$ is given below
	\begin{align*}
		\begin{pmatrix}
			\alpha_f \rho  & \alpha_s \rho  & 0 & 1 & 0 & 0 & 0 & \alpha_s \rho  & \alpha_f \rho  \\
			\alpha_s \beta_x c_s \mathcal{S}_{y} & -\alpha_f \beta_x c_f \mathcal{S}_{y} & -\beta_z & 0 & 0 & 0 & \beta_z & \alpha_f \beta_x c_f \mathcal{S}_{y}& -\alpha_s\beta_x c_s \mathcal{S}_{y}\\
			-\alpha_f c_f & -\alpha_s c_s & 0 & 0 & 0 & 0 & 0 & \alpha_s c_s & \alpha_f c_f \\
			\alpha_s \beta_z c_s \mathcal{S}_{y} & -\alpha_f \beta_z c_f \mathcal{S}_{y} & \beta_x & 0 & 0 & 0 & -\beta_x & \alpha_f \beta_z c_f \mathcal{S}_{y}& -\alpha_s \beta_z c_s \mathcal{S}_{y}\\
			\alpha_f \pll & \alpha_s \pll & 0 & 0 & 0 & 1 & 0 & \alpha_s \pll & \alpha_f \pll \\
			a^2 \alpha_f \rho  & a^2 \alpha_s \rho  & 0 & 0 & 0 & 0 & 0 & a^2 \alpha_s \rho  & a^2 \alpha_f \rho  \\
			a \alpha_s \beta_x \sqrt{\rho} & -a \alpha_f \beta_x \sqrt{\rho} & -\beta_z \sqrt{\rho} \mathcal{S}_{y} & 0 & 0 & 0 & -\beta_z \sqrt{\rho} \mathcal{S}_{y} & -a \alpha_f \beta_x \sqrt{\rho} & a \alpha_s \beta_x \sqrt{\rho} \\
			0 & 0 & 0 & 0 & 1 & 0 & 0 & 0 & 0 \\
			a \alpha_s \beta_z \sqrt{\rho} & -a \alpha_f \beta_z \sqrt{\rho} & \beta_x \sqrt{\rho} \mathcal{S}_{y} & 0 & 0 & 0 & \beta_x \sqrt{\rho} \mathcal{S}_{y} & -a \alpha_f \beta_z \sqrt{\rho} & a \alpha_s \beta_z \sqrt{\rho} \\
		\end{pmatrix}
	\end{align*}
	where, $\mathcal{S}_{y} = sgn(B_y)$. Accordingly, we obtain the scaling matrix $T^y$ as,
	\begin{align*}
		\begin{pmatrix}
			\frac{1}{2\sqrt{2\rho}} & 0 & 0 & 0 & 0 & 0 & 0 & 0 & 0\\
			0 & \frac{1}{2\sqrt{2\rho}} & 0 & 0 & 0 & 0 & 0 & 0 & 0\\
			0 & 0 & \frac{\sqrt{\per}}{2\rho} & 0 & 0 & 0 & 0 & 0 & 0\\
			0 & 0 & 0 & \frac{\sqrt{\rho}(2 \pll + \rho )}{2 \sqrt{5 \pll^{2} + 4 \pll\rho + \rho^2}} & 0 &\frac{\pll\sqrt{\rho}}{2 \sqrt{5 \pll^{2} + 4 \pll\rho + \rho^2}} &  0 & 0 & 0 \\
			0 & 0 & 0 & 0 & \sqrt{\frac{\per}{2\rho}} & 0 & 0 & 0 & 0\\
			0 & 0 & 0 & \frac{\pll\sqrt{\rho}}{2 \sqrt{5 \pll^{2} + 4 \pll\rho + \rho^2}} & 0 & \frac{\pll (5 \pll + 2 \rho)}{2 \sqrt{\rho (5 \pll^{2} + 4 \pll\rho + \rho^2)}}& 0 & 0 & 0 \\
			0 & 0 & 0 & 0 & 0 & 0 & \frac{\sqrt{\per}}{2\rho}& 0 & 0\\
			0 & 0 & 0 & 0 & 0 & 0 & 0 & \frac{1}{2\sqrt{2\rho}} & 0\\ 
			0 & 0 & 0 & 0 & 0 & 0 & 0 & 0 & \frac{1}{2\sqrt{2\rho}}
		\end{pmatrix}.
	\end{align*}
\end{remark}

\subsection{Entropy scaled right eigenvectors in $x$ and $y$ direction}
In $x$ direction, the entropy-scaled right eigenvectors corresponding to the above eigenvalues in the term of primitive variables are the following,
\begin{gather*}\tilde{R}_{\lambda_1}=
	\begin{pmatrix}
		\frac{\sqrt{\rho}(2 \pll + \rho )}{2 \sqrt{5 \pll^{2} + 4 \pll\rho + \rho^2}}\\
		0\\
		0\\
		0\\
		\frac{\pll\sqrt{\rho}}{2 \sqrt{5 \pll^{2} + 4 \pll\rho + \rho^2}}\\
		0\\
		0\\
		0\\
		0
	\end{pmatrix},\tilde{R}_{\lambda_2}=
	\begin{pmatrix}
		0\\
		0\\
		0\\
		0\\
		0\\
		0\\
		\sqrt{\frac{\per}{2 \rho}}\\
		0\\
		0
	\end{pmatrix},\tilde{R}_{\lambda_3}=
	\begin{pmatrix}
		\frac{\pll\sqrt{\rho}}{2 \sqrt{5 \pll^{2} + 4 \pll\rho + \rho^2}}\\
		0\\
		0\\
		0\\
		\frac{\pll (5 \pll + 2 \rho)}{2 \sqrt{\rho (5 \pll^{2} + 4 \pll\rho + \rho^2)}}\\
		0\\
		0\\
		0\\
		0
	\end{pmatrix},
\end{gather*}
\begin{gather*}\tilde{R}_{\lambda_{4,5}}=
	\frac{1}{2}\begin{pmatrix}
		0\\
		0\\
		\pm\frac{\sqrt{\per}}{\rho}{\beta_z}\\
		\mp\frac{\sqrt{\per}}{\rho}{\beta_y}\\
		0\\
		0\\
		0\\
		-\sqrt{\frac{\per}{\rho}}{\beta_z}\\
		\sqrt{\frac{\per}{\rho}}{\beta_y}\\
	\end{pmatrix},\tilde{R}_{\lambda_{6,7}}=
	\frac{1}{2\sqrt{2}}\begin{pmatrix}
		\alpha_f \sqrt{\rho}\\
		\pm \frac{\alpha_f c_f}{\sqrt{\rho}}\\
		\mp \frac{\alpha_s c_s \beta_y}{\sqrt{\rho}}\\
		\mp \frac{\alpha_s c_s \beta_z}{\sqrt{\rho}}\\
		\alpha_f \frac{\pll}{\sqrt{\rho}}\\
		\alpha_f \sqrt{\rho} a^2\\
		0\\
		\alpha_s a\beta_y\\
		\alpha_s a\beta_z
	\end{pmatrix},\tilde{R}_{\lambda_{8,9}}=
	\frac{1}{2\sqrt{2}}\begin{pmatrix}
		\alpha_s \sqrt{\rho}\\
		\pm \frac{\alpha_s c_s}{\sqrt{\rho}}\\
		\pm \frac{\alpha_f c_f \beta_y}{\sqrt{\rho}}\\
		\pm \frac{\alpha_f c_f \beta_z}{\sqrt{\rho}}\\
		\alpha_s \frac{\pll}{\sqrt{\rho}}\\
		\alpha_s \sqrt{\rho} a^2\\
		0\\
		-\alpha_f a\beta_y\\
		-\alpha_f a\beta_z
	\end{pmatrix},
\end{gather*}
In $y$-direction, the entropy-scaled right eigenvectors corresponding to the above eigenvalues in the term of primitive variables are the following,
\begin{gather*}\tilde{R}_{\lambda_1}=
	\begin{pmatrix}
		\frac{\sqrt{\rho}(2 \pll + \rho )}{2 \sqrt{5 \pll^{2} + 4 \pll\rho + \rho^2}}\\
		0\\
		0\\
		0\\
		\frac{\pll\sqrt{\rho}}{2 \sqrt{5 \pll^{2} + 4 \pll\rho + \rho^2}}\\
		0\\
		0\\
		0\\
		0
	\end{pmatrix},\tilde{R}_{\lambda_2}=
	\begin{pmatrix}
		0\\
		0\\
		0\\
		0\\
		0\\
		0\\
		0\\
		\sqrt{\frac{\per}{2 \rho}}\\
		0
	\end{pmatrix},\tilde{R}_{\lambda_3}=
	\begin{pmatrix}
		\frac{\pll\sqrt{\rho}}{2 \sqrt{5 \pll^{2} + 4 \pll\rho + \rho^2}}\\
		0\\
		0\\
		0\\
		\frac{\pll (5 \pll + 2 \rho)}{2 \sqrt{\rho (5 \pll^{2} + 4 \pll\rho + \rho^2)}}\\
		0\\
		0\\
		0\\
		0
	\end{pmatrix},
\end{gather*}
\begin{gather*}\tilde{R}_{\lambda_{4,5}}=
	\frac{1}{2}\begin{pmatrix}
		0\\
		\pm\frac{\sqrt{\per}}{\rho}{\beta_z}\\
		0\\
		\mp\frac{\sqrt{\per}}{\rho}{\beta_x}\\
		0\\
		0\\
		-\sqrt{\frac{\per}{\rho}}{\beta_z}\\
		0\\
		\sqrt{\frac{\per}{\rho}}{\beta_x}\\
	\end{pmatrix},\tilde{R}_{\lambda_{6,7}}=
	\frac{1}{2\sqrt{2}}\begin{pmatrix}
		\alpha_f \sqrt{\rho}\\
		\mp \frac{\alpha_s c_s \beta_x}{\sqrt{\rho}}\\
		\pm \frac{\alpha_f c_f}{\sqrt{\rho}}\\
		\mp \frac{\alpha_s c_s \beta_z}{\sqrt{\rho}}\\
		\alpha_f \frac{\pll}{\sqrt{\rho}}\\
		\alpha_f \sqrt{\rho} a^2\\
		\alpha_s a\beta_x\\
		0\\
		\alpha_s a\beta_z
	\end{pmatrix},\tilde{R}_{\lambda_{8,9}}=
	\frac{1}{2\sqrt{2}}\begin{pmatrix}
		\alpha_s \sqrt{\rho}\\
		\pm \frac{\alpha_f c_f \beta_x}{\sqrt{\rho}}\\
		\pm \frac{\alpha_s c_s}{\sqrt{\rho}}\\
		\pm \frac{\alpha_f c_f \beta_z}{\sqrt{\rho}}\\
		\alpha_s \frac{\pll}{\sqrt{\rho}}\\
		\alpha_s \sqrt{\rho} a^2\\
		-\alpha_f a\beta_x\\
		0\\
		-\alpha_f a\beta_z
	\end{pmatrix},
\end{gather*}
\section{Proof of $\evar^\top\cdot \bc_x(\con) = \evar^\top\cdot \bc_y(\con) = 0$}\label{A.4}
The entropy variable $\evar:=\frac{\p \ent}{\p \con}$ is given by
\begin{align*}
	\evar&=\Bigg\{5-s-\beta_\perp \bu^2,~2\beta_\perp\bu,~-\beta_\parallel+\beta_\perp,~-2\beta_\perp,~2\beta_\perp\B\Bigg\}^\top
\end{align*} 
where $\beta_\perp=\frac{\rho}{\per},$ and $\beta_\parallel=\frac{\rho}{\pll}$. Now the dot product of $\evar^\top$ and first coloum of $\bc_x(\con)$ is:
\begin{align*}
	\evar^\top\cdot& \left\{0,-\frac{\hb_x^2 \bu^2}{2},-\frac{\hb_x\hb_y \bu^2}{2},-\frac{\hb_x\hb_z \bu^2}{2},-\frac{2\pll \hb_x}{\rho}(\bhat\cdot\bu), \Upsilon_1^x,0,0,0\right\}=\\
	&= \left(2\frac{\rho}{\per} u_x\right)\left(-\frac{\hb_x^2 \bu^2}{2}\right)+\left(2\frac{\rho}{\per} u_y\right)\left(-\frac{\hb_x\hb_y \bu^2}{2}\right)+\left(2\frac{\rho}{\per}u_z\right)\left(-\frac{\hb_x\hb_z \bu^2}{2}\right)\\
	&\hspace{3.5cm}+\left(\frac{\rho(\pll-\per)}{\pll\per}\right)\left(-\frac{2\pll \hb_x}{\rho}(\bhat\cdot\bu)\right)+\left(-2\frac{\rho}{\per}\right)(\Upsilon_1^x)\\
	&=\left(2\frac{\rho}{\per} u_x\right)\left(-\frac{\hb_x^2 \bu^2}{2}\right)+\left(2\frac{\rho}{\per} u_y\right)\left(-\frac{\hb_x\hb_y \bu^2}{2}\right)+\left(2\frac{\rho}{\per}u_z\right)\left(-\frac{\hb_x\hb_z \bu^2}{2}\right)\\
	&\hspace{3.5cm}+\left(\frac{\rho(\pll-\per)}{\pll\per}\right)\left(-\frac{2\pll \hb_x}{\rho}(\bhat\cdot\bu)\right)\\
	&\hspace{3.5cm}+\left(-2\frac{\rho}{\per}\right)\left(-\hb_x(\hb\cdot\bu)\frac{\bu^2}{2}-\frac{(\pll - \per)\hb_x(\bhat\cdot\bu)}{\rho}\right)\\
	&=0
\end{align*}
Now the dot product of $\evar^\top$ and second coloum of $\bc_x(\con)$ is:
\begin{align*}
	\evar^\top\cdot &\left\{0,\hb_x^2 u_x,\hb_x\hb_y u_x,\hb_x\hb_z u_x,\frac{2\pll \hb_x}{\rho}\hb_x,\Upsilon_2^x,0,0,0\right\} =\\
	&=\left(2\frac{\rho}{\per} u_x\right)(\hb_x^2 u_x) + \left(2\frac{\rho}{\per} u_y\right)(\hb_x\hb_y u_x) +\left(2\frac{\rho}{\per}u_z\right)(\hb_x\hb_z u_x) \\
	&+ \left(\frac{\rho(\pll-\per)}{\pll\per}\right)\left(\frac{2\pll \hb_x}{\rho}\hb_x\right)
	+\left(-2\frac{\rho}{\per}\right)\left(\hb_x(\bhat\cdot\bu)u_x+\frac{(\pll-\per)\hb^2_x}{\rho}\right)\\
	&=0
\end{align*}
Now the dot product of $\evar^\top$ and third coloum of $\bc_x(\con)$ is:
\begin{align*}
	\evar^\top\cdot& \left\{0,\hb_x^2 u_y,\hb_x\hb_y u_y,\hb_x\hb_z u_y,\frac{2\pll \hb_x}{\rho}\hb_y,\Upsilon_3^x,0,0,0\right\} =\\
	&=\left(2\frac{\rho}{\per} u_x\right)(\hb_x^2 u_y) + \left(2\frac{\rho}{\per} u_y\right)(\hb_x\hb_y u_y) +\left(2\frac{\rho}{\per}u_z\right)(\hb_x\hb_z u_y) \\
	&+ \left(\frac{\rho(\pll-\per)}{\pll\per}\right)\left(\frac{2\pll \hb_x}{\rho}\hb_y\right) + \left(-2\frac{\rho}{\per}\right)\left(\hb_x(\bhat\cdot\bu)u_y+\frac{(\pll - \per)\hb_x\hb_y}{\rho}\right)\\
	&=0
\end{align*}
Now the dot product of $\evar^\top$ and fourth coloum of $\bc_x(\con)$ is:
\begin{align*}
	\evar^\top\cdot& \left\{0,\hb_x^2 u_z,\hb_x\hb_y u_z,\hb_x\hb_z u_z,\frac{2\pll \hb_x}{\rho}\hb_z,\Upsilon_4^x,0,0,0\right\} =\\
	&=\left(2\frac{\rho}{\per} u_x\right)(\hb_x^2 u_z) + \left(2\frac{\rho}{\per} u_y\right)(\hb_x\hb_y u_z) +\left(2\frac{\rho}{\per}u_z\right)(\hb_x\hb_z u_z) \\
	&+ \left(\frac{\rho(\pll-\per)}{\pll\per}\right)\left(\frac{2\pll \hb_x}{\rho}\hb_z\right) + \left(-2\frac{\rho}{\per}\right)\left(\hb_x(\bhat\cdot\bu)u_z+\frac{(\pll - \per)\hb_x\hb_z}{\rho}\right)\\
	&=0
\end{align*}
Now the dot product of $\evar^\top$ and fifth coloum of $\bc_x(\con)$ is:
\begin{align*}
	\evar^\top\cdot& \left\{0, \frac{3}{2}\hb_x^2, \frac{3}{2}\hb_x\hb_y, \frac{3}{2}\hb_x\hb_z ,0, \frac{3}{2}\hb_x(\hb\cdot\bu),0,0,0\right\} =\\
	&=\left(2\frac{\rho}{\per} u_x\right)(\frac{3}{2}\hb_x^2) + \left(2\frac{\rho}{\per} u_y\right)(\frac{3}{2}\hb_x\hb_y) +\left(2\frac{\rho}{\per}u_z\right)(\frac{3}{2}\hb_x\hb_z) \\
	&+ \left(-2\frac{\rho}{\per}\right)\left(\frac{3}{2}\hb_x(\bhat\cdot\bu)\right)\\
	&=0
\end{align*}
Now the dot product of $\evar^\top$ and sixth coloum of $\bc_x(\con)$ is:
\begin{align*}
	\evar^\top\cdot& \left\{0, -\hb_x^2, -\hb_x\hb_y, -\hb_x\hb_z ,0, -\hb_x(\hb\cdot\bu),0,0,0\right\} =\\
	&=\left(2\frac{\rho}{\per} u_x\right)(-\hb_x^2) + \left(2\frac{\rho}{\per} u_y\right)(-\hb_x\hb_y) +\left(2\frac{\rho}{\per}u_z\right)(-\hb_x\hb_z)\\
	&+ \left(-2\frac{\rho}{\per}\right)\left(-\hb_x(\bhat\cdot\bu)\right)\\
	&=0
\end{align*}
The seventh coloum, 8th coloum and ninth coloum of $\bc_x(\con)$ denoted by $C_7,~C_8$ and $C_9$ is given below,
\begin{gather*}C_7=
	\begin{pmatrix}
		0 \\ \hb_x^2 B_x+\frac{2\Gamma_x\hb_x}{|\B|} \\ \hb_x\hb_y B_x+\frac{\Gamma_x\hb_y - \DP b_{yxx}}{|\B|} \\ \hb_x\hb_z B_x+\frac{\Gamma_x\hb_z - \DP b{zxx}}{|\B|} \\
		0 \\ \hb_x(\bhat\cdot\bu)B_x + \Theta^x_1 \\ 0 \\ 0 \\ 0
	\end{pmatrix},C_8=
	\begin{pmatrix}
		0 \\ \hb_x^2 B_y - \frac{2\DP b_{yxx}}{|\B|} \\ b_x b_y B_y+\frac{\Gamma_y b_x - \DP b_{xyy}}{|\B|} \\ \hb_x\hb_z B_y-\frac{2\DP b_{xyz}}{|\B|} \\ 0 \\ \hb_x(\bhat\cdot\bu)B_y + \Theta^x_2 \\ 0 \\ 0 \\ 0
	\end{pmatrix},
\end{gather*}
\begin{gather*}C_9=
	\begin{pmatrix}
		0 \\ \hb_x^2 B_z-\frac{2\DP b_{zxx}}{|\B|} \\ \hb_x\hb_y B_z-\frac{2\DP b_{xyz}}{|\B|} \\ \hb_x\hb_z B_z + \frac{\Gamma_z\hb_x - \DP b_{xzz}}{|\B|} \\ 0 \\ \hb_x(\bhat\cdot\bu)B_z + \Theta^x_3 \\ 0 \\ 0 \\ 0
	\end{pmatrix}.
\end{gather*}
Now the dot product of $\evar^\top$ and seventh coloum of $\bc_x(\con)$ is:
\begin{align*}
	\evar^\top\cdot C_7 & = \left(2\frac{\rho}{\per} u_x\right)\left(\hb_x^2 B_x+\frac{2(\pll - \per)\hb_x(1-\hb_x^2)}{|\B|}\right) \\
	&+ \left(2\frac{\rho}{\per} u_y\right)\left(\hb_x\hb_y B_x+\frac{(\pll - \per)\hb_y(1-\hb_x^2)}{|\B|}- \frac{(\pll - \per)\hb_y\hb_x^2}{|\B|}\right)\\
	&+\left(2\frac{\rho}{\per}u_z\right)\left(\hb_x\hb_z B_x+\frac{(\pll - \per)\hb_z(1-\hb_x^2)}{|\B|} - \frac{(\pll - \per)\hb_z\hb_x^2}{|\B|}\right)\\
	&+\left(-2\frac{\rho}{\per}\right)\Bigg(\hb_x(\bhat\cdot\bu)B_x + \left\{(\pll - \per)(\bhat\cdot\bu)+(\pll - \per)\hb_x u_x\right\}\Bigg(\frac{1-\hb^2_x}{|\B|}\Bigg)\\
	&  \hspace{5.0cm}-(\pll - \per)\frac{\hb^2_x \hb_y u_y}{|\B|}- (\pll - \per)\frac{\hb^2_x \hb_z u_z}{|\B|}\Bigg)\\
	&=0
\end{align*}
Now the dot product of $\evar^\top$ and 8th coloum of $\bc_x(\con)$ is:
\begin{align*}
	\evar^\top\cdot C_8 & = \left(2\frac{\rho}{\per} u_x\right)\left(\hb_x^2 B_y-\frac{2(\pll - \per)\hb_y\hb_x^2}{|\B|}\right) \\
	&+ \left(2\frac{\rho}{\per} u_y\right)\left(\hb_x\hb_y B_y+\frac{(\pll - \per)\hb_x(1-\hb_y^2)}{|\B|} - \frac{(\pll - \per)\hb_x\hb_y^2}{|\B|}\right)\\
	&+\left(2\frac{\rho}{\per}u_z\right)\left(\hb_x\hb_z B_y-\frac{2(\pll - \per)\hb_x\hb_y\hb_z}{|\B|}\right)\\
	&+\left(-2\frac{\rho}{\per}\right)\Bigg(\hb_x(\bhat\cdot\bu)B_y + (\pll - \per)\hb_x u_y \Bigg(\frac{1-\hb^2_y}{|\B|}\Bigg) - (\pll - \per)\frac{\hb_x \hb_y \hb_z u_z}{|\B|} \\
	&\hspace{5.0cm}  -\{(\pll - \per)(\bhat\cdot\bu)+(\pll - \per)\hb_x u_x\}\frac{\hb_x \hb_y}{|\B|}\Bigg)\\
	&=0
\end{align*}
Now the dot product of $\evar^\top$ and ninth coloum of $\bc_x(\con)$ is:
\begin{align*}
	\evar^\top\cdot C_9 &= \left(2\frac{\rho}{\per} u_x\right)\left(\hb_x^2 B_z-\frac{2(\pll - \per)\hb_z\hb_x^2}{|\B|}\right) \\
	&+ \left(2\frac{\rho}{\per} u_y\right)\left(\hb_x\hb_y B_z-\frac{2(\pll - \per)\hb_x\hb_y\hb_z}{|\B|}\right)\\
	&+\left(2\frac{\rho}{\per}u_z\right)\left(\hb_x\hb_z B_z+\frac{(\pll - \per)\hb_x(1-\hb_z^2)}{|\B|} - \frac{(\pll - \per)\hb_x\hb_z^2}{|\B|}\right)\\
	&+\left(-2\frac{\rho}{\per}\right)\Bigg(\hb_x(\bhat\cdot\bu)B_z + (\pll - \per)\hb_x u_z \Bigg(\frac{1-\hb^2_z}{|\B|}\Bigg) - (\pll - \per)\frac{\hb_x \hb_y \hb_z u_y}{|\B|}\\
	& \hspace{5.0cm} -\{(\pll - \per)(\bhat\cdot\bu)+(\pll - \per)\hb_x u_x\}\frac{\hb_x \hb_z}{|\B|}\Bigg)\\
	&=0
\end{align*}
Similarly we can proof that $\evar^\top\cdot \bc_y(\con) = 0$ for $y$-direction.
\section{Non-symmetrizability of CGL System}\label{A.5}
In this section, we have discussed the symmetrizability of the following CGL equations. Consider CGL equations in one dimension as follows:
\begin{equation}\label{eq:50}
	\frac{\p \con}{\p t}+\frac{\p \f_{x}}{\p x} + \bc_{x}(\con)\frac{\p \con}{\p x} =0.
\end{equation}
where $\con,~\f_{x}$ and $\bc_{x}(\con)$ are defined in \eqref{sec:reformulation}. Follow~\cite{yadav2023entropy}, the CGL equations are said to be symmetrizable if the change of variable $\con \rightarrow \evar$ applied to \eqref{eq:50} and can be written as:
$$
\frac{\p \con}{\p \evar}\frac{\p \evar}{\p {t}}+\left(\frac{\p \textbf{f}_x}{\p \con} +\bc_{x}(\con) \right)\frac{\p \con}{\p \evar}\frac{\p \evar}{\p {x}}=0.
$$
the matrix $\frac{\p \con}{\p \evar}$ is a symmetric, positive definite matrix and  $\aleph(\con) = \left(\frac{\p \textbf{f}_x}{\p \con} +\bc_{x}(\con) \right)\frac{\p \con}{\p \evar}$ is symmetric matrix. For the CGL equation \eqref{eq:50}, we calculate matrix $\aleph(\con)$ to examine its symmetry and conclude that CGL equations are non-symmetrizable. The resultant matrix $\aleph(\con) -\aleph(\con)^\top$ is given below:\\
\begin{landscape}
	\begin{align*}
		\begin{pmatrix}
			0 & -\frac{\DP b_x^2}{2} & -\frac{\DP b_x b_y}{2} & -\frac{\DP b_x b_z}{2} & 0 & -\frac{\DP b_x (\bhat\cdot\bu)}{2} & 0 & 0 & 0\\
			\frac{\DP b_x^2}{2} & 0 & \frac{\DP b_x (b_x u_y - b_y u_x)}{2} & \frac{\DP b_x (b_x u_z - b_z u_x)}{2} & \frac{3 \DP\pll b_x^2}{2\rho} & -\frac{\alpha_{xyz}}{4\rho} & -\frac{\per\omega_1}{2\rho} & -\frac{\DP \per b_y b_x^2}{\rho |\B|} & -\frac{\DP \per b_z b_x^2}{\rho |\B|}\\
			\frac{\DP b_x b_y}{2} & \frac{\DP b_x (b_y u_x - b_x u_y)}{2} & 0 & \frac{\DP b_x (b_y u_z - b_z u_y)}{2} & \frac{3 \DP\pll b_x b_y}{2\rho} & -\frac{\alpha_{yxz}}{4\rho} & -\frac{\per\omega_2}{2\rho} & \frac{\DP \per b_x (1-2b_y^2)}{2\rho |\B|} & -\frac{\DP \per b_x b_y b_z}{\rho |\B|}\\
			\frac{\DP b_x b_z}{2} & \frac{\DP b_x (b_z u_x - b_x u_z)}{2} & \frac{\DP b_x (b_z u_y - b_y u_z)}{2} & 0 & \frac{3 \DP\pll b_x b_z}{2\rho} & -\frac{\alpha_{zxy}}{4\rho} & -\frac{\per\omega_3}{2\rho} & -\frac{\DP \per b_x b_y b_z}{\rho |\B|} & \frac{\DP \per b_x (1-2b_z^2)}{2\rho |\B|} \\
			0 & -\frac{3 \DP\pll b_x^2}{2\rho} & -\frac{3 \DP\pll b_x b_y}{2\rho} & -\frac{3 \DP\pll b_x b_z}{2\rho} & 0 & -\frac{3 \DP\pll b_x(\bhat\cdot\bu) }{2\rho} & 0 & 0 & 0 \\
			\frac{\DP b_x (\bhat\cdot\bu)}{2} & \frac{\alpha_{xyz}}{4\rho} & \frac{\alpha_{yxz}}{4\rho} & \frac{\alpha_{zxy}}{4\rho} & \frac{3 \DP\pll b_x(\bhat\cdot\bu) }{2\rho} & 0 & -\frac{\per \beta_1}{2 \rho} & \frac{\per \beta_2}{2 \rho} &  \frac{\per \beta_3}{2 \rho}\\
			0 & \frac{\per\omega_1}{2\rho} & \frac{\per\omega_2}{2\rho} & \frac{\per\omega_3}{2\rho} & 0 & \frac{\per \beta_1}{2 \rho} & 0 & \frac{\per u_y}{2\rho} & \frac{\per u_z}{2\rho}\\
			0 & \frac{\DP \per b_y b_x^2}{\rho |\B|} & -\frac{\DP \per b_x (1-2b_y^2)}{2\rho |\B|} & \frac{\DP \per b_x b_y b_z}{\rho |\B|} & 0 & -\frac{\per \beta_2}{2 \rho} & -\frac{\per u_y}{2\rho} & 0 & 0\\
			0 & \frac{\DP \per b_z b_x^2}{\rho |\B|} & \frac{\DP \per b_x b_y b_z}{\rho |\B|} & -\frac{\DP \per b_x (1-2b_z^2)}{2\rho |\B|} & 0 & -\frac{\per \beta_3}{2 \rho} & -\frac{\per u_z}{2\rho} & 0 & 0
		\end{pmatrix}.
	\end{align*}
	where, $\DP = (\pll-\per)$ and $\alpha_{lmn}= 2 \per B_x B_l + \DP b_x\left(2 \rho u_l (b_m u_m + b_n u_n) - b_l (3\pll +2\per + \rho (u_m^2 +u_n^2 - u_l^2))\right);~~\forall~l,m,n \in \{x,y,z\}$\\
	
	$\beta_1 = B_y u_y + B_z u_z + \frac{\DP\left(-b_y u_y - b_z u_z  + 2 b_x((\bhat\cdot\bu)b_x-u_x)\right)}{|\B|}$ and $\beta_l = B_x u_l - \frac{\DP b_x(2 b_l(\bhat\cdot\bu)-u_l)}{|\B|};~~\forall~l\in\{y,z\}$\\
	
	$\omega_1 = B_x - \frac{2 \DP b_x (1-b_x^2)}{|\B|}$ and $\omega_l = B_l - \frac{\DP b_l (1 - 2 b_x^2)}{|\B|};~~\forall~l\in\{y,z\}$
\end{landscape}


\end{document}